\documentclass[11pt]{article}
\usepackage[numbers]{natbib}
\usepackage[margin=1in]{geometry}
 
\usepackage[]{amsmath,amssymb,epsfig}
\usepackage{amsthm}
\usepackage{amsmath}
\usepackage{amssymb}
\usepackage{graphicx}
\usepackage{epstopdf}
\usepackage{comment}
\usepackage{array}
\usepackage{algorithm}
\usepackage{url}

\usepackage{lscape}
\usepackage{algpseudocode}
\usepackage{setspace}
\usepackage{multicol}
\usepackage{multirow}
\usepackage{color}
\usepackage{colortbl}
\usepackage{xcolor}
\usepackage{esvect}

\usepackage{placeins}

\usepackage{enumerate}

\usepackage[%
    font={small,sf},
    labelfont=bf,
    format=hang,    
    format=plain,
    margin=0pt,
    width=0.8\textwidth,
]{caption}
\usepackage[list=true]{subcaption}

\newtheorem{theorem}{Theorem}

\newtheorem{assumption}{Assumption}

\newtheorem{corollary}{Corollary}

\newtheorem{example}{Example}

\newtheorem{lemma}{Lemma}

\newtheorem{proposition}{Proposition}
\newtheorem{remark}{Remark}

\numberwithin{equation}{section}

\DeclareMathOperator{\diag}{diag}

\DeclareMathOperator{\logit}{logit}

\newcommand{\calA}{\ensuremath{\mathcal{A}}}
\newcommand{\calB}{\ensuremath{\mathcal{B}}}

\newcommand{\calG}{\ensuremath{\mathcal{G}}}

\newcommand{\calS}{\ensuremath{\mathcal{S}}}

\newcommand{\calN}{\ensuremath{\mathcal{N}}}
\newcommand{\calT}{\ensuremath{\mathcal{T}}}

\newcommand{\calE}{\ensuremath{\mathcal{E}}}

\newcommand{\norm}[1]{\|{#1}\|}
\newcommand{\abs}[1]{\left|{#1}\right|}

\newcommand{\set}[1]{\left\{{#1}\right\}}

\newcommand{\est}[1]{\widehat{#1}}
\newcommand{\expec}{\ensuremath{\mathbb{E}}}
\newcommand{\indic}[1]{\mathbf{1}_{#1}}
\newcommand{\matR}{\ensuremath{\mathbb{R}}}

\newcommand{\argmin}[1]{\underset{#1}{\operatorname{argmin}}}

\newcommand{\prob}{\ensuremath{\mathbb{P}}}

\newcommand{\rev}[1]{\textcolor{black}{#1}}

\newcommand{\lip}{\ensuremath{M}}

\newcommand{\lammax}{\lambda_{\max}}
\newcommand{\bmax}{b_{\max,\delta}(t)}

\newcommand{\cardN}{|\calN_{\delta}(t)|}
\newcommand{\cardNij}{|\calN_{ij,\delta}(t)|}
\newcommand{\cardNik}{|\calN_{ik,\delta}(t)|}
\newcommand{\cardNjm}{|\calN_{jm,\delta}(t)|}
\newcommand{\Nij}{\calN_{ij,\delta}(t)}
\newcommand{\Nik}{\calN_{ik,\delta}(t)}
\newcommand{\Njm}{\calN_{jm,\delta}(t)}
\newcommand{\Nmax}{N_{\max,\delta}(t)}
\newcommand{\Nmin}{N_{\min,\delta}(t)}
\newcommand{\Nminsq}{N_{\min,\delta}^2(t)}
\newcommand{\Nt}{\calN_{\delta}(t)}
\newcommand{\du}{d_{\delta}(t)}
\newcommand{\dusq}{d_{\delta}^2(t)}
\newcommand{\dumin}{d_{\min,\delta}(t)}
\newcommand{\dumax}{d_{\max,\delta}(t)}
\newcommand{\djl}{d_j^l(t)}
\newcommand{\Eu}{\calE_{\delta}(t)}
\newcommand{\cardEu}{\lvert\calE_{\delta}(t)\rvert}
\newcommand{\Gu}{G_{\delta}(t)}
\newcommand{\Gum}{G_{\delta}^{(m)}(t)}
\newcommand{\pu}{p_{\delta}(t)}
\newcommand{\pmin}{p_{\min}}
\newcommand{\pmax}{p_{\max}}
\newcommand{\psum}{p_{\delta,sum}(t)}
\newcommand{\psumsq}{p_{\delta,sum}^2(t)}

\newcommand{\xid}{\xi_{\delta}(t)}
\newcommand{\Ld}{\mathcal{L}_{\delta}(t)}

\newcommand{\Pbar}{\bar{P}(t)}
\newcommand{\Pbarij}{\bar{P}_{ij}(t)}
\newcommand{\Pbarji}{\bar{P}_{ji}(t)}
\newcommand{\Pbarjm}{\bar{P}_{jm}(t)}
\newcommand{\Pbarmm}{\bar{P}_{mm}(t)}

\newcommand{\Pbarmj}{\bar{P}_{mj}(t)}
\newcommand{\Pest}{\est{P}(t)}
\newcommand{\Pestjm}{\est{P}_{jm}(t)}
\newcommand{\Pestij}{\est{P}_{ij}(t)}

\newcommand{\Pestmm}{\est{P}_{mm}(t)}
\newcommand{\Pestmj}{\est{P}_{mj}(t)}
\newcommand{\Pestjj}{\est{P}_{jj}(t)}
\newcommand{\Peststar}{\est{P}^*(t)}
\newcommand{\Peststarij}{\est{P}^*_{ij}(t)}
\newcommand{\Peststarjm}{\est{P}^*_{jm}(t)}
\newcommand{\Peststarmj}{\est{P}^*_{mj}(t)}

\newcommand{\Peststarmm}{\est{P}^*_{mm}(t)}
\newcommand{\Pmest}{\est{P}^{(m)}(t)}
\newcommand{\Pmestij}{\est{P}^{(m)}_{ij}(t)}
\newcommand{\Pmestii}{\est{P}^{(m)}_{ii}(t)}
\newcommand{\Pmestik}{\est{P}^{(m)}_{ik}(t)}
\newcommand{\Pmg}{\est{P}^{(m),\calG}(t)}
\newcommand{\Pmgij}{\est{P}^{(m),\calG}_{ij}(t)}
\newcommand{\Pmgik}{\est{P}^{(m),\calG}_{ik}(t)}
\newcommand{\Pmgii}{\est{P}^{(m),\calG}_{ii}(t)}
\newcommand{\Pmgjj}{\est{P}^{(m),\calG}_{jj}(t)}
\newcommand{\Pmgmj}{\est{P}^{(m),\calG}_{mj}(t)}

\newcommand{\piest}{\est{\pi}(t)}
\newcommand{\piestm}{\est{\pi}_m(t)}
\newcommand{\piestj}{\est{\pi}_j(t)}
\newcommand{\piesti}{\est{\pi}_i(t)}
\newcommand{\pistar}{\pi^*(t)}
\newcommand{\pistari}{\pi^*_i(t)}
\newcommand{\pistarj}{\pi^*_j(t)}
\newcommand{\pistark}{\pi^*_k(t)}
\newcommand{\pistarm}{\pi^*_m(t)}
\newcommand{\pistarmax}{\pi^*_{\max}(t)}
\newcommand{\pistarmin}{\pi^*_{\min}(t)}
\newcommand{\pimest}{\est{\pi}^{(m)}(t)}
\newcommand{\pimestj}{\est{\pi}^{(m)}_j(t)}
\newcommand{\pimesti}{\est{\pi}^{(m)}_i(t)}
\newcommand{\pimestm}{\est{\pi}^{(m)}_m(t)}

\newcommand{\Dbarone}{\overline{\Delta_1}(t)}
\newcommand{\Dbar}{\overline{\Delta}(t)}
\newcommand{\Deloneij}{\Delta_{1,ij}(t)}
\newcommand{\Deloneik}{\Delta_{1,ik}(t)}
\newcommand{\Deloneii}{\Delta_{1,ii}(t)}
\newcommand{\Dbaroneij}{\overline{\Delta_{1,ij}}(t)}
\newcommand{\Dbarij}{\overline{\Delta}_{ij}(t)}
\newcommand{\Dbarii}{\overline{\Delta}_{ii}(t)}

\newcommand{\Gutil}{\tilde{G}_{\delta}(t)}
\newcommand{\Eutil}{\tilde{\calE}_{\delta}(t)}
\newcommand{\putil}{\tilde{p}_{\delta}(t)}
\newcommand{\dtil}{\tilde{d}_{\delta}(t)}
\newcommand{\dmintil}{\tilde{d}_{\min,\delta}(t)}
\newcommand{\dmaxtil}{\tilde{d}_{\max,\delta}(t)}
\newcommand{\Nmintil}{\tilde{N}_{\min,\delta}(t)}
\newcommand{\Nmaxtil}{\tilde{N}_{\max,\delta}(t)}
\newcommand{\xidtil}{\tilde{\xi}_{\delta}(t)}
\usepackage{authblk}

\title{Dynamic Ranking with the BTL Model: A Nearest Neighbor based Rank Centrality Method}

\author[1]{Eglantine Karlé \thanks{\texttt{Email: eglantine.karle@inria.fr}}}
\author[1]{Hemant Tyagi \thanks{\texttt{Email: hemant.tyagi@inria.fr}}}
\affil[1]{Inria, Univ. Lille, CNRS, UMR 8524 - Laboratoire Paul Painlev\'{e}, F-59000}

\allowdisplaybreaks

\begin{document}
\maketitle

\begin{abstract}
Many applications such as  recommendation systems or sports tournaments involve pairwise comparisons within a collection of $n$ items, the goal being to aggregate the binary outcomes of the comparisons in order to recover the latent strength and/or  global ranking of the items. In recent years, this problem has received significant interest from a theoretical perspective with a number of methods being proposed, along with associated statistical guarantees under the assumption of a suitable generative model. 

While these results typically collect the pairwise comparisons as one comparison graph $G$, however in many applications -- such as the outcomes of soccer matches during a tournament -- the nature of pairwise outcomes can evolve with time. Theoretical results for such a dynamic setting are relatively limited compared to the aforementioned static setting.  We study in this paper an extension of the classic BTL (Bradley-Terry-Luce) model for the static setting to our dynamic setup under the assumption that the probabilities of the pairwise outcomes evolve \emph{smoothly} over the time domain $[0,1]$. Given a sequence of comparison graphs $(G_{t'})_{t' \in \mathcal{T}}$ on a regular grid $\mathcal{T} \subset [0,1]$, we aim at recovering the latent strengths of the items $w_t^* \in \mathbb{R}^n$ at any time $t \in [0,1]$. To this end, we adapt the Rank Centrality method -- a popular spectral approach for ranking in the static case -- by locally averaging the available data on a suitable neighborhood of $t$. When $(G_{t'})_{t' \in \mathcal{T}}$ is a sequence of Erdös-Renyi graphs, we provide non-asymptotic $\ell_2$ and $\ell_{\infty}$ error bounds for estimating $w_t^*$ which in particular establishes the consistency of this method in terms of $n$, and the grid size $\abs{\mathcal{T}}$. We also complement our theoretical analysis with experiments on real and synthetic data. 
\end{abstract}

\section{Introduction} \label{sec:intro}
Ranking problems arise in many domains such as sports tournaments \citep{cattelan2013dynamic,bong2020nonparametric}, recommendation systems \citep{jannach2016recommender} and even in the study of systems in biology \citep{carpenter2006cellprofiler,shah2017simple}. Typically, the data associated with these problems are  binary outcomes of pairwise comparisons. For example in a sports tournament, the data comprises of the result of each game played between two teams or players (team $i$ beat team $j$). The aim then is to recover a ranking of the items from the pairwise outcomes, often by estimating a score for each item, that represents its latent strength/quality. We remark that comparison data are frequently prefered to the attribution of a score by the users due to the sensitivity to bias of the latter. Indeed, the evaluation of the items and the strategy of scoring usually varies considerably between users and thus lead to rather inconsistent scores \citep{jannach2016recommender}. 

The collection of pairs of items which are compared can be naturally represented as an undirected graph $G=([n],\calE)$ where the $n$ vertices are the items to rank, and an edge $\set{i,j} \in \calE$ represents a comparison between items $i$ and $j$. 
Note that in order to recover the underlying ranking, the graph $G$ necessarily needs to be connected. 
Indeed, if this is not the case, then there will exist two connected components of $G$ such that there is no information available to rank items in one component with respect to the other.

\paragraph{Classic BTL model.} A classic statistical model for ranking problems was introduced by Bradley and Terry \citep{bradley1952rank} and completed by Luce \citep{luce2012individual}; this is now referred to as the BTL (Bradley-Terry-Luce) model. Denoting $w_i^* > 0$ to be the strength of item $i$, it posits that the probability that item $j$ beats item $i$, namely $y^*_{ij}$, is given by $\frac{w^*_{j}}{w^*_{i}+w^*_{j}}.$ The data is then available in the form of independent Bernoulli random variables paramaterized by $y^*_{ij}$ for each $\set{i,j} \in \calE$. This model has been studied extensively through many appplications, such as sports tournaments \citep{cattelan2013dynamic}, measurements of pain \citep{matthews1995application} or the estimation of the risk of car-crashes \citep{li2000estimating}. Moreover, numerous methods have been proposed to analyse the model theoretically with the goal of estimating the latent strengths. Two methods which have received particular attention are the maximum likelihood estimator (MLE) \citep{bradley1952rank,negahban2015rank}, and a spectral method called Rank Centrality introduced by Neghaban et al. \citep{negahban2015rank}. Both these methods have been analyzed in detail over the years for general graph topologies (eg., \citep{negahban2015rank}) and also for the special case of Erdös-Renyi graphs $G \sim \calG(n,p)$ (eg., \citep{Chen_2019,chen2020partial}). For instance, let us denote $\pi^* = \frac{w^*}{\norm{w^*}_1}$ to be the normalized vector of true scores, and $\est{\pi}$ its estimate returned by Rank Centrality. Then assuming that each comparison between two items has been performed $L$ times (for $L$ large enough) and $p\gtrsim \frac{\log n}{n}$, it holds with high probability (w.h.p) that \citep{Chen_2019,chen2020partial}
\begin{equation*}
    \frac{\norm{\est{\pi}-\pi^*}_2}{\norm{\pi^*}_2} \lesssim \frac{1}{\sqrt{npL}} \quad \text{and} \quad \frac{\norm{\est{\pi}-\pi^*}_\infty}{\norm{\pi^*}_\infty} \lesssim \sqrt{\frac{\log n}{npL}}.
\end{equation*}
Here $\lesssim, \gtrsim$ hide absolute constants, see notations in Section \ref{sec:prob_setup}. Analogous bounds hold for the MLE as well, and moreover, these bounds are optimal \citep{Chen_2019,chen2020partial}. The $\ell_{\infty}$ bounds are particularly useful as they readily lead to conditions for exact recovery of the true ranks of the items. A weaker requirement is the identification of the set of top-$K$ ranked items for some small $K$, see \citep{Chen_2019}.

\paragraph{Dynamic ranking.} The aformentioned setup considers the `static' case where only one comparison graph is available. However, in many applications, the latent strengths of the items evolve over time. Indeed, a user preference changes with time in  recommendation systems, as does the level of a player or a team in sports tournaments. This dynamic setting has received significantly less attention than its static counterpart and has mostly been studied with a focus on applications, such as sports tournaments \citep{cattelan2013dynamic,jabin2015continuous,motegi2012network}. It has rarely been analysed theoretically in the past however, although some results exist for a state-space generalization of the BTL model \citep{fahrmeir1994dynamic,glickman1998state} and for a Bayesian framework \citep{glickman1998state,lopez2018often,maystre2019pairwise}.

\paragraph{Dynamic BTL.} The focus of this paper is to consider a dynamic version of the classical BTL model for ranking. 
In this \emph{dynamic} scenario, 
we are given pairwise comparisons on a discrete set of time instants $\calT$, hence leading to a sequence of comparison graphs $(G_{t'})_{t' \in\calT}$. Moreover, denoting  $w_{t,i}^* > 0$ to be the strength of item $i$ at time $t$, the probability that item $j$ beats item $i$ at time $t' \in \calT$ is now given by $\frac{w^*_{t',j}}{w^*_{t',i}+w^*_{t',j}}$. This model is described in detail in Section \ref{sec:dynam_btl}. The goal is to recover the underlying ranking at any given time instant $t$, where $t$ does not necessarily belong to $\calT$.  Furthermore, we make the following remarks.
\begin{itemize}
    \item Suppose $t$ lies in $\calT$, then we could consider only using the pairwise comparison data corresponding to $G_{t}$, along with a suitable algorithm for the static case. However, $G_{t}$ is not necessarily going to be connected; indeed, all the comparison graphs can now be disconnected and be very sparse. Hence this approach will fail in such scenarios. This suggests that it is in some sense crucial that we utilize the pairwise comparison information across different time instants in $\calT$ in order to get meaningful rank recovery at time $t$. This approach can of course be relevant even if $G_{t'}$ was connected for all $t' \in \calT$.
    
    \item Very recently, a couple of theoretical results were obtained for the dynamic ranking problem. Li and Wakin \citep{li2021recovery} include in their model the observed difference of scores (for example the score of a football game) which clearly contains richer information in comparison to simple binary outcomes. Bong et al. \citep{bong2020nonparametric} considered an adaptation of the BTL model to the dynamic setting similar to what has been described above. They first combine pairwise comparison information by applying kernel smoothing to the entire dataset, and then use the MLE for the smoothed data to recover the ranks. We discuss these results in detail in Section \ref{sec:discussion}.  
\end{itemize}

\paragraph{Our contributions.} We provide in this paper the following contributions. 
\begin{itemize}
    \item We consider a model adapted to the dynamic ranking setting which involves using the BTL model at each time point, but assumes that the pairwise outcome probabilities are Lipschitz functions of time. In this setting, we propose an adaptation of the Rank Centrality method, namely Dynamic Rank Centrality (Algorithm \ref{alg:spectral_algo}) to recover the ranks at any given time $t$ by averaging the data belonging to a suitable neighborhood of $t$. 
    
    \item We provide a detailed theoretical analysis of the previous algorithm by deriving $\ell_2$ and $\ell_\infty$ error bounds for estimating the latent strength vector $w_t^* = (w^*_{t,1}, \dots,w^*_{t,n})^\top$ at any given time instant $t \in [0,1]$. In particular, assuming $\calT \subset [0,1]$ is a uniform grid ($\abs{\calT} = T+1$) and the comparison graphs $(G_{t'})_{t'\in\calT}$ are Erdös-Renyi, we show (see Corollaries \ref{cor:l2},\ref{cor:linf}) for $T$ large enough that the error rates are of the order $O(T^{-1/3})$ w.h.p. This result holds in very sparse regimes where each $G_{t'}$ could be potentially disconnected. Moreover, this rate matches the classical pointwise estimation error rate for Lipschitz functions.
    
    \item Finally, we perform extensive experiments on synthetic data which validates our theoretical findings. We also evaluate our method on a real data set and show that it performs well.
\end{itemize}

\paragraph{Outline of the paper.} After presenting the dynamic BTL model and our algorithm in Section \ref{sec:prob_setup}, we summarize our main theoretical results for the $\ell_2$ and $\ell_\infty$ estimation errors in Section \ref{sec:results}. We then describe in detail the $\ell_2$ analysis in Section \ref{sec:analysis_l2}, and the $\ell_\infty$ analysis in Section \ref{sec:analysis_linf}. Finally, we conduct experiments on both synthetic and real data in Section \ref{sec:experiments}, and discuss other closely related work in Section \ref{sec:discussion}.


\section{Problem setup and algorithm}\label{sec:prob_setup}

\paragraph{Notation}
For any probability vector $\pi \in \matR^n$ with strictly positive entries, we define the vector norm $\norm{x}_\pi = \sqrt{\sum_{i=1}^n \pi_i x_i^2}$. For a matrix $A$, the corresponding induced matrix norm is then defined as $\norm{A}_\pi = \sup_{\norm{x}_\pi=1} \norm{x^\top A}_\pi$. Note that some simple inequalities follow from these definitions.
\begin{equation}\label{eq:bound_pinorm}
    \sqrt{\pi_{\min}}\norm{x}_2 \leq \norm{x}_\pi \leq \sqrt{\pi_{\max}}\norm{x}_2 \quad \text{and} \quad \sqrt{\frac{\pi_{\min}}{\pi_{\max}}}\norm{A}_2 \leq \norm{A}_\pi \leq \sqrt{\frac{\pi_{\max}}{\pi_{\min}}}\norm{A}_2
\end{equation}
where $\norm{A}_2$ denotes the spectral norm, i.e., largest singular value, of $A$. The Frobenius norm of $A$ is denoted by $\norm{A}_F$. 
If $A$ is a $n \times n$ matrix with real eigenvalues, we order them as $\lambda_n(A) \leq \cdots \leq \lambda_1(A)$. Positive (absolute) constants are denoted by $c, C, \tilde{C}$ etc. with suitable choice of indices. For $a,b \geq 0$, we write $a \lesssim b$ if there exists a constant $C > 0$ such that $a \leq C b$. Moreover, we write $a\asymp b$ if $a \lesssim b$ and $b \lesssim a$.

%
\subsection{BTL model in a dynamic setting}\label{sec:dynam_btl}

Let us formally introduce our model for dynamic pairwise comparisons, inspired by the Bradley-Terry-Luce (BTL) model. We consider a set of items $[n] = \set{1,2,\dots,n}$, with a certain quality at each time $t \in [0,1]$, represented by the weight vector $w_t^* = (w^*_{t,1}, \dots,w^*_{t,n})^\top \in \matR^n$ with $w^*_{t,i} > 0$ for each $i \in [n]$. 

Our data consists of pairwise comparisons on this set of items at times $t' \in \calT = \set{\frac{i}{T} | \, i=0,\dots,T}$. The outcomes at each \rev{$t' \in \calT$} are gathered into an undirected comparison graph $G_{t'} = ([n],\calE_{t'})$ where $\calE_{t'}$ is the set of edges. 
While the set of items $[n]$ is supposed to be the same throughout, the compared items, i.e. the set of edges $\calE_{t'}$ can change with time.

To model such data, we use the BTL model at each time \rev{$t' \in \calT$}. This model posits that the probability that an item $j$ wins over an item $i$ is proportional to its strength. At each \rev{$t' \in \calT$}, for each pair of compared items $\set{i,j} \in \calE_{t'}$, we perform $L$ independent comparisons. Their outcomes are independent Bernoulli variables, defined for $l \in \set{1,\dots,L}$ by $y_{ij}^{(l)}(t')$ where
$$\prob(y_{ij}^{(l)}(t') = 1) = \frac{w^*_{t',j}}{w^*_{t',i}+w^*_{t',j}}.$$
The proportion of times $j$ won over $i$ at time $t'$ is given by $y_{ij}(t') = \frac{1}{L}\displaystyle\sum_{l=1}^L y_{ij}^{(l)}(t')$ and the corresponding true proportion is denoted by (\rev{for any $t \in [0,1]$})
\begin{equation*}
y^*_{ij}(t) := \expec [y_{ij}(t)] = \frac{w^*_{t,j}}{w^*_{t,i}+w^*_{t,j}} \quad \forall i\neq j.
\end{equation*}

\paragraph{Smooth evolution of pairwise outcomes.} Our goal is to recover $w^*_t$ at any time $t \in [0,1]$. Suppose for the moment that $t$ is on the grid, then if $G_t$ is connected, $w^*_t$ is identifiable up to a positive scaling. However in our dynamic setting $G_t$ can be very sparse, and is not necessarily connected. Therefore for meaningful recovery of $w_t^*$, we need to make additional assumptions on the evolution of the pairwise outcomes over time.  To this end, we make the following smoothness assumption.
\begin{assumption}[Lipschitz smoothness] \label{assump:smooth1}
There exists $M \geq 0$ such that
\begin{equation}
 |y^*_{ij}(t)-y^*_{ij}(t')| \leq \lip |t-t'| \quad \forall t,t' \in [0,1], \quad i\neq j \in [n].
\end{equation}
\end{assumption}
This assumption suggests that the pairwise outcomes at nearby time instants are similar, hence it is plausible that $w^*_t$ (at any $t \in [0,1]$) could be estimated by utilizing the data lying in a neighborhood of $t$. To formalize this intuition, let us define a neighborhood at any time $t$ by
\begin{equation} \label{eq:delta_neighbor}
\Nt := \set{t' \in \calT \ | \ \, |t-t'| \leq \frac{\delta}{T}}.
\end{equation}
where $\delta \in [0,T]$. Note that if $\delta < \frac{1}{2}$, then there exists some values of $t$ for which $\Nt$ is empty; hence we will consider $\delta \in [1/2, T]$. It is easy to verify that $\delta\leq \cardN\leq 4\delta$ for all $t\in[0,1]$.

\paragraph{Neighborhood graph.} For any time $t \in [0,1]$, the data contained in the neighborhood $\Nt$ can be gathered into a union graph, defined as $\Gu = ([n],\Eu)$ where $\Eu = \displaystyle\cup_{t' \in \Nt} \calE_{t'}$. The maximum and minimum degree of a vertex in $\Gu$ will be denoted by $\dumax$ and $\dumin$ respectively.
\rev{For each $i \neq j$}, it will also be useful to denote the the time instants in $\Nt$ where $i$ and $j$ are compared as 
\begin{equation*}
\calN_{ij,\delta}(t) = \set{t' \in \Nt | \, \set{i,j}\in\calE_{t'}},
\end{equation*}
along with the quantities
\begin{equation*}
    \Nmax = \max_{\set{i,j}\in\Eu} \cardNij \quad \text{and} \quad \Nmin = \min_{\set{i,j}\in\Eu} \cardNij.
\end{equation*}
\rev{Note that alternatively,
\[ \Eu = \set{\set{i,j} \, : \, i\neq j, \, \cardNij \geq 1}.\]}

\paragraph{General recovery idea.} 
Given the above setup, a general idea for recovering $w^*_t$ is to first form the union graph $\Gu$, and to then compute the statistics
\begin{equation} \label{eq:stat_average_neigh}
    \bar{y}_{ij}(t) := \frac{1}{\cardNij}\displaystyle\sum_{t' \in \Nij} y_{ij}(t') \quad \mbox{for} \quad \set{i,j} \in \Eu.
\end{equation}
\rev{Suppose for convenience that $L \rightarrow \infty$}, we have for each $\set{i,j} \in \Eu$ that
\begin{equation*}
    \bar{y}_{ij}(t) \xrightarrow{L \rightarrow \infty} \bar{y}_{ij}^*(t) = \frac{1}{\cardNij}\displaystyle\sum_{t' \in \Nij} y^*_{ij}(t').
\end{equation*}
Due to Assumption \ref{assump:smooth1} we know that 
\begin{equation*}
  \abs{\bar{y}_{ij}^*(t) - y_{ij}^*(t)}  \leq M\delta/T, \quad \mbox { for }  \set{i,j} \in \Eu, 
\end{equation*}
hence if $\delta = o(T)$, then for each $\set{i,j} \in \Eu$, we have that  $\bar{y}_{ij}^*(t)$ converges to $y_{ij}^*(t)$ as $T \rightarrow \infty$. Moreover, if the corresponding sequence of graphs $\Gu$ is connected, the identifiability of $w^*_t$ (up to a positive scaling) is ensured. The connectivity requirement on $\Gu$ is of course much weaker than requiring the individual graph(s) to be connected, which is also a key difference between the static and dynamic settings.
Also note that the computed statistcs $\bar{y}_{ij}(t)$ are nearest neighbor estimators as they average the data over a suitable neighborhood of $t$. More generally, one could also compute $\bar{y}_{ij}(t)$ via kernel smoothing as considered in \citep{bong2020nonparametric}. \rev{As will be shown later, even if $L=1$, consistent recovery of $w^*_t$ is possible provided that $\delta = o(T)$, $\delta \overset{T \to \infty}{\longrightarrow} \infty$ and that every $\set{i,j}\in\Eu$ satisfies $\cardNij\overset{T \to \infty}{\longrightarrow} \infty$.}
%

With the above discussion in mind, a general scheme for recovering $w^*_t$ for any given $t \in [0,1]$ would be to first form $\Gu$ for a suitable choice of the parameter $\delta$, then compute the statistics $\bar{y}_{ij}(t)$ as in \eqref{eq:stat_average_neigh}, and finally apply any existing method for the static case using the comparison graph $\Gu$ and the data $(\bar{y}_{ij}(t))_{\set{i,j} \in \Eu}$. We will focus on the Rank Centrality algorithm of Negahban et al. \citep{negahban2015rank} -- a popular spectral algorithm known to achieve state of the art performance -- and adapt it to our dynamic setting.

\subsection{Spectral dynamic ranking}\label{sec:spec_setup}
The Rank Centrality \citep{negahban2015rank} method is based on the connection between pairwise comparisons and a random walk on a directed graph. In the static case, a faithful estimation of the weight vector is given by the stationary distribution of the Markov chain induced by a suitably constructed transition matrix. In the dynamic setting, this method can be adapted for estimating $w^*_t$ by now constructing the transition matrix using $\Gu$ and the data $(\bar{y}_{ij}(t))_{\set{i,j} \in \Eu}$, as in \eqref{eq:stat_average_neigh}. 

More precisely, we define a transition matrix $\Pest$ on this graph using $(\bar{y}_{ij}(t))_{\set{i,j} \in \Eu}$, with  
%
\begin{equation}\label{eq:spec_mat_pest}
\Pestij = \left\{\begin{array}{cc}
	\frac{\bar{y}_{ij}(t)}{\du} = \frac{1}{\du} \left(\frac{1}{\cardNij}\displaystyle\sum_{t' \in \Nij} y_{ij}(t') \right) & \mbox{if} \quad \set{i,j} \in \Eu \\
	1-\frac{1}{\du} \left(\displaystyle\sum_{k \neq i} \frac{1}{\cardNik}\displaystyle\sum_{t' \in \Nik} y_{ik}(t') \right) & \mbox{if} \quad i=j\\
	0 & \mbox{otherwise}
\end{array}
\right.
\end{equation}
%
%
where $\du \geq \dumax$ is a suitably chosen normalization term. Then, one can easily verify that at each time $t$, $\Pest$ is a transition matrix ($\est{P}(t)$ is stochastic) corresponding to a Markov chain on a finite state space. Thus there always exists at least one stationary distribution associated to $\est{P}(t)$. Moreover, stochastic matrices admit 1 as leading eigenvalue and so a candidate as the stationary distribution is its leading left eigenvector, i.e.
\begin{equation*}
    \piest^\top = \piest^\top\Pest.
\end{equation*}

Besides, the vector of true weights $w^*_t$ we want to recover can be seen as the stationary distribution of a transition matrix on the union graph. Specifically, denoting $\pistar = \frac{w_t^*}{\sum_{i=1}^n w_{t,i}^*}$, one can easily show that $\pistar$ is the stationary distribution of the transition matrix
\begin{equation}\label{eq:spec_mat_pbar}
\Pbarij = \left\{\begin{array}{cc}
	\frac{1}{\du} \frac{w^*_{t,j}}{w^*_{t,i}+w^*_{t,j}} & \mbox{if} \quad \set{i,j} \in \Eu \\
	1-\frac{1}{\du}\displaystyle\sum_{k \neq i} \frac{w^*_{t,k}}{w^*_{t,i}+w^*_{t,k}} & \mbox{if} \quad i=j\\
	0 & \mbox{otherwise}
\end{array}
\right.
\end{equation}
since $\Pbar$ and $\pistar$ verify the detailed balance equation of reversibility \citep{levin2017markov}
\begin{equation*}
    \Pbarij\pistari = \Pbarji\pistarj \quad \forall i,j \in [n].
\end{equation*}
%
One can reasonably expect $\piest$ to be close to $\pistar$ as they are stationary distributions of $\Pest$ and $\Pbar$ respectively, the latter of which are expected to be close. Indeed, one has the following bias-variance trade-off
\begin{equation*}
    \Pest - \Pbar = \underset{\text{variance}}{\underbrace{\Pest - \expec[\Pest]}} + \underset{\text{bias}}{\underbrace{\expec[\Pest] - \Pbar}}.
\end{equation*} 
where the variance term is typically expected to decrease with $\delta$ (due to averaging over $\Nt$) while the bias term will scale as $O(\delta/T)$ (due to the smoothness assumption \ref{assump:smooth1}). Hence for a suitably chosen $\delta = o(T)$ we will then have (for $n, T$ large enough) $\Pest \approx \Pbar$, which implies $\piest \approx \pistar$.
%
%
%
\begin{remark}
For meaningful recovery the vector $\pistar$ clearly has to be unique. This is the case if the associated Markov chain is irreducible which in turn is ensured by the connectivity of the underlying graph (here, the union graph $\Gu$) and the strict positivity of the weights on its  edges \citep{levin2017markov}. The condition on the weights is guaranteed in our setup since $\Pbarij > 0$ for each $\set{i,j} \in \Eu$ (indeed, $w^*_{t,i} > 0$ for each $t \in [0,1]$ and $i \in [n]$).
\end{remark}
Based on the above discussion we can outline the steps of our method for ranking in the dynamic setting in the form of Algorithm \ref{alg:spectral_algo}.
\begin{algorithm}
    \caption{Spectral algorithm for dynamic ranking (Dynamic Rank Centrality)}\label{alg:spectral_algo}
    \begin{algorithmic}[1]
    \State {\bf Input:} Grid $\calT \subset [0,1]$, and a given time $t \in [0,1]$. For each $t' \in \calT$: comparison graph $G_{t'}$, results of comparisons as statistics $(y_{ij}(t'))_{\set{i,j} \in \calE_{t'}}$.  
    
    \State Form the neighbourhood graph $\Gu = ([n],\Eu)$ where $\Eu = \displaystyle\cup_{t' \in \Nt} \calE_{t'}$,  and $\Nt$ is as in \eqref{eq:delta_neighbor}.
    
    \State Compute the transition matrix $\Pest$ as in \eqref{eq:spec_mat_pest} with $\du \geq \dumax$. 
    
    \State Compute the leading left eigenvector $\piest$ of $\Pest$. 
    
    \State {\bf Output:} $\est{\pi}(t) \in \matR^n$.
    \end{algorithmic}
\end{algorithm}
Our goal now is to establish conditions under which $\piest$ is close to $\pistar$ under the $\ell_2$ and $\ell_\infty$ norms. These results are summarized in the next section. In particular, we will strive to establish consistency results (i.e., the error approaching zero) when the grid size $T \rightarrow \infty$.

\paragraph{\rev{Additional definitions.}} Before proceeding, we need to define some additional quantities, \,\rev{ some being} related to the union graph $\Gu$,  which will appear in the following sections. 
Let $\Ld = D_{\delta}^{-1}(t)A_{\delta}(t)$ denote the random walk Laplacian of $\Gu$, where $D_{\delta}(t)$ is the diagonal matrix of vertex degrees, and $A_{\delta}(t)$ is its adjacency matrix. We denote $\xid = 1-\lammax(\Ld)$ where 
$$\lammax(\Ld) := \max \set{\lambda_2(\Ld),-\lambda_n(\Ld)},$$
is the second largest eigenvalue (in absolute value) of $\Ld$. Note that $\Ld$ has real eigenvalues since it is similar to the symmetric Laplacian $D_{\delta}^{-1/2}(t) A_{\delta}(t) D_{\delta}^{-1/2}(t)$. 
\rev{Let us also denote
\begin{equation*}
b(t) := \max_{i,j \in [n]} \frac{w^*_{t,i}}{w^*_{t,j}} \mbox{ for all } t \in [0,1]
\end{equation*} 
where we will require that $b(t)$ is finite for each $t \in [0,1]$.}

%
%
\section{Main results}\label{sec:results}
In Section \ref{sec:results_l2}, we present bounds on the $\ell_2$ error $\norm{\piest - \pistar}_2$, while Section \ref{subsec:results_linfty} contains our bounds on the $\ell_{\infty}$ error $\norm{\piest - \pistar}_{\infty}$. \rev{A summary of the notation used in the paper is outlined in tabular form in Appendix \ref{sec:summary_notation}.}
\subsection{$\ell_2$ error bound} \label{sec:results_l2}
For a given sequence of graphs $(G(t'))_{t' \in \calT}$, the following theorem provides an explicit $\ell_2$ error bound (holding w.h.p) which in particular highlights the dependence on parameters related to the union graph $\Gu$, the grid size $T$ and the neighborhood size $\delta$. \rev{The proofs of results in this section are outlined in Section \ref{sec:analysis_l2}.}
\begin{theorem}\label{thm:l2}
For any given $t \in [0,1]$, suppose that $\delta \in [\frac{1}{2},T]$ is such that $n \geq c_1 \log n$ and $\xid > 0$ for some constant $c_1 > 0$. Choosing $\du \geq \dumax$, there exist constants $\tilde{C_1} \geq 15, \tilde{C_2} \geq 1$ such that if
\begin{equation}\label{eq:cond_thm1}
\tilde{C_1}\sqrt{\frac{\Nmax\dumax\log n}{L\dusq\Nminsq}}+ 4\frac{M\delta\cardEu}{T\dumax} \leq \frac{\xid\dumin}{8\du b^{7/2}(t)}, 
\end{equation}
then it holds with probability at least $1 - O(n^{-10})$ that
\begin{equation}\label{eq:thm1}
    \frac{\norm{\piest - \pistar}_2}{\norm{\pistar}_2} \leq 32\frac{M\delta\cardEu b^{7/2}(t)\du}{T\xid\dumin\dumax}+ 8\tilde{C_2}\frac{b^{9/2}(t)}{\xid\dumin}\sqrt{\frac{\Nmax\dumax}{L\Nminsq}}.
\end{equation}
\end{theorem}
Let us make the following observations.
\begin{enumerate}
\item The first term in the RHS of \eqref{eq:thm1} corresponds to the bias and arises from the regularity assumption \ref{assump:smooth1}, while the second term therein is the variance term. Moreover, note that the error depends on $\delta$ -- either explicitly, or through certain quantities such as $\dumax, \Nmin$ etc. In order to obtain a more explicit dependence in terms of $\delta$, we will need to make specific  assumptions on the graphs $G_{t'}$, $t' \in \calT$. Below, we will consider the setting where the graphs are Erd\"os-Renyi graphs and derive explicit conditions on $\delta$ that lead to consistency with respect to $T$.

\item \rev{The condition \eqref{eq:cond_thm1} arises from the eigenvector perturbation result in \cite[Theorem 8]{Chen_2019} which requires the noise term (i.e., $\Pest - \Pbar$) to be small compared to the spectral gap of $\Pbar$ (i.e., $1-\lammax(\Pbar)$), see \eqref{eq:condition_thm8}.}

\item In the static case, we have $t=t'$ for some $t' \in \calT$ and only the graph $G_{t'}$ is observed. Then $M = 0$ and $\delta = 1/2$, so $\Nmin, \Nmax \equiv 1$. Denoting $d_{\min}(t), d_{\max}(t), \xi(t)$ to be the corresponding quantities with the $\delta$ suffix suppressed,  condition \eqref{eq:cond_thm1} is satisfied for $L$ large enough. Moreover, the error bound is then $O(\frac{b^{9/2}(t)}{\xi(t) d_{\min}(t)}\sqrt{\frac{d_{\max}(t)}{L}})$ which matches the $\ell_2$ bound of Negahban et al. \citep[Theorem 1]{negahban2015rank} with the $\sqrt{\log n}$ factor therein removed, but with an extra $b^2(t)$ factor. 
%
\end{enumerate}
\begin{remark} \label{rem:diff_Lij}
    \rev{The term $\Nmax$ is admittedly counterintuitive and is an artifact of the proof technique. In particular, this occurs due to certain concentration inequalities used within the proof (e.g., Lemma \ref{lem:bound_delta}). To our knowledge, similar issues would arise in the static case for analyzing the Rank Centrality method if each comparison $\set{i,j}$ was made $L_{ij}$ times. This would occur, for instance, in the proof of \cite[Lemma 3]{negahban2015rank} due to Hoeffding's inequality, in which case both the minimum and maximum of the $L_{ij}$'s would appear in the bounds.} 
\end{remark}
Now we consider the important case where the comparison graphs are Erdös-Renyi graphs, i.e., $G_{t'} = \calG(n,p(t'))$ for each $t' \in \calT$. It is easily seen that the union graph $\Gu$ is also then Erdös-Renyi denoted by $\calG(n,\pu)$ where $\pu$ is given by
\begin{equation} \label{eq:pdel_def}
    \pu = 1-\prod_{t' \in \Nt} (1-p(t')).
\end{equation}
%
%
In this setting, the bound in Theorem \ref{thm:l2} can be simplfied using concentration results for parameters related to $\Gu$ (see Lemma \ref{lem:prop_gnp}). Specifically, we have that if $\pu \gtrsim \log n/n$, then w.h.p
\begin{equation*}
\frac{n\pu}{2} \leq \dumax, \dumin \leq \frac{3n\pu}{2}; \quad \xid \geq \frac{1}{2}; \quad \abs{\Eu} \leq 2 n^2 \pu.    
\end{equation*}
In particular, we will choose the normalization factor $\du = 3n\pu$ which is a valid choice (w.h.p). Lemma \ref{lem:prop_gnp} also states that if $\psum := \sum_{t'\in\Nt}p(t') \gtrsim \log n$, then  $\Nmax, \Nmin \asymp \psum$ w.h.p. These considerations lead to the following simplification of Theorem \ref{thm:l2}.
\begin{theorem}\label{thm:l2_gnp}
Suppose that $G_{t'} \sim \calG(n,p(t'))$ for all $t' \in \calT$ so that $\Gu \sim \calG(n,\pu)$ (for any given $t \in [0,1]$) with $\pu$ as in \eqref{eq:pdel_def}, and denote $\psum := \sum_{t'\in\Nt}p(t')$. Choosing $\du = 3 n\pu$, let $\delta \in [\frac{1}{2},T]$ be such that $n \geq c_1 \log n$, $n\pu \geq c_0\log n$, and $\psum \geq c_2 \log n$ with constant $c_1 > 0$ as in Theorem \ref{thm:l2}, and constants $c_0, c_2 \geq 1$. Then for constants $\tilde{C_1}, \tilde{C_2}$ as in Theorem \ref{thm:l2}, if
\begin{equation}\label{eq:cond_thm_gnp}
    2\tilde{C_1}\sqrt{\frac{\log n}{Ln\pu\psum}}+16\frac{M\delta n}{T}\leq \frac{1}{96b^{7/2}(t)} 
\end{equation}
holds, we have with probability at least $1 - O(n^{-10})$ that
\begin{align*}
    \frac{\norm{\piest - \pistar}_2}{\norm{\pistar}_2} & \leq 1536\frac{M\delta nb^{7/2}(t)}{T} + 64\tilde{C_2}b^{9/2}(t)\sqrt{\frac{3}{Ln\pu\psum}}.
\end{align*}
\end{theorem}
%
The following remarks are in order.
\begin{enumerate}
\item As can be seen, the bias term is $O(\frac{n\delta}{T})$ while the variance term scales as $O(\frac{1}{\sqrt{L n \pu \psum}})$. Hence if $\psum$ grows with $\delta$, then the variance error will reduce as  $\delta$ increases. \rev{Furthermore, if $\delta = o(T)$ and $\delta$ increases with $T$ then it would imply that the LHS of \eqref{eq:cond_thm_gnp} decreases with $T$, and hence \eqref{eq:cond_thm_gnp} will be satisfied for $T$ sufficiently large.}

\item In the static case we observe a single comparison graph $G_{t'}$ (for $t' \in \calT$) with $t = t'$. Thus $\psum \equiv 1$ and the condition $\psum \gtrsim \log n$ is not needed, while $\pu = p(t)$. Hence, if $p(t) \gtrsim \frac{\log n}{n}$ and $L$ is suitably large, the $\ell_2$ error is bounded by $O(\frac{1}{\sqrt{Lnp(t)}})$, which corresponds to the bound obtained by Chen et al. \citep[Theorem 9]{Chen_2019}. So our result is coherent with existing results for the static case for Erd\"os-Renyi graphs.

\item The choice $\du = 3n\pu$ ensures that $\dumax \leq \du$ w.h.p. In fact, we could have chosen $\du$ to be a constant ($\geq 1$) multiple of $\dumax$ as well. However for the $\ell_{\infty}$ analysis later on, it will be crucial to choose $\du$ as a constant times $n\pu$ for technical reasons arising in the analysis. Similar considerations for the choice of the normalization factor exist in the static setting as well (see \citep{Chen_2019, chen2020partial}). Note that this choice of $\du$ requires us to know $\pu$, but in case we don't know $\pu$ in practice, we can instead use its empirical estimate which can be easily computed.
%
%
\end{enumerate}

We now derive an appropriate choice for $\delta$ that leads to an $\ell_2$ error rate of $O(T^{-1/3})$. To this end, we first need to explicitly show the dependence on $\delta$ for $\pu, \psum$. 
Let us assume for simplicity that 
\begin{equation} \label{eq:pmin_def}
\pmin := \min_{t' \in \calT} p(t') > 0. 
\end{equation}
Since $\delta \leq \cardN \leq 4\delta$, we have for all $t \in [0,1]$ that $\psum \geq \delta\pmin$. Besides, as shown in Proposition \ref{prop:pu}, 
$\pu \gtrsim \min\set{1,\delta\pmin}$. Hence if $\delta\pmin \gtrsim \log n$ then it implies 
$$\psum \gtrsim \log n \quad \text{and} \quad \pu \gtrsim 1  \ (\geq \log n/n)$$ 
meaning that the conditions on $\pu$ and $\psum$ in Theorem \ref{thm:l2_gnp} are satisfied. 
\begin{remark}\label{rem:cond_psum}
\rev{The condition $\psum \gtrsim \log n$ is needed to ensure that $\Nmin, \Nmax$ concentrate around $\psum$, as shown in Lemma \ref{lem:prop_gnp}. This is in fact a strong condition as requiring $\delta\pmin \gtrsim \log n$ imposes that the union graph $\Gu$ is complete (see proof of Lemma \ref{lem:prop_gnp}). We will show in Section \ref{sec:gutil} how to weaken this assumption.}
\end{remark}
\begin{corollary}\label{cor:l2}
Under the same notations as in Theorem \ref{thm:l2_gnp}, for all $t \in [0,1]$ suppose that $n \gtrsim \log n$ and $\pmin$ is as in \eqref{eq:pmin_def}. Choosing $\delta = \min\set{\frac{(b(t))^{\frac{2}{3}}}{(2M)^{\frac{2}{3}}n(L\pmin)^{\frac{1}{3}}} T^{2/3}, T}$ and $\du = 3n\pu$, if $T$ is such that $\delta \gtrsim \frac{\log n}{\pmin}$ and
\begin{align*}
\sqrt{\frac{\log n}{Ln \pmin}} \max \set{\frac{(2M)^{1/3} n^{1/2} (L \pmin)^{1/6}}{b^{1/3}(t) T^{1/3}}, \frac{1}{\sqrt{T}}} 
+ Mn \min \set{\frac{b^{2/3}(t)}{(2M)^{2/3} n (L \pmin)^{1/3} T^{1/3}}, 1} \lesssim \frac{1}{b^{7/2}(t)}, 
\end{align*}
then with probability at least $1-O(n^{-10})$, 
\begin{align*}
\frac{\norm{\piest-\pistar}_2}{\norm{\pistar}_2} &\lesssim Mn b^{7/2}(t) \min \set{\frac{b^{2/3}(t)}{(2M)^{2/3} n (L\pmin)^{1/3} T^{1/3}}, 1} \\ 
&+ \frac{b^{9/2}(t)}{\sqrt{Ln\pmin}} \max\set{\frac{(2M)^{1/3} n^{1/2} (L\pmin)^{1/6}}{b^{1/3}(t) T^{1/3}}, \frac{1}{\sqrt{T}}}.
\end{align*}
\end{corollary}
The following observations are useful to note.
\begin{enumerate}
    \item When $M > 0$, Corollary \ref{cor:l2} states that for $\delta = \Theta(T^{2/3})$, if $n, T$ are large enough (\rev{thus ensuring that all the stated conditions are satisfied)}), then w.h.p $\norm{\piest-\pistar}_2 = O(T^{-1/3})$. This matches the rate for the pointwise risk for estimating univariate Lipschitz functions (see for e.g. \citep[Theorem 1.3.1]{nemirovski2000topics}).  
    
    \item If $M = 0$ then $\delta = T$ which makes sense since $y^{*}_{ij}(t)$ is a constant function for each $i \neq j$. Indeed, the problem is then the same as the setting where the comparison  graph is $\cup_{t' \in \calT} G_{t'}$, and we observe (a potentially different number of) i.i.d pairwise outcomes for each given edge in this graph. In this case, the corollary states that provided $T$ is large enough, the $\ell_2$ error is $\lesssim \frac{b^{9/2}(t)}{\sqrt{L n \pmin T}}$. This is logically faster than the $T^{-1/3}$ nonparametric rate, and is analogous to the optimal $\ell_2$ bound for Erd\"os Renyi graphs in the static setting (see \citep[Theorem 9]{Chen_2019}). 
    
\end{enumerate}


\subsection{$\ell_\infty$ error bound}
\label{subsec:results_linfty}
We now discuss our results for bounding the $\ell_{\infty}$ error $\norm{\piest - \pistar}_\infty$ at any given time $t$. Such bounds are particularly desirable in the context of ranking as they lead to guarantees for recovering the ranks of the items. We will assume that all the comparison graphs (at each $t' \in \calT$) are Erdös-Renyi graphs. The following theorem is the $\ell_{\infty}$ counterpart of Theorem \ref{thm:l2_gnp}, \rev{the proofs of results in this section are outlined in Section \ref{sec:analysis_linf}}. 
%
%
\begin{theorem}\label{thm:linf}
Under the notation and assumptions of Theorem \ref{thm:l2_gnp}, there exists a constant $\tilde{C_3} \geq  1$ such that if additionally  
\begin{equation}\label{eq:cond_thm3}
    96b^{\frac{5}{2}}(t) \left(\frac{4M\delta}{T}+\tilde{C_3}\sqrt{\frac{\log n}{n\pu}} \right) \leq \frac{1}{2}, 
\end{equation}
then there exist constants $\tilde{C_4},\tilde{C_5},\tilde{C_6} \geq 1$ such that with probability at least $1 - O(n^{-9})$, 
\begin{equation*}
    \frac{\norm{\piest - \pistar}_\infty}{\norm{\pistar}_\infty}\leq \frac{12\bmax}{1-\tilde{C_4}\bmax\sqrt{\frac{\log n}{n\pu}}} \left(\tilde{C_5} \gamma_{n,\delta}(t) \sqrt{\frac{\log n}{Ln\pu\psum}} + \tilde{C_6}\frac{Mn\delta b^{\frac{7}{2}}(t)}{T} \right),
\end{equation*}
where $\gamma_{n,\delta}(t) := (1+\frac{b^{\frac{5}{2}}(t)}{\sqrt{\log n}} \max\{b^2(t), \frac{\log n}{\sqrt{n\pu}}\})$ and \rev{$\bmax := \max_{t'\in\Nt} b(t')$.}
\end{theorem}
As before for Theorem \ref{thm:l2_gnp}, let us interpret Theorem \ref{thm:linf} for the static setting where $t = t'$ for some $t' \in \calT$, and only $G_{t'}$ is observed. Then Theorem \ref{thm:linf} states that if $np(t) \gtrsim b^{5}(t) \log n$, and $n,L$ are large enough, then w.h.p, the $\ell_{\infty}$ error is  
$$\frac{\norm{\piest - \pistar}_\infty}{\norm{\pistar}_\infty} \lesssim b(t) \left(1+\frac{b^{\frac{5}{2}}(t)}{\sqrt{\log n}} \max\{b^2(t), \frac{\log n}{\sqrt{n\rev{p(t)}}}\} \right) \sqrt{\frac{\log n}{L n p(t)}}.$$
Hence if $b(t) = O(1)$ then the bound is $O(\sqrt{\frac{\log n}{L n p(t)}})$ which matches the corresponding bound of Chen et al. \citep[Theorem 3]{Chen_2019}.

Let us now denote 
\begin{equation} \label{eq:bmax_gamma_defs}
    b_{\max} := \max_{\rev{t' \in [0,1]}} b(t'), \quad  \gamma_n(t) := \left(1+\frac{b^{\frac{5}{2}}(t)}{\sqrt{\log n}} \max\{b^2(t), \frac{\log n}{\sqrt{n}}\} \right)
\end{equation}
so that $b_{\max,\delta}(t) \leq b_{\max}$. \rev{Since $\pu \gtrsim 1$ provided that $\delta\pmin \gtrsim \log n$, hence} $\gamma_{n,\delta}(t) \lesssim \gamma_n(t)$.

\begin{remark}\label{rem:b_lip}
    \rev{Suppose that for every $t\in[0,1]$ and $i\neq j$, $y^*_{ij}(t) > y_{\min}^*$, with $y_{\min}^* \in (0,\frac{1}{2})$. Then one can show that $b(t)$ is a Lipschitz function with constant $\frac{M}{y_{\min}^{*^2}}$. Note that
    \[ \frac{1}{y^*_{ij}(t)} = \frac{w^*_{t,j}}{w^*_{t,i}+w^*_{t,j}} \Longrightarrow b_{\max} = \frac{1}{y^*_{\min}} - 1.\]
    Hence, if $y_{\min}^* \gtrsim 1$, then $b_{\max} = O(1)$ and consequently for every $t'\in [0,1]$, $b(t') = O(1)$.}
\end{remark}

Then as for the $\ell_2$ case, one can derive a value for $\delta$ that leads to a $\ell_{\infty}$ error rate of $T^{-\frac{1}{3}}$.
\begin{corollary}\label{cor:linf}
Under the same notations as in Theorem \ref{thm:linf}, for all $t \in [0,1]$ suppose that $n \gtrsim \log n$, $\pmin$ is as in \eqref{eq:pmin_def} and $b_{\max}, \gamma_n(t)$ are as as in \eqref{eq:bmax_gamma_defs}. Choosing $\delta = \min\set{\frac{(\gamma_n(t))^{\frac{2}{3}}(\log n)^{\frac{1}{3}}}{(2M)^{\frac{2}{3}}nb^{\frac{7}{3}}(t)(L\pmin)^{\frac{1}{3}}} T^{2/3}, T}$ and $\du = 3n\pu$, if $T$ is such that $\delta \gtrsim \frac{\log n}{\pmin}$ and
\begin{align*}
\min\set{\left(\frac{(\gamma_n(t))^{\frac{2}{3}}(M \log n)^{\frac{1}{3}}}{2^{\frac{2}{3}}nb^{\frac{7}{3}}(t)(L\pmin)^{\frac{1}{3}}} \right) \frac{1}{T^{1/3}}, M} + \sqrt{\frac{\log n}{n}} \lesssim \frac{1}{b^{5/2}(t)},    
\end{align*}
then with probability at least $1-O(n^{-9})$, 
\begin{align*}
    \frac{\norm{\piest-\pistar}_\infty}{\norm{\pistar}_\infty} &\lesssim \left(\frac{b_{\max}}{1-b_{\max}\sqrt{\frac{\log n}{n}}} \right) \Bigg[\gamma_n(t)\sqrt{\frac{\log n}{Ln\pmin}} \max\set{\frac{(2M)^{1/3} n^{1/2} b^{7/6}(t) (L \pmin)^{1/6}}{\gamma_n^{1/3}(t) (\log n)^{1/6} T^{1/3}}, \frac{1}{\sqrt{T}}} \\
    &+ Mnb^{7/2}(t) \min\set{\frac{\gamma_n^{2/3}(t) (\log n)^{1/3}}{(2M)^{2/3} n b^{7/3}(t) (L \pmin)^{1/3} T^{1/3}} , 1} \Bigg].
    \end{align*}
\end{corollary}
%
%

\rev{Note that when $M > 0$ and $b_{\max} = O(1)$, which implies that $b(t) = O(1)$ and $\gamma_n(t) = O(1)$ (see Remark \ref{rem:b_lip}), then} Corollary \ref{cor:linf} asserts that for $\delta = \Theta(T^{2/3})$, if $n, T$ are large enough, then w.h.p $\norm{\piest-\pistar}_{\infty} = O(T^{-1/3})$. This matches the rate for the pointwise risk for estimating univariate Lipschitz functions (see for e.g. \citep[Theorem 1.3.1]{nemirovski2000topics}).

\subsection{\rev{Different construction of $\Gu$}}\label{sec:gutil}

\rev{As noted in Remark \ref{rem:cond_psum}, the condition $\psum \gtrsim \log n$ appearing in the results of Sections \ref{sec:results_l2}, \ref{subsec:results_linfty} is quite strict as it implicitly imposes that the union graph $\Gu$ is complete. Indeed, this condition comes from our need to bound the quantities $\cardNij$ with high probability (see Lemma \ref{lem:prop_gnp}). It appears to be difficult to get meaningful concentration bounds on $\Nmin, \Nmax$ in the sparser regime where $\delta \pmin = o(\log n)$.
}

\rev{ One way to avoid these difficulties is to construct a graph such that $\cardNij$ is already controlled for any $i\neq j$. To this end, we will consider the graph $\Gutil = G([n],\Eutil)$ in Algorithm \ref{alg:spectral_algo} where
\begin{equation} \label{eq:diff_constr_graph}
    \Eutil := \set{\set{i,j} \, : \, \cardNij \in \left[ \max\left(1,\frac{\psum}{2}\right),\max\left(2\psum,6\log\frac{4}{\delta\pmin}\right)\right]}.
\end{equation}
}

\begin{lemma}\label{lem:ptil}
\rev{Suppose that there exist $\pmin>0$ and $\pmax< \frac{1}{2}$ such that for any $t' \in \calT$, $\pmin \leq p(t')\leq \pmax$. Let $\Gutil = G([n],\Eutil)$ be constructed as in \eqref{eq:diff_constr_graph}.
Then $\Gutil \sim G(n,\putil)$, and the following is true.}
\rev{\begin{enumerate}
    \item If $\delta\pmin \geq 3$, then $\putil \geq 1-2e^{-\frac{1}{4}}$ (hence $\Gutil$ is dense) and \[ \Eutil := \set{\set{i,j} \, : \, \cardNij \in \left[ \frac{\psum}{2},2\psum\right]}.\]
    \item If $\delta\pmax \leq \frac{1}{8}$, then $\putil\in\left[\frac{\delta\pmin}{4},8\delta\pmax\right]$ and 
    \[\Eutil := \set{\set{i,j} \, : \, \cardNij \in \left[ 1,6\log\frac{4}{\delta\pmin}\right]}.\]
\end{enumerate} 
}
\end{lemma}
\rev{The proof of Lemma \ref{lem:ptil} is outlined in Appendix \ref{appsec:proof_diff_constr_graph}. We now show how this can be used to obtain error bounds in the regimes where $\delta \pmin = o(\log n)$. For simplicity, we will only show this for the $\ell_2$ error, but the idea can be used in an analogous manner to obtain $\ell_{\infty}$ bounds as well\footnote{This is omitted due to space considerations.}.} 

\rev{As for the simple union graph $\Gu$, one can introduce analoguous notations for some specific quantities related to the graph $\Gutil$. Let us denote $\dmintil$ (resp. $\dmaxtil$) to be the minimal (resp. maximal) vertex degree in $\Gutil$, and $\xidtil$ to be the spectral gap of the random walk Laplacian of $\Gutil$. Also, let
\[ \Nmintil := \underset{\set{i,j}\in \Eutil}{\min} \cardNij \quad \text{and} \quad \Nmaxtil := \underset{\set{i,j}\in \Eutil}{\max} \cardNij.\]
Then choosing $\dtil =3n\putil$, one can employ Algorithm \ref{alg:spectral_algo} using the data gathered in $\Gutil$, and consequently derive error bounds for the cases $\delta\pmin \gtrsim 1$ (dense regime) and $\delta\pmax \lesssim 1$ (sparse regime). Note that if $\Eutil$ is non-empty -- which will happen w.h.p -- it implies lower (resp. upper) bounds on $\Nmintil$ (resp. $\Nmaxtil$) from \eqref{eq:diff_constr_graph}.}
\begin{remark}
\rev{To construct $\Gutil$ in practice, we would need to estimate $p(t')$ for each $t' \in \calT$ in order to have estimates for $\psum$ and $\pmin$. This was not the case when working with the union graph $\Gu$.}
\end{remark}
\begin{theorem}[Dense regime]\label{thm:denseGu}
    \rev{Suppose that $n\gtrsim \log n$ and $G_{t'}\sim G(n,p(t'))$ for all $t'\in\calT$ so that $\Gutil \sim G(n,\putil)$. Choosing $\dtil = 3n\putil$ in Algorithm \ref{alg:spectral_algo} with the graph $\Gutil$ and $\delta = \min\set{T,\frac{T^{2/3}b^{2/3}(t)}{n(2M)^{2/3}(L\pmin)^{1/3}}}$, if $T$ is such that $\delta\pmin \geq 3$ and
    \begin{equation*}
        \sqrt{\frac{\log n}{Ln \pmin}} \max \set{\frac{(2M)^{1/3} n^{1/2} (L \pmin)^{1/6}}{b^{1/3}(t) T^{1/3}}, \frac{1}{\sqrt{T}}} 
+ Mn \min \set{\frac{b^{2/3}(t)}{(2M)^{2/3} n (L \pmin)^{1/3} T^{1/3}}, 1} \lesssim \frac{1}{b^{7/2}(t)},
    \end{equation*}
    then with probability at least $1-O(n^{-10})$,
\begin{align*}
\frac{\norm{\piest-\pistar}_2}{\norm{\pistar}_2} &\lesssim Mn b^{7/2}(t) \min \set{\frac{b^{2/3}(t)}{(2M)^{2/3} n (L\pmin)^{1/3} T^{1/3}}, 1} \\ 
&+ \frac{b^{9/2}(t)}{\sqrt{Ln\pmin}} \max\set{\frac{(2M)^{1/3} n^{1/2} (L\pmin)^{1/6}}{b^{1/3}(t) T^{1/3}}, \frac{1}{\sqrt{T}}}.
\end{align*}
    }
\end{theorem}

\begin{proof}
    \rev{The proof of this theorem is analoguous to the proof of Corollary \ref{cor:l2}. We simply use Lemma \ref{lem:ptil} (first part) and Lemma \ref{lem:prop_gnp} for the new union graph $\Gutil$ in order to bound $\dmintil, \dmaxtil$, $\xidtil$, $\Nmintil$, $\Nmaxtil$ and $|\Eutil|$, w.h.p.}
\end{proof}
\rev{Observe that Theorem \ref{thm:denseGu} gives the same error bound as in Corollary \ref{cor:l2}, but with the relatively milder condition $\delta\pmin \gtrsim 1$. From Lemma \ref{lem:ptil}, note that this means that $\Gutil$ is still dense. The following theorem provides an error bound in a sparser regime where $\frac{\log n}{n\pmin} \lesssim \delta\lesssim \frac{1}{\pmax}$.}

\begin{theorem}[Sparse regime] \label{thm:l2_diff_graph_fin}
     \rev{Suppose that $n\gtrsim \log n$ and $G_{t'}\sim G(n,p(t'))$ for all $t'\in\calT$ so that $\Gutil \sim G(n,\putil)$. Choosing $\dtil = 3n\putil$ in Algorithm \ref{alg:spectral_algo} with the graph $\Gutil$ and $\delta = \min\set{T,\frac{T^{2/3}b^{2/3}(t)(\log n)^{1/3}}{n(2M)^{2/3}(L\pmin)^{1/3}}}$, if $T$ is such that $\frac{\log n}{n\pmin} \lesssim \delta\lesssim \frac{1}{\pmax}$ and
    \begin{equation}\label{eq:cond_thm_pmaxtil}
        \frac{\log n}{\sqrt{Ln\pmin}}\max \set{\frac{(2M)^{1/3} n^{1/2} (L \pmin)^{1/6}}{b^{1/3}(t) T^{1/3}(\log n)^{1/6}}, \frac{1}{\sqrt{T}}} 
+ Mn \min \set{\frac{b^{2/3}(t)(\log n)^{1/3}}{(2M)^{2/3} n (L \pmin)^{1/3} T^{1/3}}, 1} \lesssim \frac{1}{b^{7/2}(t)},
    \end{equation}
    then with probability at least $1-O(n^{-10})$,
\begin{align*}
\frac{\norm{\piest-\pistar}_2}{\norm{\pistar}_2} &\lesssim Mn b^{7/2}(t) \min \set{\frac{b^{2/3}(t)(\log n)^{1/3}}{(2M)^{2/3} n (L \pmin)^{1/3} T^{1/3}}, 1} \\ 
&+ b^{9/2}(t)\sqrt{\frac{\log n}{Ln\pmin}}\max \set{\frac{(2M)^{1/3} n^{1/2} (L \pmin)^{1/6}}{b^{1/3}(t) T^{1/3}(\log n)^{1/6}}, \frac{1}{\sqrt{T}}}.
\end{align*}
    }   
\end{theorem}

\begin{proof}
    \rev{Similar proof technique as Corollary \ref{cor:l2} using Lemmas \ref{lem:ptil} and \ref{lem:prop_gnp} for $\Gutil$.}
\end{proof}
\rev{The above error bound has an extra $(\log n)^{1/3}$ factor as compared to that of Corollary \ref{cor:l2}, however the dependence on $T$ is unchanged. The condition $\frac{\log n}{n\pmin} \lesssim \delta\lesssim \frac{1}{\pmax}$ implies both a lower and an upper bound on $T$, and is of course feasible provided $\pmin$ and $\pmax$ are of the same order. Let us instantiate Theorems \ref{thm:denseGu} and \ref{thm:l2_diff_graph_fin} on an example to see the conditions therein more clearly.}

%
%
%
%
%
%
\begin{example}
\rev{Let us choose $L=1$, $b(t) \asymp 1$, $\lip \asymp 1$ and $\pmin,\pmax\asymp \frac{1}{n}$. 
\begin{enumerate}
\item (Theorem \ref{thm:l2_diff_graph_fin}) For $\delta = \frac{T^{2/3}(\log n)^{1/3}}{n^{2/3}}$, if $n(\log n)^{5/2} \lesssim T \lesssim \sqrt{\frac{n^5}{\log n}}$, then it holds with high probability that
\begin{align*}
\frac{\norm{\piest-\pistar}_2}{\norm{\pistar}_2} \lesssim \left(\frac{n\log n}{T}\right)^{1/3}.
\end{align*}
\item (Theorem \ref{thm:denseGu}) For $\delta = (\frac{T}{n})^{2/3}$, if $T \gtrsim n^{5/2}$, then it holds with high probability that
\begin{align*}
\frac{\norm{\piest-\pistar}_2}{\norm{\pistar}_2} \lesssim \left(\frac{n}{T}\right)^{1/3}.
\end{align*}
\end{enumerate}
}
\end{example}

%
\section{$\ell_2$-analysis of the spectral estimator}\label{sec:analysis_l2}
We now describe the main ideas that lead to the $\ell_2$ bound in Theorem \ref{thm:l2}. We will essentially proceed in three steps following the ideas in \citep{Chen_2019}. 
\begin{align}
    \norm{\piest-\pistar}_2 & \overset{(i)}{\leq}  \frac{1}{\sqrt{\pistarmin}}\norm{\piest-\pistar}_{\pistar} \nonumber \\
    & \overset{(ii)}{\leq} \frac{8\du b^{7/2}(t)}{\xid\dumin}\norm{\pistar^\top(\Pest-\Pbar)}_2 \label{eq:spec_step2}\\
    & \overset{(iii)}{\leq} \frac{8\du b^{7/2}(t)}{\xid\dumin}\left(4\frac{M\delta\cardEu}{T\dumax}+\tilde{C_2} \sqrt{\frac{\Nmax \dumax b^2(t)}{L\dusq\Nminsq}} \right)\norm{\pistar}_2 \label{eq:spec_step3}.
\end{align}
\begin{enumerate}[(i)]
    \item The first step is easy to verify, due to the definition of the norm $\norm{.}_{\pi^*(t)}$
\begin{equation*}
    \norm{\pistar - \piest}_{\pistar}^2  =  \displaystyle\sum_{i=1}^n \pistari(\pistari -\piesti)^2  
      \geq   \pistarmin \norm{\pistar - \piest}_2^2.
\end{equation*}

    \item This is given by the combination of Lemmas \ref{lem:bounds_deterministic}, \ref{lem:bound_delta} and \ref{lem3} which in turn are derived using \citep[Theorem 8]{Chen_2019} and \citep[Lemma 6]{negahban2015rank}.
    
    \item For this step, we can decompose $\Pest - \Pbar = \Delta(t)+\Delta_1(t)$ as in \eqref{eq:decomposition}. Then one has to bound $\norm{\pistar^\top\Delta(t)}_2$ and $\norm{\pistar^\top\Delta_1(t)}_2$. The second term is completely deterministic and can be bounded using Assumption \ref{assump:smooth1}. A bound on $\norm{\pistar^\top\Delta(t)}_2$ is found following the same steps as in the proof of  \citep[Theorem 9]{chen2020partial}.
\end{enumerate}
%

\subsection{Proof of Theorem \ref{thm:l2}}
A bound on $\norm{\pi^*(t) - \est{\pi}(t)}_{\pi^*(t)}$ is given by \citep[Theorem 8]{Chen_2019}, which is recalled in Appendix \ref{appendix:A}. Denoting $\lammax(\Pbar)$ to be the second largest eigenvalue of $\Pbar$ in absolute value, i.e.,
$$\lammax(\Pbar) = \max\set{\lambda_2(\Pbar),-\lambda_n(\Pbar)},$$
this theorem gives the bound
\begin{equation} \label{eq:thm8bd_chen}
    \norm{\piest-\pistar}_{\pistar} \leq \frac{\norm{\pistar^\top(\Pest-\Pbar)}_2}{1 - \lammax(\Pbar)-\norm{\Pest-\Pbar}_{\pistar}}
\end{equation}
provided that the following condition holds.
\begin{equation}\label{eq:condition_thm8}
    \norm{\Pest-\Pbar}_{\pistar} < 1 - \lambda_{\max}(\Pbar).
\end{equation}
First let us note that these eigenvalues are real, and so \eqref{eq:condition_thm8} is well defined. Indeed, denoting $\Pi^*(t) = \diag(\pistar)$ and $S = \Pi^*(t)^{1/2}\Pbar\Pi^*(t)^{-1/2}$,  $S$ is similar to $\Pbar$, and $S$ is symmetric due to the reversibility of $\Pbar$.  

To prove \eqref{eq:condition_thm8}, we will use results similar to  \citep[Lemma's 3,4]{negahban2015rank}. The main idea is to write the following decomposition 
\begin{equation}\label{eq:decomposition}
    \Pest = \Pbar +\underset{\Delta(t)}{\underbrace{\Pest-\Peststar}} + \underset{\Delta_1(t)}{\underbrace{\Peststar - \Pbar}} 
\end{equation}
where $\Peststar = \expec\Pest$ whose entries are given by
\begin{equation*}
\Peststarij = \left\{\begin{array}{cc}
	\displaystyle\frac{1}{\du} \left(\frac{1}{\cardNij}\sum_{t' \in \Nij} \frac{w^*_{t',j}}{w^*_{t',i}+w^*_{t',j}} \right) & \mbox{if} \quad \set{i,j} \in \Eu \\
	\displaystyle 1-\frac{1}{\du}\sum_{k \neq i} \frac{1}{\cardNik}\displaystyle\sum_{t' \in \Nik} \frac{w^*_{t',k}}{w^*_{t',i}+w^*_{t',k}} & \mbox{if} \quad i=j\\
	0 & \mbox{otherwise.}
\end{array}
\right.
\end{equation*}
We now provide bounds on $\norm{\Delta_1(t)}_2$ and $\norm{\Delta(t)}_2$ in Lemma's  \ref{lem:bounds_deterministic} and  \ref{lem:bound_delta} respectively. The proofs of all results in this section are outlined in  Appendix \ref{appsec:proofs_l2_analysis}.
%
%
\begin{lemma}[Bound on $\norm{\Delta_1(t)}_2$]\label{lem:bounds_deterministic}
It holds that $\norm{\Delta_1(t)}_2 \leq 4\frac{M\delta\cardEu}{T\dumax}.$
\end{lemma}

%
%
\begin{lemma}[Bound on $\norm{\Delta(t)}_2$]\label{lem:bound_delta}
There exists a constant $\tilde{C_1} \geq 15$ such that
\begin{equation*}
   \norm{\Delta(t)}_2 \leq \tilde{C_1}\sqrt{\frac{\Nmax\dumax\log n}{L\dusq\Nminsq}}
\end{equation*}
with probability at least $1 - O(n^{-10})$.
\end{lemma}
The proof of Lemma \ref{lem:bounds_deterministic} follows from the smoothness condition in Assumption \ref{assump:smooth1}, while the proof of Lemma \ref{lem:bound_delta} follows the proof steps of \citep[Lemma 3]{negahban2015rank}. Next, we show that if $\xid > 0$ (which implies that $\Gu$ is connected) and if the perturbation $\norm{\Pest - \Pbar}_2$ is sufficiently small, then we can ensure \eqref{eq:condition_thm8}. 
\begin{lemma}\label{lem3}
Denoting $\rho(t) = \lammax(\Pbar) + \norm{\Pest-\Pbar}_2\sqrt{\frac{\pistarmax}{\pistarmin}}$, recall that $\xid = \lambda_{max}(\Ld)$ (where $\Ld$ is the Laplacian of $\Gu$) and $b(t) := \max_{i,j} \frac{w^*_{t,i}}{w^*_{t,j}} = \frac{\pistarmax}{\pistarmin}$. If $\xid > 0$ and $\du \geq \dumax$, then we have that $1-\lammax(\Pbar) \geq \frac{\xid\dumin}{4\du b^3(t)}$. Moreover, if %
\begin{align}\label{eq:cond_lem3}
    \norm{\Pest - \Pbar}_2 \leq \frac{\xid\dumin}{8\du b^{7/2}(t)}
\end{align}
then it holds that
\begin{equation*}
    1 - \rho(t) \geq \frac{\xid\dumin}{8\du b^3(t)} >0.
\end{equation*}
\end{lemma}
The statement is analogous to that of \citep[Lemma 4]{negahban2015rank}. The 
bound on $1-\lammax(\Pbar)$ is clearly the crucial statement, and requires using \citep[Lemma 6]{negahban2015rank}. We remark in passing that the dependence on $b(t)$ is $b^3(t)$ in Lemma \ref{lem3}, we could not verify the dependence stated in \citep[Lemma 6]{negahban2015rank} (which is $b^2(t)$). For completeness, we outline the proof of Lemma \ref{lem3} in Appendix \ref{appsec:proofs_l2_analysis}.   

Condition \eqref{eq:cond_lem3} is ensured via Lemma's \ref{lem:bounds_deterministic} and \ref{lem:bound_delta} (with high probability) whenever \eqref{eq:cond_thm1} holds. 
Then, \eqref{eq:bound_pinorm} readily implies that
%
%
%
\begin{align}\label{eq:gamma}
    1-\lambda_{\max}(\Pbar) - \norm{\Pest-\Pbar}_{\pistar}   \geq 1 - \rho(t) 
    \geq  \frac{\xid\dumin}{8\du b^3(t)} > 0,
\end{align}
thus ensuring \eqref{eq:condition_thm8}. Using \eqref{eq:thm8bd_chen} and \eqref{eq:gamma} we finally obtain \eqref{eq:spec_step2} as follows.
\begin{align} \label{eq:l2_gen_bd_tmp}
    \norm{\piest - \pistar}_{\pistar} 
    \leq  \frac{8\du b^3(t)}{\xid\dumin}\norm{\pistar^\top(\Pest-\Pbar)}_{\pistar} 
    \leq  \frac{8\du b^{7/2}(t)}{\xid\dumin}\norm{\pistar^\top(\Pest-\Pbar)}_2
\end{align}
where the last inequality uses \eqref{eq:bound_pinorm}.

Finally, we can bound  $\norm{\pistar^\top(\Pest-\Pbar)}_2$ 
using the decomposition in \eqref{eq:decomposition} along with the triangular inequality, leading to
\begin{equation*}
    \norm{\pistar^\top(\Pest-\Pbar)}_2 \leq \norm{\pistar^\top\Delta(t)}_2 + \norm{\pistar^\top\Delta_1(t)}_2.
\end{equation*}
Bounds on $\norm{\pistar^\top\Delta(t)}_2$ and $\norm{\pistar^\top\Delta_1(t)}_2$ are provided in the following lemma's. 
\begin{lemma}
\begin{equation}
    \norm{\pistar^\top\Delta_1(t)}_2 \leq \norm{\pistar}_2\norm{\Delta_1(t)}_2 \leq 4\frac{M\delta\cardEu}{T\dumax}\norm{\pistar}_2.
\end{equation}
\end{lemma}
The statement follows directly from Lemma \ref{lem:bounds_deterministic}.
\begin{lemma}\label{lem:5}
There exist constants $c_1 > 0, \tilde{C_2} \geq 1$ such that if $n \geq c_1\log n$ then with probability at least $1-O(n^{-10})$, we have that
\begin{equation*}
    \norm{\pistar^\top\Delta(t)}_2 \leq \tilde{C_2}\sqrt{\frac{\Nmax\dumax b^2(t)}{L\dusq\Nminsq}}\norm{\pistar}_2.
\end{equation*}
\end{lemma}
The proof of Lemma \ref{lem:5} follows the ideas in the proof of \citep[Theorem 9]{Chen_2019}. Applying these bounds in \eqref{eq:l2_gen_bd_tmp} finally leads to the stated bound in Theorem \ref{thm:l2}.
 
%
\subsection{Proof of Theorem \ref{thm:l2_gnp}}
This theorem follows directly from Theorem \ref{thm:l2} and from the propreties of Erdös-Renyi graphs gathered in Lemma \ref{lem:prop_gnp}.
Using \eqref{eq:cond_thm_gnp} and Lemma \ref{lem:prop_gnp} along with the choice $\du = 3n\pu$, it holds with probability at least $1-O(n^{-10})$ that
\begin{equation*}
    \tilde{C_1}\sqrt{\frac{\Nmax\dumax\log n}{L\dusq\Nminsq}} + 4\frac{M\delta\cardEu}{T\dumax} \leq 2\tilde{C_1}\sqrt{\frac{\log n}{Ln\pu\psum}} + 16\frac{M\delta n}{T} \leq \frac{1}{96 b^{\frac{7}{2}}(t)} \leq \frac{\xid\dumin}{8\du b^{\frac{7}{2}}(t)}.
\end{equation*}
Hence, condition \eqref{eq:cond_thm1} is satisfied with high probability and Theorem \ref{thm:l2} implies that
\begin{align*}
    \frac{\norm{\piest - \pistar}_2}{\norm{\pistar}_2} & \leq 32\frac{M\delta\cardEu \du b^{7/2}(t)}{T\xid\dumin\dumax}
    + 8\tilde{C_2}\frac{b^{9/2}(t)}{\xid\dumin}\sqrt{\frac{\Nmax\dumax}{L\Nminsq}}.
\end{align*}
Again, using Lemma \ref{lem:prop_gnp}, we can simplify the above bound so that with probability at least $1-O(n^{-10})$,
\begin{align*}
    \frac{\norm{\piest - \pistar}_2}{\norm{\pistar}_2} & \leq 1536 \frac{M\delta n^2\pu b^{7/2}(t)}{Tn\pu}+ 64\tilde{C_2}\frac{b^{9/2}(t)}{n\pu}\sqrt{\frac{3\psum n\pu}{L\psumsq}} \\
    & \leq 1536 \frac{M\delta nb^{7/2}(t)}{T}+ 64\tilde{C_2}b^{9/2}(t)\sqrt{\frac{3}{Ln\pu\psum}}.
\end{align*}

\subsection{Proof of Corollary \ref{cor:l2}}
Since $\psum \geq \pmin \delta$, therefore the condition $\delta \gtrsim \frac{\log n}{\pmin}$ implies $\psum \gtrsim \log n$, as well as $\pu \gtrsim 1$ (due to Proposition \ref{prop:pu}), thus satisfying the requirements of Theorem \ref{thm:l2_gnp}. Additionally, $\delta$ is required to satisfy $\delta \leq T$, and also condition \eqref{eq:cond_thm_gnp}, i.e., 
\begin{equation} \label{eq:cond_thm_gnp_tmp1}
    \sqrt{\frac{\log n}{Ln\delta\pmin}} + \frac{M\delta n}{T}\lesssim \frac{1}{b^{\frac{7}{2}}(t)}.
\end{equation}
If $\delta$ satisfies the three aforementioned conditions, and if $n \gtrsim \log n$, we have with probability at least $1-O(n^{-10})$ the $\ell_2$ bound
\begin{equation}\label{eq:cor1_bound_thm_2}
        \frac{\norm{\piest-\pistar}_2}{\norm{\pistar}_2} \lesssim \frac{M\delta nb^{\frac{7}{2}}(t)}{T} + \frac{b^{\frac{9}{2}}(t)}{\sqrt{Ln\delta\pmin}}.
\end{equation}
The optimal choice of $\delta \in (0,T]$  that minimizes the RHS of \eqref{eq:cor1_bound_thm_2} is easily verified to be
\begin{equation*}
    \delta^* = \min\set{\frac{(b(t)T)^{\frac{2}{3}}}{(2M)^{\frac{2}{3}}n(L\pmin)^{\frac{1}{3}}}, T}.
\end{equation*}
Now it remains to ensure that $\delta^*$ satisfies the previously stated conditions on $\delta$. Clearly $\delta^* \leq T$, and condition \eqref{eq:cond_thm_gnp_tmp1} are equivalent to the stated conditions on $T$ in the corollary. Hence provided $\delta \asymp \delta^*$, $T$ satisfies the stated conditions, and $n \gtrsim \log n$, we arrive at the stated $\ell_2$ bound in the corollary.
%
%

%
\section{$\ell_\infty$-analysis of the spectral estimator}\label{sec:analysis_linf}
The main goal of this section is to present the steps of the proof of Theorem \ref{thm:linf}. We will follow the steps of the analysis carried out by Chen et al. \citep{Chen_2019}, and adapt it to our setting. The proofs of all results from this section are provided in Appendix \ref{appsec:linf_analysis}. 

As $\piest$ and $\pistar$ are stationary distributions of the transition matrices $\Pest$ and $\Pbar$, then denoting $P_{.m}$ to be the $m^{th}$ column of a matrix $P$, it holds for all $m \in [n]$ that,
\begin{align*}
    \piestm-\pistarm & = \left(\piest^\top\Pest\right)_m - \left(\pistar^\top\Pbar\right)_m \\
    & = \piest^\top\est{P}_{.m}(t) - \pistar^\top\bar{P}_{.m}(t) \\
    & = \pistar^\top(\est{P}_{.m}(t) - \bar{P}_{.m}(t)) + (\piest - \pistar)^\top\est{P}_{.m}(t) \\
    & = \underset{:=I_0^m}{\underbrace{\pistar^\top(\est{P}_{.m}(t) - \est{P}^*_{.m}(t))}} + \underset{:=I_1^m}{\underbrace{\pistar^\top(\est{P}^*_{.m}(t) - \bar{P}_{.m}(t))}}+  \underset{:=I_2^m}{\underbrace{(\piestm - \pistarm)\est{P}_{mm}(t)}} \\
    & + \sum_{j:j\neq m} (\piestj-\pistarj)\Pestjm.\\
\end{align*}
Let us first focus on bounding $\abs{I_0^m}, \abs{I_1^m}$ and $\abs{I_2^m}$. The last term will be treated carefully due to the occurrence of some statistical dependencies therein. Since $\Peststar = \expec[\Pest]$, hence $\abs{I_0^m}$ can be bounded using Hoeffding's inequality.
\begin{lemma}\label{lem:I0}
Suppose that $n\pu \geq c_0\log n$ and $\psum \geq c_2 \log n$ for constants $c_0, c_2 \geq 1$ from Lemma \ref{lem:prop_gnp}. Then there exists a constant $C_1\geq 1$ such that with probability at least $1-O(n^{-9})$,
\begin{equation}
    \abs{I_0^m} \leq C_1\sqrt{\frac{\log n}{Ln\pu\psum}}\norm{\pistar}_\infty \quad \forall m \in [n].
\end{equation}
\end{lemma}
The second term $I_1^m$ can be bounded easily using Assumption \ref{assump:smooth1}. 
%
%
\begin{lemma}\label{lem:I1}
With probability at least $1-O(n^{-10})$, we have 
\begin{equation}
    \abs{I_1^m} \leq \frac{2M\delta}{T}\norm{\pistar}_\infty \quad \forall m \in [n].
\end{equation}
\end{lemma}
The next lemma shows that with high probability, $\abs{I_2^m} \lesssim \norm{\piest-\pistar}_{\infty}$ for all $m \in [n]$.
\begin{lemma}\label{lem:I2}
Recall that $\bmax = \max_{t\in\Nt} b(t)$ and suppose that $n\pu \geq c_0\log n$ and $\psum \geq c_2 \log n$ for constants $c_0, c_2 \geq 1$ from Lemma \ref{lem:prop_gnp}. Then there exists a  constant $C_2 \geq 1$ such that it holds with probability at least $1-O(n^{-9})$ that 
\begin{equation}
    \abs{I_2^m} \leq \left(1-\frac{1}{12\bmax}+C_2\sqrt{\frac{\log n}{Ln\pu\psum}} \right)\norm{\piest-\pistar}_{\infty} \quad \forall m \in [n].
\end{equation} 
\end{lemma}
The last term to bound is more difficult to handle due to the statistical dependency between $\piest$ and $\Pest$. The idea is then to use the same ``leave-one-out'' trick as in \citep{Chen_2019} and introduce a new matrix $\Pmest$ with entries given by (for all $i\neq j$) 
\begin{equation*}
    \Pmestij = \left\lbrace \begin{array}{cc}
        \Pestij & \text{if }i\neq m, j\neq m \\
        \frac{\pu}{\du}\frac{w_{t,j}^*}{w_{t,j}^*+w_{t,i}^*} & \text{if }i= m \text{ or } j= m.
    \end{array} \right.
\end{equation*}
To ensure that $\Pmest$ is a transition matrix, its diagonal entries are defined as 
\begin{equation*}
    \Pmestii =  1 - \displaystyle\sum_{k\neq i}\Pmestik.
\end{equation*}
Here, the $m^{th}$ line and column of $\Pest$ have been replaced by their expectation, unconditionally of the union graph $\Gu$. 
Let us denote $\pimest$ to be the leading left eigenvector of $\Pmest$. This vector is now statistically independent of the connectivity and the pairwise comparison outputs involving the $m^{th}$ item, and one can reasonably expect that it is close to $\pistar$. Hence, we can decompose the last term as
\begin{equation*}
    \sum_{j:j\neq m}\left(\piestj-\pistarj\right)\Pestjm = \underset{:= I_3^m}{\underbrace{\sum_{j:j\neq m}\left(\piestj-\pimestj\right)\Pestjm}} + \underset{:= I_4^m}{\underbrace{\sum_{j:j\neq m}\left(\pimestj-\pistarj\right)\Pestjm}}. 
\end{equation*}
To bound $\abs{I_3^m}$, note that the  Cauchy-Schwarz inequality implies 
\begin{equation*}
    \abs{I_3^m} \leq \sqrt{\sum_{j:j\neq m} \Pestjm^2}\norm{\piest-\pimest}_2.
\end{equation*}
Since for all $j \neq m$, $\Pestjm \leq \frac{1}{\du}$, hence $\abs{I_3^m} \leq \frac{\sqrt{\dumax}}{\du} \norm{\piest-\pimest}_2$. The important step now is to bound $\norm{\piest-\pimest}_2$ since $\dumax \leq \du$ w.h.p (when $G_{t'} \sim \calG(n,p(t'))$ for all  $t' \in \Nt$) due to Lemma \ref{lem:prop_gnp}. This is shown in the following lemma.
\begin{lemma}\label{lem:I3}
Suppose that $n\pu \geq c_0\log n$ and $\psum \geq c_2 \log n$ for constants $c_0, c_2 \geq 1$ from Lemma \ref{lem:prop_gnp}. Then there exist  constants $C_3, C_6 \geq 1$ such that if
\begin{equation}\label{eq:cond_lem9}
    96b^{\frac{5}{2}}(t) \left(\frac{4M\delta}{T}+C_3\sqrt{\frac{\log n}{n\pu}} \right) \leq \frac{1}{2}, 
\end{equation}
then it holds with probability at least $1-O(n^{-9})$ that for all $m \in [n]$,
\begin{equation*}
    \norm{\pimest-\piest}_2 \leq 192b^{\frac{5}{2}}(t)\left(\frac{4M\delta}{T} + C_3\sqrt{\frac{\log n}{Ln\pu\psum}} \right)\norm{\pistar}_{\infty}+\norm{\piest-\pistar}_{\infty}.
\end{equation*}
Consequently, we have with probability at least $1-O(n^{-9})$ that for all $m \in [n]$,
\begin{equation}
    \abs{I_3^m} \leq 192\frac{b^{\frac{5}{2}}(t)}{\sqrt{3n\pu}}\left(\frac{4M\delta}{T} + C_3\sqrt{\frac{\log n}{Ln\pu\psum}} \right)\norm{\pistar}_{\infty}+\frac{1}{\sqrt{3n\pu}}\norm{\piest-\pistar}_{\infty}.
\end{equation}
\end{lemma}
Finally, we can bound $\abs{I_4^m}$ using the statistical independence between $\pimest$ and $\est{P}_{.m}(t), \est{P}_{m.}(t)$.
\begin{lemma}\label{lem:I4}
Suppose that $n\pu \geq c_0\log n$ and $\psum \geq c_2 \log n$ for constants $c_0, c_2 \geq 1$ from Lemma \ref{lem:prop_gnp}; $n\geq c_1\log n$ for the constant $c_1 > 0$ in Theorem \ref{thm:l2}, and that condition \eqref{eq:cond_thm_gnp} of Theorem \ref{thm:l2_gnp} holds. Then there exist  constants $C_7,C_8,C_9 \geq 1$ such that  with probability at least $1-O(n^{-9})$, we have for all $m \in [n]$,
\begin{align*}
    \abs{I_4^m} \leq & \left(C_7\frac{Mn\delta b^{\frac{7}{2}}(t)}{T} + C_8 \frac{b^{\frac{5}{2}}(t)\max\set{b^2(t),\frac{\log n}{\sqrt{n\pu}}}}{\sqrt{Ln\pu\psum}}\right)\norm{\pistar}_{\infty} + C_9\sqrt{\frac{\log n}{n\pu}}\norm{\piest-\pistar}_{\infty}.
\end{align*}
\end{lemma}

%
\subsection{Proof of Theorem \ref{thm:linf}}
As seen in Section \ref{sec:analysis_linf}, the bound in  Theorem \ref{thm:linf} follows from the combination of the bounds on $I_0^m,I_1^m,I_2^m,I_3^m,$ and $I_4^m$. Note that we can identify two types of terms in these bounds -- those depending on $\norm{\pistar}_{\infty}$ and the ones which depend on $\norm{\piest-\pistar}_{\infty}$. Then, our bound can be written as 
\begin{equation*}
    \norm{\piest-\pistar}_{\infty} \leq \alpha\norm{\pistar}_{\infty} + \beta\norm{\piest-\pistar}_{\infty}
\end{equation*}
where $\beta < 1$, which in turn implies  $\frac{\norm{\piest-\pistar}_{\infty}}{\norm{\pistar}_{\infty}} \leq \frac{\alpha}{1-\beta}$.

We can lower bound the term $1-\beta$ as
\begin{align}\label{eq:bound_beta}
    1-\beta = & 1 - \left(1-\frac{1}{12\bmax}+C_2\sqrt{\frac{\log n}{Ln\pu\psum}} + \frac{1}{\sqrt{n\pu}} + C_9\sqrt{\frac{\log n}{n\pu}}\right) \nonumber\\
    = & \frac{1}{12\bmax}-C_2\sqrt{\frac{\log n}{Ln\pu\psum}} - \frac{1}{\sqrt{n\pu}} - C_9\sqrt{\frac{\log n}{n\pu} \nonumber}\\
    \geq & \frac{1}{12\bmax}-\tilde{C}\sqrt{\frac{\log n}{n\pu}} \quad (\text{ since } \psum \geq \log n)
\end{align}
for some constant $\tilde{C} \geq 1$.

Concerning the terms in $\alpha$, one can divide them in two groups depending on whether they depend on $T$ or not. The sum of the terms depending on $T$ is 
\begin{align*}
    \frac{2M\delta}{T} + 768\frac{M\delta b^{\frac{5}{2}}(t)}{T\sqrt{3n\pu}} + C_7\frac{Mn\delta b^{\frac{7}{2}}(t)}{T} \leq \tilde{C_6}\frac{Mn\delta b^{\frac{7}{2}}(t)}{T}
\end{align*}
for some constant $\tilde{C_6} \geq 1$.
Furthermore, the sum of the terms which are independent of $T$ is 
%
%
\begin{align*}
    &C_1\sqrt{\frac{\log n}{Ln\pu\psum}} + 192C_3b^{\frac{5}{2}}(t)\sqrt{\frac{\log n}{3L\psum n^2p^2_\delta(t)}} +  C_8 \frac{b^{\frac{5}{2}}(t)\max\set{b^2(t),\frac{\log n}{\sqrt{n\pu}}}}{\sqrt{Ln\pu\psum}} \\ 
    &\leq \left(C_1 + 192 C_3 \frac{b^{\frac{5}{2}}(t)}{\sqrt{3n\pu}} + C_8 \frac{b^{\frac{5}{2}}(t)}{\sqrt{\log n}}\max\set{b^2(t),\frac{\log n}{\sqrt{n\pu}}} \right)\sqrt{\frac{\log n}{Ln\pu\psum}} \\
    &\leq \tilde{C_5} \underbrace{\left(1 + \frac{b^{\frac{5}{2}}(t)}{\sqrt{\log n}}\max\set{b^2(t),\frac{\log n}{\sqrt{n\pu}}}  \right)}_{=: \gamma_{\delta,n}(t)}  \sqrt{\frac{\log n}{Ln\pu\psum}} \quad (\text{ since } n\pu \geq \log n)
\end{align*}
for some constant $\tilde{C_5} \geq 1$. 
Thus we arrive at the bound
\begin{equation}\label{eq:bound_alpha}
    \alpha \leq \tilde{C_5}\gamma_{n,\delta}(t)\sqrt{\frac{\log n}{Ln\pu\psum}} + \tilde{C_6}\frac{Mn\delta b^{\frac{7}{2}}(t)}{T}.
\end{equation}

Combining \eqref{eq:bound_beta} and \eqref{eq:bound_alpha}, we readily arrive at the stated bound in Theorem \ref{thm:linf}.
%
%
%
\subsection{Proof of Corollary \ref{cor:linf}}
As in the proof of Corollary \ref{cor:l2}, since $\psum \geq \pmin\delta$, the condition $\delta \gtrsim \frac{\log n}{\pmin}$ implies $\psum \gtrsim \log n$ and $\pu \gtrsim 1$. Moreover, $\delta$ has to statisfy $\delta \leq T$ and condition \eqref{eq:cond_thm3}, i.e.
\begin{equation*}
    \frac{M\delta}{T}+\sqrt{\frac{\log n}{n}} \lesssim \frac{1}{b^{\frac{5}{2}}(t)}.
\end{equation*}

If $\delta$ satisfies these assumptions and if $n \gtrsim \log n$, we have with probability at least $1-O(n^{-9})$,
\begin{equation}\label{eq:bound_cor_inf}
    \frac{\norm{\piest-\pistar}_\infty}{\norm{\pistar}_\infty} \lesssim \frac{b_{\max}}{1-b_{\max}\sqrt{\frac{\log n}{n}}}\left( \gamma_n(t) \sqrt{\frac{\log n}{Ln\delta\pmin}} + \frac{Mn\delta b^{\frac{7}{2}}(t)}{T}\right).
\end{equation}
%
%
The optimal choice of $\delta \in (0,T]$ minimizing the RHS of \eqref{eq:bound_cor_inf} is given by
\begin{equation*}
    \delta^* = \min\set{\frac{(\gamma_n(t) T)^{\frac{2}{3}}(\log n)^{\frac{1}{3}}}{(2M)^{\frac{2}{3}}nb^{\frac{7}{3}}(t)(L\pmin)^{\frac{1}{3}}} , T}.
\end{equation*}
It now remains to check that $\delta^*$ satisfies the aforementioned conditions on $\delta$. Clearly $\delta^* \leq T$ while $\delta^* \gtrsim \frac{\log n}{\pmin}$ and condition \eqref{eq:cond_thm3} are ensured for the stated conditions on $T$. Hence, for $\delta \asymp \delta^*$, if $T$ satisfies the stated conditions and $n\gtrsim\log n$, we arrive at the stated expression for the $\ell_\infty$ error bound in the corollary.

\section{Experiments}\label{sec:experiments}
We now empirically evaluate the performance of our method via numerical experiments\footnote{Code available at : \url{https://github.com/karle-eglantine/Dynamic_Rank_Centrality}} on synthetic data, and on a real dataset. We will in particular compare our method with the MLE approach of Bong et al. \citep{bong2020nonparametric}, as it is for now the only other method we are aware of that theoretically analyzes a `smoothly evolving' dynamic BTL model as us.
\subsection{Synthetic data}
The synthetic data is generated as follows.
\begin{enumerate}
    \item For $i \in [n]$, we simulate the strength of item $i$ across the grid $w^*_i \in \matR^{|\calT|}$ using the same gaussian process $GP(\mu_i,\Sigma_i)$ as in \citep{bong2020nonparametric}.
    \begin{itemize}
        \item $\mu_i = (\mu_i(0), \dots ,\mu_i(T))$ with $\mu_i(t)\sim \calN(0,0.1)$ for all $t\in\calT$, and 
        
        \item $\Sigma_i$ is a Toeplitz symmetric matrix, with its coefficients in the first row defined as $\Sigma_{i,1t} = 1-\frac{t}{T+1}$.
    \end{itemize}
    We then define $w^*_i = \exp\left(GP(\mu_i,\Sigma_i)\right) \in \matR^{|\calT|}$ for each $i \in [n]$.
    
    \item For all $t\in\calT$, we simulate an Erdös-Renyi comparison graph $G(n,p(t))$ with $p(t)$ chosen randomly from the interval $[\frac{1}{n}, \frac{\log n}{n}]$. We check that the union graph of all the data on the grid $\calT$ is connected. Indeed, it is a sufficient condition for the existence at all times $t \in [0,1]$ of a $\delta$ such that the union graph $\Gu$ is connected (which is required for the ranking recovery).
    
    \item For all $t\in\calT$, for all $1\leq i<j \leq n$ and for all $l\in[L]$, we draw the outcomes of the comparisons as $y_{ij}^{(l)}(t) \sim \calB\left(\frac{w^*_{t,j}}{w^*_{t,i}+w^*_{t,j}}\right)$ and define $y_{ji}^{(l)}(t) = 1-y_{ij}^{(l)}(t)$.
\end{enumerate}
    Starting with an initial value of $\delta$ as in Corollary \ref{cor:l2} ($\delta \approx T^{\frac{2}{3}}$), we increase its value till the union graph $\Gu$ is connected. We then recover the weight vectors $w^*_t$ for all $t \in \calT$ using Algorithm \ref{alg:spectral_algo}. This process is repeated over \rev{$60$} Monte Carlo runs.

Apart from the MLE approach \citep{bong2020nonparametric}, we also evaluate against an adaptation of the simple Borda Count method from the static setting to the dynamic setup. \rev{Let us describe briefly those two methods.}

\paragraph{\rev{MLE method.}} \rev{The analysis of Bong et al \citep{bong2020nonparametric} relies on a kernel smoothing of the data followed by a maximum likelihood estimation. We choose the bandwidth parameter $h= T^{-3/4}$ for the kernel smoothing, as suggested in their synthetic experiments\footnote{see their GitHub repository \url{https://github.com/shamindras/bttv-aistats2020}}. An alternative is to use a cross-validation procedure to tune $h$, but we avoid this method for computational reasons (see Tables \ref{tab:run_time_N100}, \ref{tab:run_time_N400}).}

\paragraph{Borda Count.} This method, analysed by Ammar and Shah \citep{ammar2011ranking} in the static case, gives a score to each item based on its win rate. To estimate the scores at time $t$ in our dynamic setup, we compute the win rate of each item $i$ using the neighborhood $\Nt$ as
\begin{equation*}
    s(t,i) = \frac{\sum_{t'\in\Nt}\text{Number of times }i \text{ has won at time } t'}{\sum_{t'\in\Nt}\text{Number of times }i \text{ has been compared at time } t'}.
\end{equation*}
The scores $(s(t,i))_{i \in [n]}$ then yield a ranking of the items at time $t$ (the higher the winrate for an item, the better its rank). \rev{Note that the neighborhood used to compute these scores is the same neighborhood used in the DRC estimation. Hence, the parameter $\delta$ is chosen as its theoretical optimal value $\delta^* \simeq T^{2/3}$.}

\paragraph{Ranking error.} In order to compare the rankings produced by these three methods, we compute for all them an estimation error with the error metric defined by Negahban \citep{negahban2015rank}. Let $\pi^* = \frac{w^*}{\norm{w^*}_1}$ denote the normalized true weight vector and $\sigma$ denote an estimated ranking, with $\sigma_i<\sigma_j \Leftrightarrow i$ is better than $j$. The error metric is then defined\footnote{The quantities $\pi^*$, $\sigma$ and $\est{\pi}$ are obviously defined at the time instant $t$ where the estimation is being carried out, however we suppress the dependence on $t$ for ease of exposition.} as follows.
\begin{equation*}
    D_{\pi^*}(\sigma) = \sqrt{\frac{1}{2n\norm{\pi^*}_2^2}\sum_{i<j}(\pi^*_i-\pi^*_j)^2 \mathbf{1}_{(\pi^*_i-\pi^*_j)(\sigma_i-\sigma_j)>0}}.
\end{equation*}
It has been shown in \citep[Lemma 1]{negahban2015rank} that this error criterion is related to the $\ell_2$ error that we have bounded in Theorem \ref{thm:l2}. If the ranking $\sigma$ comes from a weight vector $\est{\pi}$, then 
\begin{equation}\label{eq:bound_metric}
    D_{\pi^*}(\sigma) \leq \frac{\norm{\pi^*-\est{\pi}}_2}{\norm{\pi^*}_2}.
\end{equation}
Although \eqref{eq:bound_metric} doesn't necessarily require $\est{\pi}$ to satisfy $\norm{\est{\pi}}_1 = 1$, we will impose this in what follows since $\est{\pi}$ will be the strength estimates returned by the methods being compared.
\paragraph{\rev{Results.}} \rev{The results are summarized below.}
\begin{enumerate}
\item \rev{In Figure \ref{fig:plots_mse}, we consider $T$ ranging from $10$ to $150$, fix $L = 5$ and take $n=100$ or $n=400$. We plot the mean $\ell_2$ error $\norm{\pi^*-\est{\pi}}_2$ versus $T$ for our Algorithm \ref{alg:spectral_algo} (dubbed DRC for Dynamic Rank Centrality), as well as the MLE method (since Borda Count is not designed for recovering the latent weight vector $w^*$). One can observe that the $\ell_2$ error decreases with $T$ for both DRC and MLE and they achieve similar performance, which is consistent with the theoretical results developed in the present paper, and in \citep{bong2020nonparametric}. The error bars illustrate that the variance of the errors typically decreases with $T$. We note that the error curves for $n=400$ are lower than those for $n=100$, and the variance of the errors are also slightly smaller for $n=400$.}

\item \rev{We show in Figure \ref{fig:plots_error_metric} the evolution of the mean ranking error $D_{\pi^*}(\sigma)$ as a function of $T$, where for each $T$ the mean is taken across all time instants in $\calT$, and all Monte Carlo runs.  We compute this error for the DRC, MLE and Borda Count methods, with $\sigma, \pi$ denoting the estimated ranks and weights by these algorithms. Using \eqref{eq:bound_metric}, this implies that $D_{\pi^*}(\sigma)$ should decrease with $T$ which is what we observe in Figure \ref{fig:error_metric_N100}. As for the $\ell_2$ error, MLE and DRC have similar performance. The Borda Count method performs well for rank recovery, as its error is only slightly worse than the other methods.}

\item \rev{We plot the evolution of the $\ell_\infty$ error $\norm{\pi^*-\est{\pi}}_\infty$ versus $T$ for $n=100$ and $n=400$ in Figure \ref{fig:linf}. We observe that the errors decrease with $T$, as shown theoretically, and that both MLE and DRC methods perform similarly.}

\item \rev{In Figure \ref{fig:perf}, we show that the optimal value derived theoretically for $\delta$ is close to the numerically optimal $\delta$. The minimal $\ell_2$ error is indeed obtained for $\delta$ close to $\delta^* \simeq T^{2/3}$.}

\item \rev{As the DRC and MLE methods performs similarly in term of the estimation error, it is of interest to look at their computational cost. We track the running time of both methods, for different values of $n$ and $T$, as shown in Tables \ref{tab:run_time_N100} and \ref{tab:run_time_N400}. We can see that the DRC method is far quicker than the MLE. Indeed, the MLE method solves an optimisation problem via a gradient descent algorithm whereas the DRC method only solves an eigenvalue problem.}
\end{enumerate}

\begin{figure}[!h]
\centering
\begin{subfigure}{.5\textwidth}
  \centering
  \includegraphics[width=\linewidth]{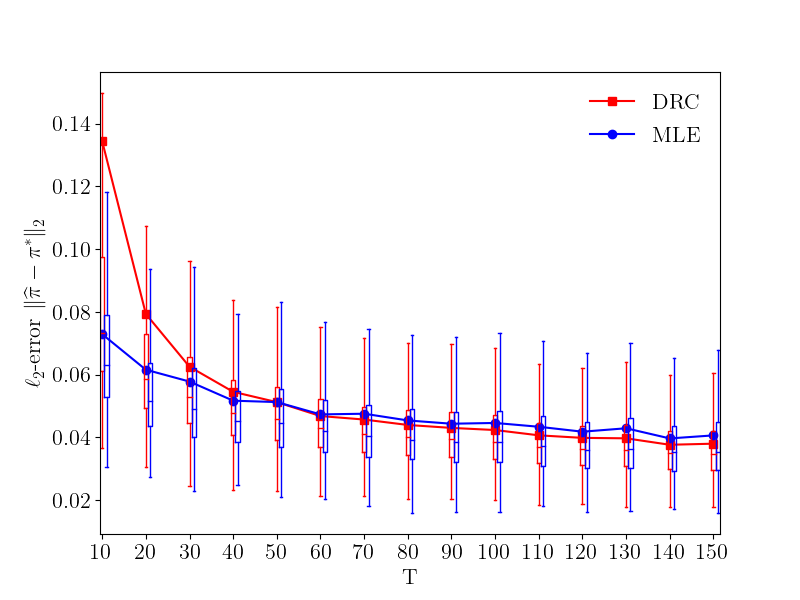}
  \caption{Evolution of the $\ell_2$ error for $n=100$}
  \label{fig:mse_N100}
\end{subfigure}%
\begin{subfigure}{.5\textwidth}
  \centering
  \includegraphics[width=\linewidth]{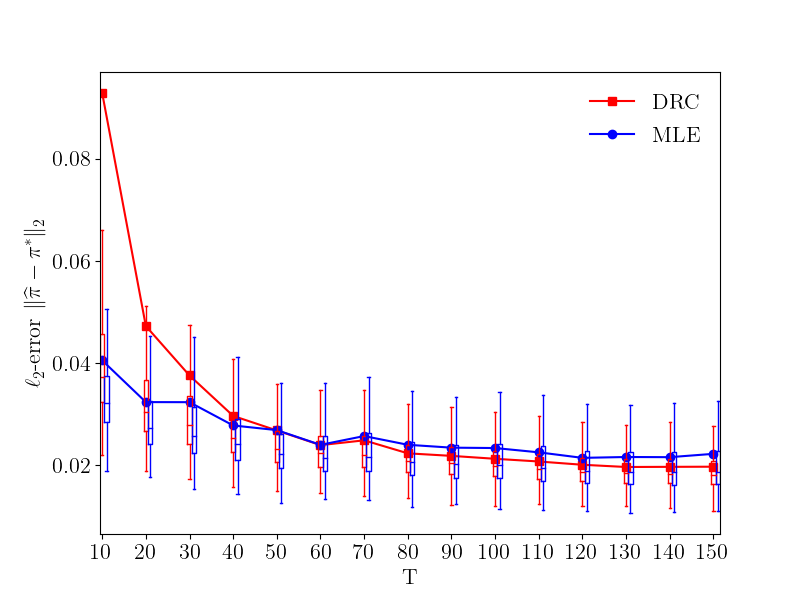}
  \caption{Evolution of the $\ell_2$ error for $n=400$}
  \label{fig:mse_N400}
\end{subfigure}
\caption{\rev{Evolution of the $\ell_2$ error with $T$ for Dynamic Rank Centrality, the MLE and Borda Count method for $n=100$ and $n = 400$. The results are averaged over the grid $\calT$ as well as $60$ Monte Carlo runs.}}
\label{fig:plots_mse}
\end{figure}
\begin{figure}[h!]
\centering
\begin{subfigure}{.5\textwidth}
  \centering
  \includegraphics[width=\linewidth]{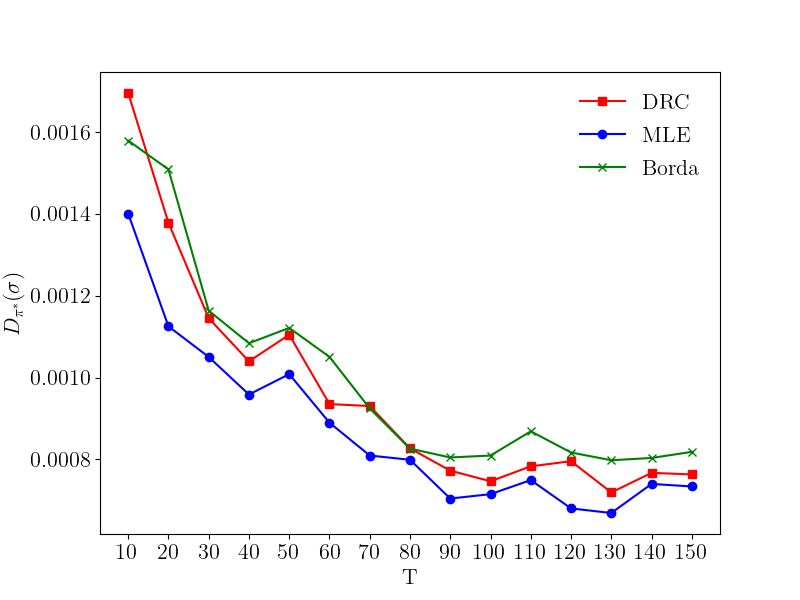}
  \caption{Evolution of the error metric $D_w(\sigma)$ for $n=100$}
  \label{fig:error_metric_N100}
\end{subfigure}%
\begin{subfigure}{.5\textwidth}
  \centering
  \includegraphics[width=\linewidth]{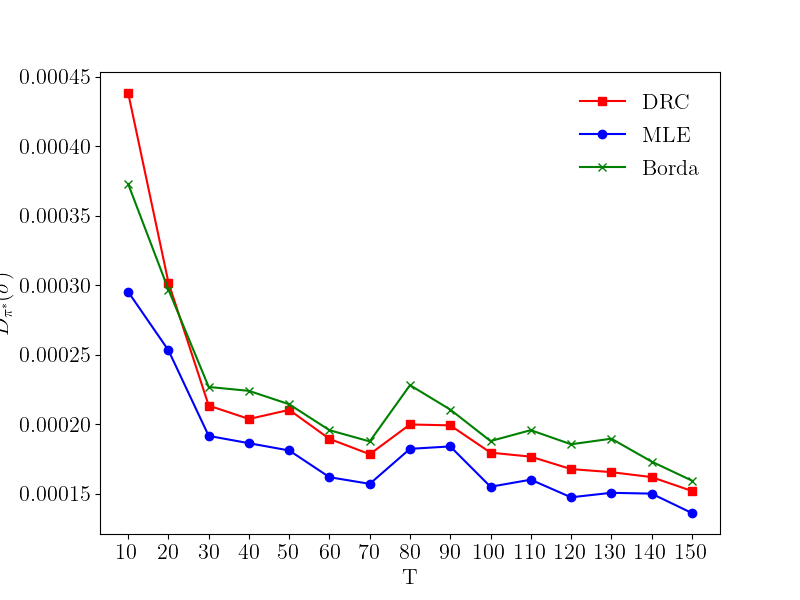}
  \caption{Evolution of the error metric $D_w(\sigma)$ for $n=400$}
  \label{fig:error_metric_N400}
\end{subfigure}%

\caption{\rev{Evolution of the errors $D_w(\sigma)$ with $T$ for Dynamic Rank Centrality, the MLE and Borda Count method for $n=100$ and $n=400$. The results are averaged over the grid $\calT$ as well as $60$ Monte Carlo runs.}}
\label{fig:plots_error_metric}
\end{figure}
\begin{figure}[!h]
\centering
\begin{subfigure}{.5\textwidth}
  \centering
  \includegraphics[width=\linewidth]{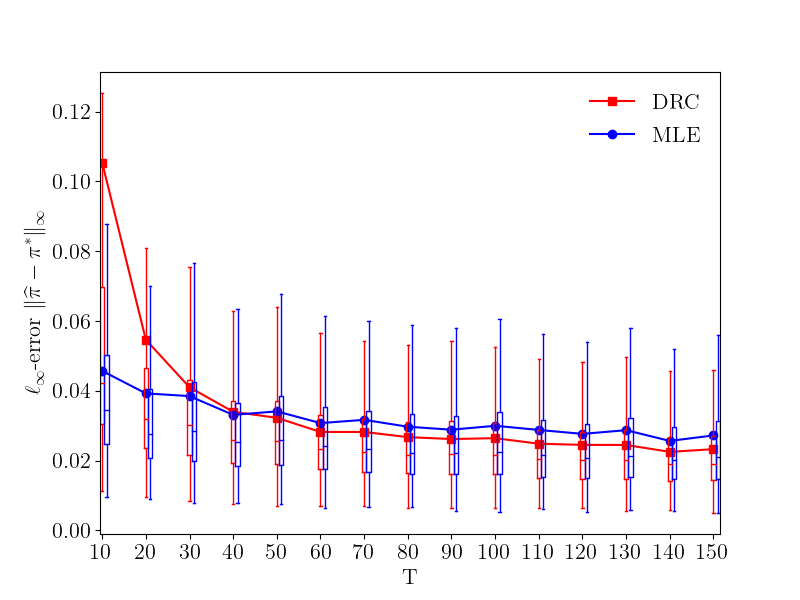}
  \caption{$n = 100$}
  \label{fig:linf_N100}
\end{subfigure}%
\begin{subfigure}{.5\textwidth}
  \centering
  \includegraphics[width=\linewidth]{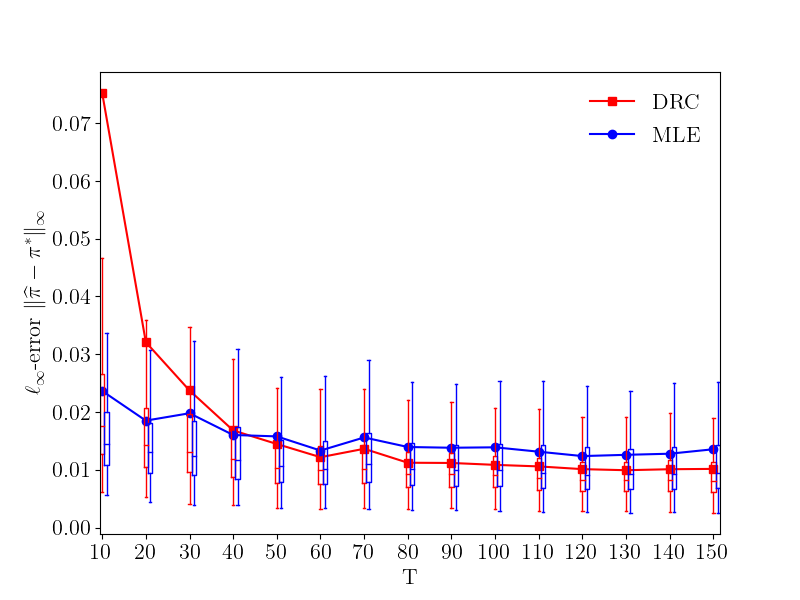}
  \caption{$n = 400$}
  \label{fig:linf_N400}
\end{subfigure}
\caption{\rev{Evolution of the $\ell_\infty$ error with $T$ for Dynamic Rank Centrality and MLE method for $n=100$ and $n=400$. The results are averaged over the grid $\calT$ as well as $60$ Monte Carlo runs.}}
\label{fig:linf}
\end{figure}
\begin{figure}[!h]
\centering
\begin{subfigure}{.5\textwidth}
  \centering
  \includegraphics[width=\linewidth]{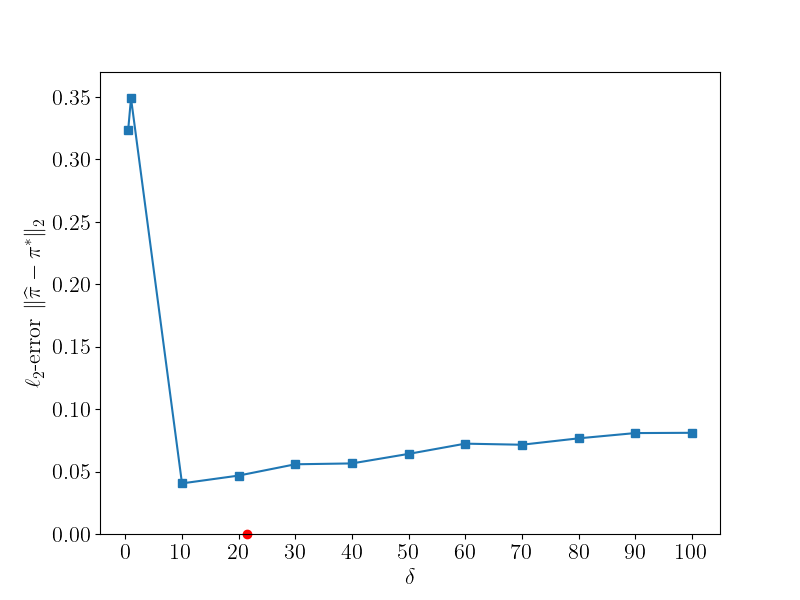}
  \caption{$n = 100$}
  \label{fig:perf_N100}
\end{subfigure}%
\begin{subfigure}{.5\textwidth}
  \centering
  \includegraphics[width=\linewidth]{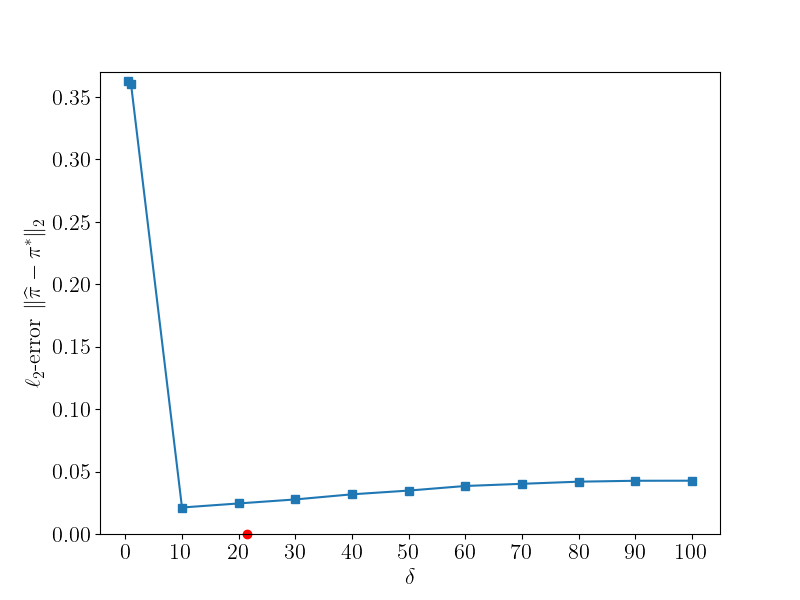}
  \caption{$n = 400$}
  \label{fig:perf_N400}
\end{subfigure}

\caption{\rev{Performance of the Dynamic Rank Centrality method for different value of parameter $\delta$, with $T = 100$, $n=100$ and $n=400$. We highlight in red on the x-axis the theoretical optimal value of parameter $\delta = T^{2/3}$. The results are averaged over the grid $\calT$ as well as $20$ Monte Carlo runs.}}
\label{fig:perf}
\end{figure}

\begin{table}[!h]
\centering
\begin{tabular}{llllll}
\hline
 T   & 10            & 20            & 30            & 40            & 50            \\
\hline
 DRC & 0.48$\pm$0.01 & 0.97$\pm$0.02 & 1.53$\pm$0.02 & 2.13$\pm$0.03 & 2.83$\pm$0.04 \\
 MLE & 2.17$\pm$0.84 & 3.69$\pm$0.52 & 6.06$\pm$0.72 & 9.14$\pm$1.23 & 12.2$\pm$2.42 \\
\hline
\end{tabular}
\hspace{0.5cm}
\begin{tabular}{llllll}
\hline
 T   & 60            & 70             & 80              & 90             & 100            \\
\hline
 DRC & 3.61$\pm$0.03 & 4.44$\pm$0.05  & 5.36$\pm$0.07   & 6.44$\pm$0.07  & 7.38$\pm$0.66  \\
 MLE & 15.67$\pm$3.2 & 21.61$\pm$4.16 & 34.62$\pm$20.68 & 27.35$\pm$3.08 & 39.5$\pm$15.93 \\
\hline
\end{tabular}
\begin{tabular}{llllll}
\hline
 T   & 110             & 120            & 130            & 140            & 150            \\
\hline
 DRC & 8.25$\pm$0.14   & 9.52$\pm$0.14  & 10.38$\pm$0.18 & 11.85$\pm$0.13 & 13.2$\pm$0.19  \\
 MLE & 56.86$\pm$38.95 & 48.69$\pm$9.54 & 50.27$\pm$4.35 & 55.93$\pm$2.62 & 65.99$\pm$5.97 \\
\hline
\end{tabular}

\caption{\rev{Average running time and associated standard deviations (in seconds) of DRC and MLE methods for $n = 100$. Results are averaged over 20 Monte Carlo runs. }}
\label{tab:run_time_N100}
\end{table}

\begin{table}[!h]
\centering
\begin{tabular}{lrrrrr}
\hline
 T   &    10 &    20 &     30 &     40 &     50  \\
\hline
 DRC &  5.41$\pm$ 0.41& 10.21$\pm$0.06 & 16.07$\pm$1.86 & 20.93$\pm$0.18 & 26.56$\pm$ 0.18\\
 MLE & 37.04$\pm$35.7 & 68.92$\pm$16.68 & 130.84$\pm$39.06 & 148.84$\pm$27.88 & 251.99$\pm$175.51 \\
\hline
\end{tabular}
\hspace{0.5cm}
\begin{tabular}{lrrrrr}
\hline
 T   &     60 &    70 &    80 &    90 &    100 \\
\hline
 DRC &  32.58$\pm$0.40 & 39.55$\pm$1.52 & 46.24$\pm$1.33 & 53.25$\pm$ 0.65& 59.58$\pm$2.28 \\
 MLE & 292.74$\pm$93.97 & 337.47$\pm$50.92 & 553.51$\pm$306.46 & 583.90$\pm$297.78 & 551.54$\pm$ 56.71 \\
\hline
\end{tabular}

\caption{\rev{Average running time and associated standard deviations (in seconds) of DRC and MLE methods for $n = 400$. Results are averaged over 15 Monte Carlo runs.}}
\label{tab:run_time_N400}
\end{table}

%
\subsection{Real dataset: NFL data}
We now evaluate our method on a real dataset which consists of the results of National Football League (NFL) games for each season between 2009 and 2015, that are available in the \texttt{nflWAR} package \citep{yurko2019nflwar}. The aim is to recover the ranking of the $n=32$ teams at the end of each season, which contains $T=16$ rounds. The dataset is hence composed of $16$ comparison graphs with $32$ nodes each, and comparison outcomes $(y_{ij}(t))_{i,j,t}$, where $y_{ij}(t) = 1$ if team $j$ beat team $i$ during in round $t$. We fit our model to this data by estimating the underlying strengths (and therefore the ranks) of the teams, at the end of a season. To do so, we tune the parameter $\delta$ using a Leave-One-Out Cross-Validation (LOOCV) procedure, described below.
\begin{enumerate}
    \item Fix a list of potential values of $\delta$ and compute for each of them the associated estimates of the strength $\est{\pi}$ at the end of the season.
    
    \item For every possible values of $\delta$, repeat the following steps several times.
    \begin{itemize}
    \item Select randomly a game during the season, identified by a time $t$ and the pair of compared teams $\set{i,j}$.
        
        \item Consider the dataset obtained by removing the outcome $y_{ij}(t)$ of this game. Use this dataset to compute the estimated strength $\Tilde{\pi}(t)$ at this time $t$.
        
        \item Compute the prediction error $\norm{y_{ij}(t)-\frac{\Tilde{\pi}_j(t)}{\Tilde{\pi}_j(t)+\Tilde{\pi}_i(t)}}_2$.
    \end{itemize}
    We then compute the mean of these prediction errors for each value of $\delta$.
    
    \item Finally, we select $\delta^*$ which has the smallest mean prediction error, and take as our estimate of the strengths at the end of the season the associated vector $\est{\pi}$.
\end{enumerate}
An analogous LOOCV procedure is performed in the MLE approach \citep{bong2020nonparametric} \rev{to tune the bandwidth parameter $h$}.

We then compare our estimated ranking with the ones obtained by the MLE method, the Borda Count method and also with the ELO ratings.  The latter are reputed to be relevant estimations of the teams qualities and are also openly available  \citep{paine2015}. \rev{Using the ELO ratings as ground truth, we assess the performance of the methods by comparing the estimated top 10 teams and by computing the correlation between the ELO rankings and the estimated ranks as well as correlations between the ELO ratings and the estimated strengths.}

\paragraph{\rev{Top 10 teams.}} Figure \ref{fig:tab_nfl} contains the estimation of the top $10$ teams by all the methods, for the seasons $2011$ to $2015$. Considering the ELO ranks as the true ranks, one can observe that the DRC and MLE methods perform similarly, as a majority of the top $10$ teams are recovered for each season. The Borda Count does not perform as well as it did on synthetic data; we can observe that it recovers the same ranks for several teams. This is explained by the fact that in this dataset, each team plays a small number of games, and thus the win rates take a finite (and small) number of values. However, it still recovers a large fraction of the top $10$ teams. 

\begin{figure}[!h]
    \centering
    \begin{subfigure}[b]{0.65\textwidth}
       \includegraphics[width=\linewidth]{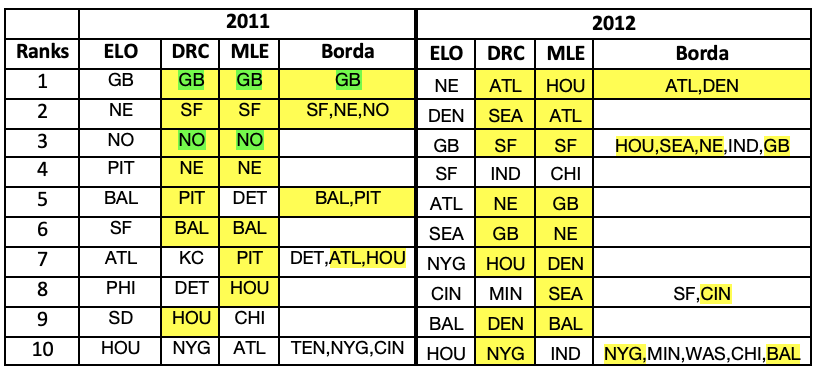}
       \label{fig:nfl_11_12} 
    \end{subfigure}
    
    \begin{subfigure}[b]{0.85\textwidth}
       \includegraphics[width=\linewidth]{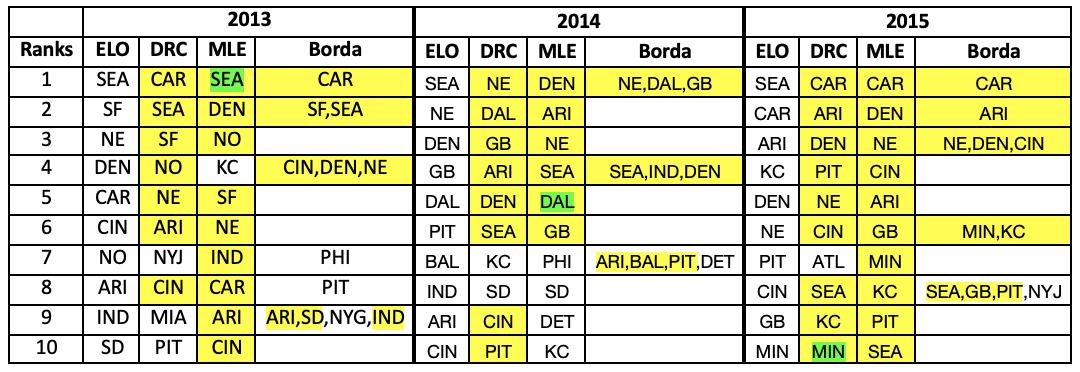}
       \label{fig:nfl_13_15}
    \end{subfigure}
    \caption{The top $10$ teams for seasons $2011$ to $2015$, using ELO ranks, DRC, the MLE, and Borda Count. Teams highlighted in yellow for a particular recovery method are teams appearing in the top $10$ list for ELO rankings.  Teams highlighted in green are recovered at the same rank as in the ELO rankings. 
    }
    \label{fig:tab_nfl}
\end{figure}

\paragraph{\rev{Correlation with ELO rankings.}} \rev{As all of the methods estimate the ranking of the teams, one can also compute the Kendall rank correlation}\footnote{\rev{The Kendall rank correlation \citep{kendall} between two rankings $x,y \in \set{1,\dots,n}$ is defined as}
\[ \rev{\tau := \frac{2}{n(n-1)}\sum_{i<j} sign(x_i-x_j)sign(y_i-y_j).}\]
\begin{itemize}
    \item \rev{If $\tau = 1$, $x$ and $y$ are the same.}
    \item \rev{If $\tau = -1$, $x$ and $y$ are reversed.}
    \item \rev{If $\tau = 0$, the correlation between $x$ and $y$ is no better than the correlation between 2 random rankings.}
\end{itemize} } \rev{between the estimated rankings of each method with the ELO ranking (considered as the true ranks). The results, gathered in Table \ref{tab:corr_ranks_nfl}, show that the methods have similar performance.}

\begin{table}[!h]
    \centering
    \begin{tabular}{lrrrrrrr}
\hline
       &  2011 &   2012 &   2013 &   2014 &   2015 \\
\hline
 DRC   & -0.092 &  0.104 &  0.116 &  0.245 & -0.201 \\
 MLE   & -0.052 &  0.161 &  0.125 &  0.125 & -0.084 \\
 Borda &  0.1   &  0.104 &  0.08  &  0.209 & -0.036 \\
\hline
\end{tabular}
    \caption{\rev{Kendall rank correlations between the ELO ranks and the estimated ranks by DRC, MLE and Borda Count methods at the end of each seasons. All of the methods perform similarly.}}
    \label{tab:corr_ranks_nfl}
\end{table}

\paragraph{\rev{Correlation with ELO ratings.}}\rev{The ELO ratings provide a vector of underlying ground truth strengths for each team. Then, as the DRC and MLE methods estimate the underlying strengths, one can also compute the correlation between the estimated strengths and the ELO ratings (all normalized such that their $\ell_1$ norm is 1). These results are gathered in Table 4 and show that the DRC performs much better than the MLE for all seasons. In particular, it shows that even if the ranks are not perfectly recovered, the estimated strengths by DRC are positively correlated with the ELO ratings.}

\begin{table}[!h]
    \centering
    \begin{tabular}{lrrrrrrr}
\hline
     &    2011 &   2012 &   2013 &   2014 &   2015 \\
\hline
 DRC &    0.425 &  0.518 &  0.284 &  0.478 &  0.414 \\
 MLE &    0.092 & -0.237 & -0.337 & -0.026 &  0.002 \\
\hline
\end{tabular}
    \caption{\rev{Correlations between the normalized ELO ratings and the DRC and MLE estimated strengths at the end of each seasons. Note here that DRC performs much better than the MLE.}}
    \label{tab:corr_nfl}
\end{table}

\section{Discussion} \label{sec:discussion}
We now provide a detailed discussion with closely related work for dynamic ranking, and conclude with future directions for research.
\subsection{Related work}\label{sec:related_work}
As mentioned in Section \ref{sec:intro}, existing theoretical results for the dynamic ranking setup are limited, and the only works we are aware of are the recent results of Bong et al. \citep{bong2020nonparametric} and of Li and Wakin \citep{li2021recovery}. We now discuss these two results in more detail.
%
%
%
%
%
\paragraph{Comparison to Bong et al.\citep{bong2020nonparametric}.} The dynamic BTL model proposed by Bong et al. \citep{bong2020nonparametric} is closely related to the one we presented in Section \ref{sec:prob_setup}. They consider the $\logit$ version of the BTL model 
\begin{equation*}
    \logit\left( \prob(i \text{ beats } j \text{ at time } t)\right) = \beta^*_i(t)-\beta^*_j(t)
\end{equation*}
where $\beta^*(t)$ represents the vector of scores at time $t$ with $w^*(t) = \exp(\beta^*(t))$. The main differences with our model are that (a) the grid $\calT$ can be non-uniform, and (b) the number of comparisons $L_{ij}(t)$ made for each pair $\set{i,j}$ at each time $t \in [0,1]$ can vary. We assumed the grid $\calT$ to be uniform only for simplicity, our analysis can be easily extended to handle the non-uniform setting as long as $\calT$ is ``sufficiently regular''. Similarly, \rev{one could potentially extend our analysis to handle varying number of comparisons $L_{ij}(t)$, although the current proof technique will lead to the presence of the maximum of the $L_{ij}(t)$'s (recall Remark \ref{rem:diff_Lij}).} The assumption $L_{ij}(t) = L$ was only made for simplicity in the proofs as is typically done for the static setting (see for eg., \citep[Remark 3]{chen2020partial}). 

The pairwise comparison data at each time $t' \in \calT$ are gathered into a matrix $X(t')$ where $X_{ij}(t')$ is the number of times $i$ beat $j$ at time $t'$. In order to use the temporal aspect of the data, they smooth the data using a kernel function; this is analogous to (weighted) averaging the data in a suitable neighborhood. More precisely, to estimate the ranks at time $t$, they first compute the smoothed data as
\begin{equation*}
    \Tilde{X}(t) = \sum_{t'=0}^T W_h(t',t)X(t')
\end{equation*}
where $W_h$ is a kernel function with bandwidth $h$. Then, $\beta^*$ is estimated by minimizing the negative log-likelihood, i.e.,
\begin{equation} \label{eq:mle_bong}
    \argmin{\beta: \sum_{i=1}^n \beta_i = 0} \est{R}(\beta;t) \equiv \argmin{\beta: \sum_{i=1}^n \beta_i = 0} \left( \sum_{i,j:i \neq j} \Tilde{X}_{ij}(t)\log \left(1+\exp(\beta_j(t)-\beta_i(t)) \right) \right)
\end{equation}
%
using a proximal gradient descent algorithm. If the matrix $|\Tilde{X}(t)|$ is considered to be the adjacency matrix of a weighted directed graph, then the strong connectivity of this graph is sufficient for the unique existence of the solution of \eqref{eq:mle_bong}. 
%
%
%
The main assumptions needed for their theoretical analysis are the following. 
\begin{itemize}
    \item The probabilities $\prob(i \text{ beats } j \text{ at time } t)$ are Lispchitz functions of time $t \in [0,1]$ for all $i \neq j \in [n]$ (same as Assumption \ref{assump:smooth1}).
    
    
    \item Each pair of teams $\set{i,j}$ has been compared at least at one time point $t'\in\calT$. Translated to our notation, this means that the union graph $\cup_{t' \in \calT} G_{t'}$ is complete.
\end{itemize}

We remark that this last assumption is a stronger assumption than the connectivity assumption on the union graph $\Gu$ that we required in our analysis.

Bong et al. derive bounds  in the $\ell_{\infty}$ norm for estimating $\beta^*$. For a bandwidth $h \gtrsim \left(\frac{\log n}{T}\right)^{\frac{1}{3}}$, denoting $\est{\beta}(t)$ to be the \rev{``MLE''} estimator\footnote{\rev{Note that \eqref{eq:mle_bong} is not the MLE for the dynamic BTL model with Lipschitz evolving win/loss probabilities (Assumption \ref{assump:smooth1}). It is in fact the MLE for the static BTL model, applied to the kernel-smoothed observations.}} (i.e., the solution of \eqref{eq:mle_bong}), it is shown \citep[Theorem 5.2]{bong2020nonparametric} that with high probability,
\begin{equation*}
    \norm{\est{\beta}(t)-\beta^*(t)}_\infty \lesssim \delta_h(t) + \left(\frac{\log n}{nT}\right)^{\frac{1}{3}}
\end{equation*}
where $\delta_h(t)$ is denoted to be a discrepancy parameter in \citep{bong2020nonparametric}, and is small when all the teams play the same number of games at all times. However, the dependence of $\delta_h$ on $h$ is not made explicit. Note that they also recover the $T^{-1/3}$ rate for pointwise estimation of Lipschitz functions.  
Additionally, they also derive a rate for the uniform error. For $h \gtrsim \left(\frac{\log (nT^3)}{T}\right)^{\frac{1}{3}}$, it is shown  \citep[Theorem 5.3]{bong2020nonparametric} that with high probability, 
\begin{equation*}
    \sup_{t\in[0,1]} \norm{\est{\beta}(t)-\beta(t)}_\infty \lesssim \sup_{t\in[0,1]}\delta_h(t) + \left(\frac{\log (nT)}{T}\right)^{\frac{1}{3}}.
\end{equation*}
While we provide pointwise estimation error bounds for any given $t \in [0,1]$, it is also possible to extend our results to obtain error bounds holding uniformly over all $t \in [0,1]$. The first main idea here would be to observe that there are $O(T)$  different number of neighborhoods in the construction of $\Pest$ in \eqref{eq:spec_mat_pest}, which implies that there are $O(T)$ different possible values of $\est{\pi}(t)$ over all $t \in [0,1]$. Secondly, one can verify that $\pistar$ is a Lipschitz function of $t$. These two facts together with a union bound argument can be used to establish $\ell_2$ and $\ell_{\infty}$ bounds holding uniformly over $t \in [0,1]$, with probability at least $1 - O(T n^{-c})$ where $c$ is a suitably large constant. In our analysis, we had taken $c$ to be $10$ (resp. $9$) for the $\ell_2$ (resp. $\ell_{\infty}$) analysis, but it can be any suitably large value, at the expense of worsening the other constants in the accompanying theorems in Section \ref{sec:results}.   
%

\paragraph{Comparison to Li and Wakin \citep{li2021recovery}.} The model introduced by Li and Wakin \citep{li2021recovery} aims at recovering a pairwise comparison matrix $X(T)$ at a time $T$ from noisy linear measurements of the matrices $X(t)$ for $t \in [T]$. For each pair of items $\set{i,j}$, $X_{ij}(t)$ quantifies the preference of  item $j$ over  item $i$ at time $t$ with $X_{ij}(t) > 0$ if item $j$ is preferred to  item $i$. One can for example consider $X_{ij}(t)$ to be the goal difference during a football game between the teams $j$ and $i$ at time $t$. Note that these data contain more information than the simple binary outcomes we use in the dynamic BTL model. A specificity of this model is that the preferences over the collection of items are supposed to depend on $2r$ latent factors, captured in the matrices $S_t\in\matR^{n \times r}$ and $Q\in\matR^{n \times r}$. The model considered in \citep{li2021recovery} assumes that  
\begin{equation*}
     X(t) = S_t Q^\top - QS_t^{\top} \quad \forall t \in [T].
\end{equation*}
The matrix $Q$ contains information on factors that do not depend on time. The time-dependent data are contained in the matrix $S_t$ and evolve under the following generative model.
\begin{equation*}
    S_t = S_{t-1} + E_t \quad \forall t \in [T]
\end{equation*}
where $E_t$ is an innovation matrix with i.i.d. centered Gaussian entries. The intuition behind this model comes from the case of a single factor ($r=1$) with $Q \in \matR^n$ being the all ones vector. In this case, for each pair of items $\set{i,j}$,
\begin{equation*}
    X_{ij}(t) = s_{t,i}-s_{t,j} \quad \forall t\in [T].
\end{equation*}
The outcomes of the comparisons then, as in the BTL, only depend on the strength of each item at time $t$. Then by recovering the matrix $X(T)$, one can subsequently also derive a ranking of the items (see for eg., \citep{ranksync21}). 

Denoting $M$ to be the number of measurements available at each time $t$, they show a bound on the estimation error $\norm{X(T)-\est{X}(T)}_F$ that holds with high probability and decreases with $M$, where $\est{X}(T)$ is computed as the solution of an optimization problem. It is however unclear how one can then derive a ranking of the items from this estimated comparison matrix $\est{X}(T)$ for $r>1$.

\subsection{Future work}
For future work it would be interesting to extend our study to other models. Indeed, the BTL model is classic for the ranking problem but suffers from some limitations. 
\begin{enumerate}
    \item \textit{Independence.} One of the main assumptions of the dynamic BTL model we consider is the independence of the outcomes for all comparisons at a given time $t$, and the independence across different time points. The latter assumption can in particular be questioned, as the choices of one user across time  are typically not going to be independent. It would be interesting to model these dependencies across time, for instance, by modelling them as a Markov process.
    
    \item \textit{Home effect.} Another model, used in the study of sports tournaments by Cattelan et al. \citep{cattelan2013dynamic} adapts the BTL to include a home effect, which is beneficial to the hosting team. Indeed, because of the familiar environment and the supporting public, a team is more likely to win if they play in their stadium. It would be interesting to theoretically analyze such a model. 
    
    
    \item \textit{Non-parametric dynamic models.} Considering a parametric model as the BTL model has limitations as it assumes that the preference outcomes depend only on one parameter $w^*$, which can be seen as a strong transitivity constraint \citep{shah2016stochastically}. That is why one can instead use non-parametric models as described in \citep{shah2016stochastically,shah2017simple,pananjady2017worst}, where only a certain transitivity assumption is made on the comparison matrix. In particular, the BTL model belongs to this class of models. It has been shown by Shah et al. \citep{shah2016stochastically} that in the static case, the matrix of probabilities can be estimated at the same rate as in a parametric model. As such extensions can be considered in the static setting, it would be interesting to adapt these models to the dynamic case and study their performance.
\end{enumerate}

\bibliographystyle{plain}
\bibliography{references}

\newpage
\appendix
\section{Summary of notation} \label{sec:summary_notation}
 
\begin{center}
\begin{table}[!ht]
\begin{tabular}{ | m{4cm}| m{10cm} | } 
\hline
\textit{Symbol} & \textit{Definition} \\
\hline\hline
 $\calT$ & Uniform grid of $[0,1]$, with size $T+1$.  \\ 
\hline

$G_{t'} = ([n], \calE_{t'})$ & Undirected comparison graph at time $t' \in \calT$ \\

\hline
$w_t^* \in \matR^n$   & Ground-truth strengths at time $t \in [0,1]$ \\

\hline

$y^*_{ij}(t)$ and $y_{ij}(t)$ & (resp. population and empirical) Fraction of times $j$ beats $i$ at time $t$ \\

\hline

$M$ & Lipschitz constant of $y^*_{ij}(t)$ \\

\hline 

$L$ & Number of independent comparisons made for each $\set{i,j} \in \calE_{t'}$ \\
\hline 

$\Nt$ & Neighborhood of size $\delta$, see \eqref{eq:delta_neighbor} \\

\hline 

$\Nij$ & $\set{t' \in \Nt | \, \set{i,j}\in\calE_{t'}}$ \\

\hline 

$\Gu = ([n], \Eu)$ & ``Union graph'' where $\Eu = \set{\set{i,j} \, : \, i\neq j, \, \cardNij \geq 1}$ \\
\hline

$\Nmax$ and $\Nmin$ & $\Nmax = \max_{\set{i,j}\in\Eu} \cardNij \quad \text{and} \quad \Nmin = \min_{\set{i,j}\in\Eu} \cardNij$\\
\hline 

$\Pest, \Pbar \in \matR^{n \times n}$ & Empirical and population transition matrices resp., see \eqref{eq:spec_mat_pest} and \eqref{eq:spec_mat_pbar} \\
\hline

$\piest, \pistar \in \matR^n$ & Stationary distributions of $\Pest$ and $\Pbar$ resp. \\ 

\hline 

$\dumax, \dumin$ & Maximum (resp. minimum) degree of $\Gu$ \\

\hline 

$\du$ $(\geq \dumax)$ & Normalization term in \eqref{eq:spec_mat_pest} \\
\hline

$b(t)$ & $\max_{i,j \in [n]} \frac{w^*_{t,i}}{w^*_{t,j}} \mbox{ for all } t \in [0,1]$ \\
\hline

$\bmax$ & $\max_{t'\in\Nt} b(t')$ \\

\hline

$\Ld$ & Random walk Laplacian of $\Gu$ \\

\hline

$\xid$ & $\xid = 1-\lammax(\Ld)$ where 
$\lammax(\Ld)$ is the second largest eigenvalue (in absolute value) of $\Ld$ \\

\hline

$G_{t'} = \calG(n,p(t'))$ & Erd\"os-Renyi graph at $t' \in\calT$ with parameter $p(t') \in [0,1]$ \\
\hline 

$\pu$ & $1-\prod_{t' \in \Nt} (1-p(t'))$, see \eqref{eq:pdel_def} \\

\hline 

$\psum$ & $\sum_{t' \in \Nt} p(t')$ \\

\hline

$\pmin, \pmax$ & $\pmin := \min_{t' \in \calT} p(t')$ and $\pmax := \max_{t' \in \calT} p(t')$ \\

\hline
\end{tabular}
\caption{Summary of symbols used throughout the paper along with their definitions.}
\end{table}
\end{center}

\section{Technical tools}\label{appendix:A}
We collect here some technical results that are used for proving our main results. We begin by recalling the following result from \citep{Chen_2019}.
\begin{theorem}[{\citep[Theorem 8]{Chen_2019}}]\label{thm:thm8}
Suppose that $P, \est{P}, P^*$ are probability transition matrices with stationary distributions $\pi, \est{\pi}, \pi^*$ respectively. We assume that $P^*$ represents a reversible Markov chain. When $\norm{P-\est{P}}_{\pi^\star} < 1-\max\{\lambda_2(P^\star),-\lambda_n(P^\star)\}$ holds, then we have 
\begin{equation}
    \norm{\pi - \est{\pi}}_{\pi^\star} \leq \frac{\norm{\pi^\top(P - \est{P})}_{\pi^\star}}{1 - \max\{\lambda_2(P^\star), - \lambda_n(P^\star)\}-\norm{P^\star - \est{P}}_{\pi^\star}}.
\end{equation}
\end{theorem}
Next, we recall some useful concentration results from probability starting with the classical Chernoff bound.
\begin{theorem}[\textbf{Chernoff bound}, {\citep[Theorems 4.4, 4.5]{mitzenmacher2017probability}}]\label{thm:chernoff}
\rev{Let $X_1,\dots,X_n$ be independent Bernoulli variables with $\prob(X_i=1) = p_i$. Let $X = \sum_{i=1}^n X_i$.
\begin{enumerate}
    \item For any $\mu \geq \expec[X]$, the following is true. \begin{itemize}
        \item For $\delta \in (0,1]$, $\prob\left(X \geq (1+\delta)\mu\right) \leq e^{-\delta^2\mu/3}$. 
        \item For $\delta \geq 1$, $\prob\left(X \geq (1+\delta)\mu\right) \leq e^{-\delta\mu/3}$. 
    \end{itemize}
    \item For any $\mu \leq \expec[X]$ and $\delta \in (0,1)$, $\prob\left(X \leq (1-\delta)\mu\right) \leq e^{-\delta^2\mu/2}.$
\end{enumerate}}
\end{theorem}
\begin{theorem}[\textbf{Hoeffding's inequality}] \label{thm:hoeffding}
Let $X_1,\dots,X_n$ be a sequence of independent random variables where $X_i \in [a_i,b_i]$ for each $i \in [n]$, and $S_n = \sum_{i=1}^n X_i$. Then
\begin{equation*}
    \prob\left(\abs{S_n-\expec[S_n]} \geq t \right) \leq 2\exp\left(-\frac{2t^2}{\sum_{i=1}^n(b_i-a_i)^2}\right).
\end{equation*}
\end{theorem}
\begin{theorem}[\textbf{Matrix Bernstein}, \citep{Tropp15}] \label{thm:bernstein}
Let $Z_1,\dots,Z_n \in \mathbb{R}^{d_1 \times d_2}$ be independent, zero-mean random matrices, each satisfying (almost surely) $\norm{Z_i}_2 \leq B$. Then for any $s\geq 0$, 
\begin{equation*}
    \prob\left(\left\|\sum_{i=1}^n Z_i\right\|_2\geq s\right) \leq (d_1+d_2)\exp\left(-\frac{3s^2}{6\nu +2Bs}\right)
\end{equation*}
where $\nu = \max\set{\norm{\expec[\sum_{i=1}^n Z_i^\top Z_i]}_2, \norm{\expec[\sum_{i=1}^n Z_i Z_i^\top]}_2}$.
\end{theorem}

\begin{proposition}\label{prop:useful_ineg}
\rev{The following is true.
\begin{enumerate}
    \item $1-\frac{x}{2} \geq e^{-x}$ for $x\in[0,1.59]$
    \item $1-x \leq e^{-x}$ for $x\in[-1,1]$.
\end{enumerate}}
\end{proposition}

%
\section{Properties of Erdös-Renyi graphs}
Recall that in our setup, we have $T+1$ Erdös-Renyi graphs $G_{t'} \sim \calG(n,p(t'))$ for $t' \in \calT$. Moreover, for any given $t \in [0,1]$, we have the union graph $\Gu \sim \calG(n,\pu))$ with $\pu$ as in \eqref{eq:pdel_def}. Also recall that $\psum := \sum_{t'\in\Nt}p(t')$.
%
%
\begin{lemma}\label{lem:prop_gnp}
Consider the events
\begin{enumerate}
    \item $\calA_1 = \set{\frac{n\pu}{2}\leq \dumin\leq \dumax\leq \frac{3n\pu}{2}}$,
    \item $\calA_2 = \set{\xid > \frac{1}{2}}$,
    \item $\calA_3 = \set{\cardEu \leq 2n^2\pu}$,
    \item $\calA_4 = \set{\frac{\psum}{2}\leq \Nmin\leq \Nmax \leq 2\psum}$.
\end{enumerate}
Then the following is true.
\begin{enumerate}
    \item There exists a constant $c_0 \geq 1$ such that if $\pu \geq c_0 \frac{\log n}{n}$, then $\prob(\calA_i) \geq 1-O(n^{-10})$ for $i=1,2,3$.
    
    \item There exists a constant $c_2 \geq 1$ such that if $\psum \geq c_2 \log n$, then $\prob(\calA_4) \geq 1-O(n^{-10})$. 
\end{enumerate}
\end{lemma}
\begin{proof}
\begin{enumerate}
    \item This follows for $\calA_1$ from the Chernoff bound (see Appendix \ref{appendix:A})
    with $\varepsilon = 1/2$, and the union bound. For $\calA_2$, the statement follows directly by applying \citep[Lemma 7]{negahban2015rank} to $\Gu$. For $\calA_3$, this is again a standard use of the Chernoff bound with $\mu = {n \choose 2} \pu$ and $\varepsilon = 1/2$.
    
    \item Note that for any $i \neq j \in [n]$, $\cardNij$ is a sum of independent Bernoulli random variables and $\expec[\cardNij] = \sum_{t'\in \Nt}p(t') = \psum$. Then if $\psum \geq c_2 \log n$ for a suitably large constant $c_2 > 0$, it holds with probability at least $1-O(n^{-12})$ that
    \begin{equation*}
        \frac{\psum}{2}\leq  \cardNij\leq 2\psum.
    \end{equation*}
    %
    Now the union bound implies that
    \begin{equation*}
        \prob\left(\forall i \neq j: \  \frac{\psum}{2}\leq \cardNij \leq 2\psum\right) \geq 1-O(n^{-10}).
    \end{equation*}
    %
    %
    %
  Since $\Nmin \geq \min_{i \neq j} \cardNij$ and $\Nmax \leq \max_{i \neq j} \cardNij$ is always true, hence the statement follows.
\end{enumerate}
\end{proof}
%
%
\begin{proposition}\label{prop:pu}
    With $\pmin$ as in \eqref{eq:pmin_def}, we have for all $t \in [0,1]$ that  $\pu\gtrsim\min\set{1,\pmin\delta}$.
Therefore $\pu\gtrsim \frac{\log n}{n}$ if $\pmin\delta \gtrsim \frac{\log n}{n}$.
\end{proposition}
\begin{proof}
Starting with the definition of $\pu$ in \eqref{eq:pdel_def}, we have \rev{using Proposition \ref{prop:useful_ineg} that}
\begin{equation}
    \pu = 1-\prod_{t'\in\Nt}(1-p(t')) \geq 1-\left(1-\pmin\right)^{\cardN} \geq 1-e^{-\pmin\cardN} \geq 1-e^{-\pmin\delta}
\end{equation}
where the last inequality follows uses the fact  $\cardN \geq \delta$ for all $t$. 
Now there are two cases to distinguish.
\begin{itemize}
    \item If $\pmin\delta\leq 1$, then $1-e^{-\pmin\delta} \geq \frac{\pmin\delta}{2}$ using \rev{Proposition \ref{prop:useful_ineg}.}
    
    \item If $\pmin\delta\geq 1$, then $1-e^{-\pmin\delta}\geq 1-e^{-1}$. 
\end{itemize}
This concludes the proof.
\end{proof}

%
\section{Proofs of results in Section \ref{sec:analysis_l2}} \label{appsec:proofs_l2_analysis}
\subsection{Proof of Lemma \ref{lem:bounds_deterministic}} \label{appsubsec:proof_lem_bounds_deter}
Recall that the entries of $\Delta_1(t)$ are given by
\begin{equation*}
    \Deloneij = \left\{\begin{array}{cc}
        \displaystyle\frac{1}{\du\cardNij}\sum_{t' \in \Nij} \left(y_{ij}^*(t')-y_{ij}^*(t)\right) & \mbox{if} \quad \set{i,j} \in \Eu,  \\
        \displaystyle - \sum_{k \neq i} \Deloneik & \mbox{if} \quad i=j, \\
        0 & \mbox{otherwise.}
    \end{array}\right.
\end{equation*}
Let us denote $D_1(t)$ to be the diagonal matrix containing the elements $\Deloneii$ and $\Dbarone = \Delta_1(t)-D_1(t)$. As $D_1(t)$ is diagonal, we then have 
\begin{equation*}
    \norm{\Delta_1(t)}_2 \leq \norm{D_1(t)}_2 + \norm{\Dbarone}_2 \leq \max_i \abs{\Deloneii}+ \norm{\Dbarone}_F.
\end{equation*}
Let us first bound $\norm{\Dbarone}_F$. Assumption \ref{assump:smooth1} implies that for any $\set{i,j} \in \Eu$,
\begin{align*}
    \abs{\Dbaroneij}\leq  \frac{1}{\du\cardNij}\displaystyle\sum_{t' \in \Nij} \abs{y_{ij}^*(t')-y_{ij}^*(t)}
    \leq  \frac{1}{\du\cardNij}\displaystyle\sum_{t' \in \Nij}M\abs{t'-t} 
    \leq  \frac{M\delta}{T\du},
\end{align*}
which in turn implies $\norm{\Dbarone}_F \leq 2\frac{M\cardEu\delta}{T\du}$.
%
%
In order to bound $\norm{D_1(t)}_2$, we simply note that
\begin{align*}
    \abs{D_{1,ii}(t)} = \abs{ -\sum_{j\neq i} \Delta_{1,ij}(t)} \leq \dumax \max_{j \neq i} \abs{\Delta_{1,ij}(t)}\leq \dumax \frac{M\delta}{T\du}\leq \frac{M\delta}{T}
\end{align*}
since $\du \geq \dumax$. Hence $\norm{D_1(t)}_2 \leq \frac{M \delta}{T}$, and so
\begin{equation*}
    \norm{\Delta_1(t)}_2 \leq \frac{M\delta}{T}\left(1+2\frac{\cardEu}{\du}\right)\leq \frac{M\delta}{T}\left(1+2\frac{\cardEu}{\dumax}\right) \leq 4\frac{M\delta\cardEu}{T\dumax} \quad \left(\text{ since } 1 \leq 2\frac{\cardEu}{\dumax} \right).
\end{equation*}
\begin{flushright}
$\square$
\end{flushright}
\subsection{Proof of Lemma \ref{lem:bound_delta}} \label{appsubsec:proof_bound_delta}
Our goal is to bound $\norm{\Delta(t)}_2$ for the random matrix $\Delta(t) = \Pest-\Pbar$, whose entries are 
\begin{equation*}
    \Delta_{ij}(t) = \left\{\begin{array}{cc}
    \displaystyle\frac{1}{\du\cardNij}\sum_{t' \in \Nij}\left(y_{ij}(t')-y_{ij}^*(t')\right) & \mbox{if} \quad \set{i,j}\in\Eu \\
    - \displaystyle\sum_{k \neq i} \Delta_{ik}(t) & \mbox{if} \quad i=j\\
    0 & \mbox{otherwise.}
\end{array}\right.
\end{equation*}
Once again $\Delta(t)$ can be decomposed as $\Delta(t) = D(t) + \Dbar$ where $D(t)$ is a diagonal matrix  and $\Dbar$ contains its off-diagonal coefficients. Since 
$$\norm{\Delta(t)}_2 \leq \norm{D(t)}_2 + \norm{\Dbar}_2,$$
we will bound $\norm{D(t)}_2$ and $\norm{\Dbar}_2$ using ideas from the proof of \citep[Lemma 3]{negahban2015rank}.
\subsubsection{Bound on $\norm{D(t)}_2$}
The matrix $D(t)$ is diagonal, hence $\norm{D(t)}_2 = \max_{i \in [n]} \abs{D_{ii}(t)}$ where the diagonal entries are given by
\begin{equation*}
    D_{ii}(t) = -\sum_{k \neq i} \sum_{t' \in \Nik} \sum_{l=1}^L \frac{1}{L\cardNik \du}(y_{ik}^{(l)}(t') - y_{ik}^*(t')).
\end{equation*}
This is a sum of at most $L\dumax\Nmax$ independent and centered variables, which  almost surely lie in the interval $\left[ -\frac{1}{L\du\Nmin },\frac{1}{L\du\Nmin}\right]$. Thus, by applying Hoeffding's inequality (see Theorem \ref{thm:hoeffding} in Appendix \ref{appendix:A}) we obtain for any $s > 0$
\begin{align*}
    \prob\left(\abs{D_{ii}(t)}>s\right)  \leq 2 \exp\left(-\frac{2s^2 L^2\dusq\Nminsq}{4L\dumax \Nmax}\right) 
     = 2\exp\left(-\frac{L\dusq\Nminsq s^2}{2\Nmax\dumax}\right)
\end{align*}
which implies via a union bound that 
\begin{equation*}
    \prob(\norm{D(t)}_2 >s) \leq 2n\exp\left(-\frac{L\dusq \Nminsq s^2}{2\Nmax\dumax}\right).
\end{equation*}
Choosing $s = c_1\sqrt{\frac{\Nmax \dumax\log n}{L\dusq\Nminsq}}$ where $c_1\geq 2$ is a constant, we have that
\begin{equation*}
    \norm{D(t)}_2 \leq c_1\sqrt{\frac{\Nmax \dumax\log n}{L\dusq\Nminsq}}
\end{equation*}
with probability at least $1-2n^{1-\frac{c_1^2}{2}}$.

%
\subsubsection{Bound on $\norm{\Dbar}_2$ when $\dumax\Nmax \leq \log n$}
In order to bound $\norm{\Dbar}_2$, we first recall the standard inequality $\norm{M}_2 \leq \sqrt{\norm{M}_1\norm{M}_\infty}$, and also $\norm{M}_1 = \norm{M^\top}_\infty$. 
Since $\Dbar$ is skew-symmetric, it follows that $$ \norm{\Dbar}_2 \leq \sqrt{\norm{\Dbar}_1\norm{\Dbar^\top}_1} \leq \norm{\Dbar}_1.$$ 
Our aim now is to bound $\norm{\Dbar}_1$.
Recall that for every $i \neq j$,
\begin{equation*}
        \Dbarij = \left\{\begin{array}{cc}
        \displaystyle\frac{1}{\du\cardNij}\sum_{t' \in \Nij}\left(y_{ij}(t')-y_{ij}^*(t')\right) &  \mbox{if} \quad \set{i,j}\in\Eu\\
        0 & \mbox{otherwise.}
    \end{array}\right.
\end{equation*}
Now denoting for all $i$,
\begin{equation*}
    R_i(t) = \sum_{j \neq i}\sum_{t' \in \Nij}\abs{\sum_{l=1}^L \frac{1}{L\du \cardNij}y_{ij}^{(l)}(t') - y_{ij}^*(t')},
\end{equation*}
clearly $\norm{\overline{\Delta}}_2 \leq \max_{i\in [n]} \sum_{j=1}^n \abs{\Dbarij} \leq \max_{i\in [n]} R_i(t)$ (since $\Dbarii = 0$). In order to bound $R_i(t)$ for any given $i \in [n]$, let us first denote 
$$\calS_i = \set{(\xi^{(i)}_{jt'})_{\set{i,j} \in \calE_{t'},t' \in \Nij} \, | \  \xi^{(i)}_{jt'} \in \set{-1,1}}.$$ 
Since the random variables 
$\frac{\xi^{(i)}_{jt'}}{\cardNij}\left(y_{ij}^{(l)}(t') - y_{ij}^*(t')\right)$ lie almost surely in the  interval $[-\frac{1}{\Nmin},\frac{1}{\Nmin}]$, we obtain via a simple union bound argument, along with Hoeffding's inequality, that
%
%
%
\begin{align*}
    \prob(R_i(t) > s) & \leq \sum_{\xi^{(i)} \in \calS_i} \prob\left(\sum_{j\neq i}\sum_{t' \in \Nij} \xi^{(i)}_{jt'} \sum_{l=1}^L
    \frac{1}{L\du \cardNij}(y_{ij}^{(l)}(t') - y_{ij}^*(t')) > s\right) \\
    & \leq \sum_{\xi^{(i)} \in \calS_i}\prob\left(\sum_{j \neq i}\sum_{t' \in \Nij} \xi^{(i)}_{jt'} \sum_{l=1}^L \frac{1}{\cardNij}(y_{ij}^{(l)}(t') - y_{ij}^*(t')) > L\du s\right) \\
    & \leq \sum_{\xi^{(i)} \in \calS_i}\exp\left( -\frac{L^2\dusq \Nminsq s^2}{2L\dumax \Nmax}\right) \\
    & \leq \sum_{\xi^{(i)}  \in \calS_i}\exp\left( -\frac{L\dusq\Nminsq s^2}{2\dumax\Nmax}\right).
\end{align*}
%
%
%
Since $\abs{\calS_i} \leq 2^{\dumax\Nmax}$ for each $i \in [n]$, we obtain
\begin{align*}
    \prob(R_i(t) > s) \leq 2^{\dumax\Nmax}\exp\left( -\frac{L\dusq\Nminsq s^2}{2\dumax\Nmax}\right) = \exp\left( \du\Nmax \ln 2-\frac{L\dusq\Nminsq s^2}{2\dumax\Nmax}\right).
\end{align*}
Applying a union bound over $[n]$ now leads to
\begin{equation*}
    \prob(\norm{\Dbar}_2 > s) \leq n\exp\left( \dumax\Nmax \ln 2-\frac{L\dusq\Nminsq s^2}{2\dumax\Nmax}\right).
\end{equation*}
%
%
%
%
Finally, since $\dumax\Nmax \leq \log n$ by assumption we obtain for any constant $c_2 > 2$ that
\begin{equation*}
    \prob\left(\norm{\Dbar}_2 > c_2\sqrt{\frac{\Nmax\dumax\log n}{L\dusq\Nminsq}}\right)\leq n^{2-\frac{c_2^2}{2}}.
\end{equation*}
\subsubsection{Bound on $\norm{\Dbar}_2$ when $\dumax\Nmax \geq \log n$}
In this case, we will use Bernstein's inequality for random matrices (see Theorem \ref{thm:bernstein} in Appendix \ref{appendix:A}). 
For each $i < j \in [n]$, let $Z^{ij,l}(t')$ be a $n \times n$ matrix with all entries equal to zero except for $Z^{ij,l}_{ij}(t')$ and $Z^{ij,l}_{ji}(t')$ defined as follows.
\begin{align*}
    Z^{ij,l}_{ij}(t') = \frac{C_{ij}^{(l)}(t')}{L\du\cardNij} \quad \text{and} \quad 
    Z^{ij,l}_{ji}(t') = -\frac{C_{ij}^{(l)}(t')}{L\du\cardNij}
\end{align*}
where $C_{ij}^{(l)}(t') = y_{ij}^{(l)}(t')-y_{ij}^*(t')$ if $\set{i,j}\in\Eu$, which is a centered Bernoulli variable, and $C_{ij}^{(l)}(t')=0$ otherwise. Hence, $\expec[Z^{ij,l}(t')]=0$. Now we can write 
$$\Dbar = \sum\limits_{\substack{i<j \\ \set{i,j}\in\Eu}}{\sum_{t' \in \Nij} \sum_{l=1}^L Z^{ij,l}(t')}.$$
To apply Bernstein's inequality, we first note that $\norm{Z^{ij,l}(t')}_2$ is uniformly bounded for all $t',i,j$.
\begin{align*}
    \norm{Z^{ij,l}(t')}_2 \leq \norm{Z^{ij,l}(t')}_F = \frac{\sqrt{2}\abs{C^{(l)}_{ij}(t')}}{L\du\cardNij} \leq \frac{\sqrt{2}}{L\du\Nmin} =: B.
\end{align*}
Since $Z^{ij,l}(t')$ are skew-symmmetric and independent matrices, we also have that
$$\nu = \norm{\expec[\Dbar \Dbar^\top]}_2 =  \left\|\sum\limits_{\substack{i<j \\ \set{i,j}\in\Eu}}{\sum_{t' \in \Nij} \sum_{l=1}^L -\expec [Z^{ij,l}(t')^2]}\right\|_2.$$
The matrix $(Z^{ij,l}(t'))^2$ only has two non-zero entries, which are at the locations  $(i,i)$ and $(j,j)$, and they are equal to 
$-\frac{C_{ij}^{(l)}(t')^2}{L^2\dusq\cardNij^2}$. Then its expectation (which also has only two non zero entries) satisfies
\begin{align*}
    -\expec[Z^{ij,l}(t')^2]_{i,i} & = -\expec[Z^{ij,l}(t')^2]_{j,j} 
     = \frac{\expec[C_{ij}^{(l)}(t')^2]}{L^2\dusq\cardNij^2} 
     \leq \frac{Var(C_{ij}^{(l)}(t'))}{L^2\dusq\Nminsq} 
     \leq \frac{1}{4L^2\dusq\Nminsq}.
\end{align*}
since $Var(C_{ij}^{(l)}(t')) = y_{ij}^*(t')(1-y_{ij}^*(t')) \leq 1/4$. 
Furthermore, $\expec[\Dbar \Dbar^\top]$ is a  diagonal matrix with positive entries, hence
\begin{equation*}
    \nu = \max_{k \in [n]} \left\lvert \left(\sum\limits_{\substack{i < j \\ \set{i,j} \in \Eu}}{\sum_{t'\in\Nij}\sum_{l=1}^L -\expec [Z^{ij}(t')^2]}\right)_{kk} \right\rvert \leq \frac{L\dumax\Nmax}{4L^2\dusq\Nminsq} = \frac{\Nmax\dumax}{4L\dusq\Nminsq}.
\end{equation*}
Finally applying Bernstein's inequality, it holds for any $s > 0$ that
\begin{align*}
    \prob(\norm{\Dbar}_2 >s) & \leq 2n\exp\left( -\frac{3s^2}{6\nu + 2Bs}\right) \\
    & \leq 2n\exp\left( -\frac{3s^2}{\frac{6\Nmax\dumax}{4L\dusq\Nminsq} + \frac{2s\sqrt{2}}{L\du\Nmin}}\right) \\
    & \leq 2n\exp\left( -\frac{6L\dusq\Nminsq s^2}{3\Nmax\dumax+4\sqrt{2}\Nmin\du s}\right).
\end{align*}
%
%
%
Choosing $s = c_3\sqrt{\frac{\Nmax\dumax\log n}{L\dusq\Nminsq}}$ for a constant $c_3 \geq 2$ and since $\dumax\Nmax\geq \log n$ by assumption, it holds that 
\begin{equation}
    \prob\left(\norm{\Dbar}_2 > c_3\sqrt{\frac{\Nmax\dumax\log n}{L\dusq\Nminsq}}\right) \leq 2n^{1-\frac{6 c_3^2}{3+4\sqrt{2}c_3}}.
\end{equation}
Finally, we observe that up to a multiplicative constant, the bounds for $\norm{D(t)}_2$ and $\norm{\Dbar}_2$ are the same. Thus, for  constants $c_1,c_2,c_3 \geq 2$, we have
\begin{equation*}
    \norm{\Delta(t)}_2 \leq \left(c_1+\max\set{c_2,c_3}\right)\sqrt{\frac{\Nmax\dumax\log n}{L\dusq\Nminsq}}.
\end{equation*}
with probability at least $1-2n^{1-\frac{c_1^2}{2}} - \max\set{n^{2-\frac{c_2^2}{2}},2n^{1-\frac{6 c_3^2}{3+4\sqrt{2}c_3}}}$.
In simpler words, there exists a constant $\tilde{C_1} \geq 15$ such that with probability at least $1-O(n^{-10})$,
\begin{equation}
    \norm{\Delta(t)}_2 \leq \tilde{C_1} \sqrt{\frac{\Nmax\dumax\log n}{L\dusq\Nminsq}}.
\end{equation}
\begin{flushright}
$\square$
\end{flushright}

\subsection{Proof of Lemma \ref{lem3}}
The steps below follow the proof of \citep[Lemma 4]{negahban2015rank} and are provided for completeness. We will first show $1-\lammax(\Pbar) \geq \frac{\xid\dumin}{4\du b^3(t)}$ using \citep[Lemma 6]{negahban2015rank}. To this end, let us denote $Q_{ij}(t) = \frac{1}{d_i(t)}\indic{\set{i,j} \in\Eu}$ where $d_{i}(t)$ denotes the degree of the $i$th vertex in $\Gu$ and $\indic{}$ is an indicator variable. Clearly, $Q(t)$ has a unique stationary distribution as it is irreducible (since $\Gu$ is connected), defined as $\mu_i(t) = \frac{d_i(t)}{\sum_{j=1}^n d_j(t)}$. It satisfies the detailed balanced equations of reversibility
$$\mu_i(t)Q_{ij}(t) = \mu_j(t)Q_{ji}(t).$$
Using \citep[Lemma 6]{negahban2015rank}, we obtain
\begin{equation} \label{eq:lammax_comp_ineq}
    \frac{1-\lammax(\Pbar)}{1-\lammax(Q(t))} \geq \frac{\alpha}{\beta} \quad \text{with } \alpha = \min_{\set{i,j}\in\Eu} \frac{\pistari\Pbarij}{\mu_i(t)Q_{ij}(t)}, \quad \beta = \max_i \frac{\pistari}{\mu_i(t)}.
\end{equation}
Note that the transition matrix $Q$ is equal to the Laplacian $\Ld$ of the graph $\Gu$ and so $1-\lammax(Q(t)) = \xid(t)$. We now need to suitably bound $\alpha$ and $\beta$. 

Denoting $d_i(t)$ to be the degree of node $i$ in $\Gu$, clearly $\sum_{j=1}^n d_j(t) = 2\cardEu$.  Hence we have 
\begin{equation*}
    \mu_i(t)Q_{ij}(t) = \frac{1}{\sum_{k=1}^n d_k(t)} \leq \frac{1}{\cardEu} \, \text{for } \set{i,j}\in\Eu \quad \text{and} \quad \mu_i(t) = \frac{d_i(t)}{2\cardEu} \geq \frac{\dumin}{2\cardEu} \, \text{for } i \in [n]
\end{equation*}
%
%
It is also easy to see that
\begin{equation*}
    \pistari\Pbarij \geq \frac{1}{2\du nb^2(t)}\, \text{for } \set{i,j}\in\Eu, \quad \text{and} \quad\pistari \leq \frac{b(t)}{n}\, \text{for } i \in [n]
\end{equation*}
thus leading to the bounds
\begin{equation*}
    \alpha \geq \frac{\cardEu}{2\du nb^2(t)} \quad \text{and}\quad \beta \leq \frac{2b(t)\cardEu}{n\dumin}.
\end{equation*}
Plugging this in \eqref{eq:lammax_comp_ineq}  leads to
\begin{equation*}
    1-\lammax(\Pbar) \geq \frac{\xid\cardEu n\dumin}{4\du nb^3(t)\cardEu} \geq \frac{\xid\dumin}{4\du b^3(t)} > 0
\end{equation*}
which together with \eqref{eq:cond_lem3} readily leads to the stated lower bound on $1-\rho(t)$.
%
%
%
%
\begin{flushright}
$\square$
\end{flushright}
\subsection{Proof of Lemma \ref{lem:5}}

We will bound $\norm{\pistar^\top\Delta(t)}_2$, using the same ideas as in the proof of \citep[Theorem 9]{Chen_2019}. Recall that the entries of $\Delta(t)$ are given by
\begin{equation*}
    \Delta_{ij}(t) = \left\{\begin{array}{cc}
        \displaystyle\frac{1}{\du\cardNij}\sum_{t' \in \Nij} (y_{ij}(t')-y^*_{ij}(t')) & \mbox{if} \quad \set{i,j} \in \Eu,  \\
        \displaystyle - \sum_{k \neq i} \Delta_{ik}(t)& \mbox{if} \quad i=j, \\
        0 & \mbox{otherwise.}
    \end{array}\right.
\end{equation*}
Let $\Delta^l(t)$ denote the lower triangular part of $\Delta(t)$ and  $\Delta^u(t)$ its upper triangular part, both having zeros on the diagonal. Let us define the diagonal matrices 
\begin{equation}
    \Delta^{diag,l}(t) = -\diag{\left(\sum_{j<i}\Delta_{ij}(t)\right)_{i=1\dots n}} \quad \text{and} \quad \Delta^{diag,u}(t) = -\diag{\left(\sum_{j>i}\Delta_{ij}(t)\right)_{i=1\dots n}}.
\end{equation}
Note that both these matrices have independent entries. The matrix $\Delta(t)$ is now decomposed as 
\begin{equation*}
    \Delta(t) = \Delta^l(t) + \Delta^u(t) + \Delta^{diag,l}(t) + \Delta^{diag,u}(t)
\end{equation*}
and the triangle inequality implies
\begin{equation}
    \norm{\pistar^\top\Delta(t)}_2 \leq \underset{= I_l}{\underbrace{\norm{\pistar^\top\Delta^l(t)}_2}} + \underset{= I_u}{\underbrace{\norm{\pistar^\top\Delta^u(t)}_2}} + \underset{= I_{diag,l}}{\underbrace{\norm{\pistar^\top\Delta^{diag,l}(t)}_2}} + \underset{= I_{diag,u}}{\underbrace{\norm{\pistar^\top\Delta^{diag,u}(t)}_2}}.
\end{equation}
The goal is to bound each of these terms separately, using the same method for all of them. Let us describe the steps for the first term $I_l = \norm{\pistar^\top\Delta^l(t)}_2$.  For all $j \in [n]$,
\begin{align*}
\left[\pistar^\top\Delta^l(t)\right]_j =  \sum_{i=1}^n \pistari \Delta^l_{ij}(t) = \sum_{i>j} \pistari \Delta_{ij}(t).
\end{align*}
For any pair $\set{i,j} \in \Eu$, $\Delta_{ij}(t)$ is a sum of at most $L\Nmax$ independent centered random variables, thus $\left[\pistar^\top\Delta^l(t)\right]_j$ is a sum of at most $L\Nmax \djl$ independent centered random variables, where $\djl$ is defined as 
$$\djl := \lvert \set{i\in [n] \ | \ \set{i,j} \in \Eu \text{ and } i>j}\rvert.$$
By applying Hoeffding's inequality we obtain for any $s>0$,
\begin{align*}
\prob(\abs{\left[\pistar^\top\Delta^l(t)\right]_j} > s ) & \leq 2\exp\left( -\frac{2L^2\Nminsq \dusq s^2}{4L\Nmax \djl\norm{\pistar}^2_\infty}\right) \\
& \leq 2\exp\left( -\frac{L\Nminsq \dusq s^2}{2\Nmax \djl\norm{\pistar}^2_\infty}\right).
\end{align*}

Hence, $\left[\pistar^\top\Delta^l(t)\right]_j$ can be seen as a sub-Gaussian random variable with sub-Gaussian norm at most $c_1\sigma_j$ with $c_1 > 0$ a constant, and $\sigma^2_j = \frac{2\djl\Nmax\norm{\pistar}^2_\infty}{L\dusq\Nminsq}$. As $\djl < \dumax$, it holds that
\begin{equation*}
\sigma^2_j \leq 2\frac{\Nmax\dumax}{L\dusq\Nminsq}\norm{\pistar}_\infty^2
\end{equation*}
Denoting $\sigma^2 = 2c_1\frac{\Nmax\dumax}{L\dusq\Nminsq}\norm{\pistar}_\infty^2$, then we see that $\pistar^\top\Delta^l(t)$ is a random vector with independent, centered $\sigma$- sub-gaussian entries. Then,
$$I_l^2 = \norm{\pistar^\top\Delta^l(t)}^2_2 = \sum_{j=1}^n [\pistar^\top\Delta^l(t)]_j^2$$
is a quadratic form of a sub-gaussian vector, so its expectation satisfies $\expec[I_l^2] \leq c_2n\sigma^2$ for some constant $c_2 > 0$.

Using the Hanson-Wright inequality \citep[Theorem 1.1]{rudelson2013hanson}, we have for any $s>0$ that
\begin{equation*}
\prob(\lvert I_l^2 - \expec[I_l^2]\rvert > s ) \leq 2\exp\left( -c_3\min{\left(\frac{s^2}{\sigma^4n},\frac{s}{\sigma^2}\right)} \right) = 2\exp\left( -c_3\frac{s}{\sigma^2}\min{\left(\frac{s}{\sigma^2n},1\right)} \right)
\end{equation*}
for $c_3 >0$ a constant.
Choosing $s = \sqrt{\frac{C_1}{c_3}}\sigma^2\sqrt{n\log n}$ with $C_1\geq1$ a constant, note that $\frac{s}{n\sigma^2} = \sqrt{\frac{C_1}{c_3}}\sqrt{\frac{\log n}{n}} \leq 1$ if $n \geq \frac{C_1}{c_3}\log n$, and so, 
\begin{equation*}
\prob\left(\lvert I_l^2 - \expec[I_l^2]\rvert > \sqrt{\frac{C_1}{c_3}}\sigma^2\sqrt{n\log n} \right) \leq 2\exp^{-C_1\log n} = 2n^{-C_1}.
\end{equation*}
Then, it holds with probability atleast  $1-O(n^{-C_1})$ that
\begin{align*}
I_l^2 & \leq \expec[I_l^2] + \sigma^2\sqrt{n\log n} \\
& \leq c_2n\sigma^2 + \sqrt{\frac{C_1}{c_3}}\sigma^2\sqrt{n\log n} \\
& \leq n\sigma^2(c_2+\sqrt{\frac{C_1}{c_3}}\sqrt{\frac{\log n}{n}}) \\
& \leq C_2\frac{n\Nmax\dumax}{L\dusq\Nminsq}\norm{\pistar}^2_\infty \\
& \leq C_2\frac{\Nmax\dumax b^2(t)}{L\dusq\Nminsq}\norm{\pistar}^2_2
\end{align*}
as $\norm{\pistar}_2^2 \geq n \pistarmin^2 \geq \frac{n}{b^2(t)}\pistarmax^2 = \frac{n}{b^2(t)}\norm{\pistar}_{\infty}^2$.

The same analysis leads to the same bound for $I_u^2, I_{diag,l}^2,I_{diag,u}^2$. Thus for any constant $C_1 \geq 1$, there exists a constant $C_2 \geq 1$ such that it holds with probability greater than $1-O(n^{-C_1})$ that
\begin{equation}
\norm{\pistar^\top\Delta(t)}_2 \leq C_2 b(t)\sqrt{\frac{\Nmax\dumax}{L\dusq\Nminsq}}\norm{\pistar}_2.
\end{equation}
%

%
\section{Proofs of results in Section \ref{sec:analysis_linf}} \label{appsec:linf_analysis}
The proofs are largely inspired from the proofs in \citep[Section C]{Chen_2019}, which are adapted to our problem setup.
\subsection{Proof of Lemma \ref{lem:I0}}
Consider any given collection of graphs $(G_{t'})_{t' \in \Nt}$. Then following the same ideas as in \citep[Lemma 2]{Chen_2019}, we have for any given $m \in [n]$ that 
\begin{align*}
    I_0^m & = \sum_{j=1}^n \pistarj(\Pestjm-\Peststarjm) \\
    & = \sum_{j:j\neq m} \pistarj(\Pestjm-\Peststarjm) + \pistarm(\Pestmm - \Peststarmm) \\
    & = \sum_{j:j\neq m} (\pistarj+\pistarm)(\Pestjm-\Peststarjm) \\ 
    & = \sum_{j:j\neq m} (\pistarj+\pistarm)\sum_{t' \in \Njm}\sum_{l=1}^L \left(\frac{y_{jm}^{(l)}(t') - y_{jm}^*(t')}{L\du\cardNjm} \right).
\end{align*}
Applying Hoeffding's inequality, it holds for any constant $C_1 \geq 2$,

\begin{align*}
    \prob(\abs{I_0^m}>C_1\sqrt{\frac{\Nmax\dumax\log n}{L\dusq\Nminsq}}\norm{\pistar}_{\infty}) \leq 2n^{-\frac{C_1^2}{2}}.
\end{align*}
If $G_{t'} \sim \calG(n,p(t'))$ for all  $t' \in \Nt$, then using Lemma \ref{lem:prop_gnp}, it holds with probability at least $1-O(n^{-10})$ that $\frac{\Nmax\dumax}{\dusq\Nminsq} \leq \frac{4}{3n\pu\psum}$. Hence after a union bound over $[n]$, we conclude that there exists a suitably large constant $C_2 \geq 1$ such that 
\begin{align*}
    \prob\left(\forall m \in [n]: \ \abs{I_0^m}>C_2\sqrt{\frac{\log n}{Ln\pu\psum}}\norm{\pistar}_{\infty}\right) \leq O(n^{-9}).
\end{align*}

\begin{flushright}
$\square$
\end{flushright}
\subsection{Proof of Lemma \ref{lem:I1}}
For any given collection of graphs $(G_{t'})_{t' \in \Nt}$, recall that the entries of $\Peststar - \Pbar$ have already been bounded in 
the proof of Lemma \ref{lem:bounds_deterministic} as follows
\begin{equation}\label{eq:deterministic_infty_bound}
    \abs{\Peststarij-\Pbarij} \leq \left\lbrace\begin{array}{cc}
        \frac{M\delta}{T\du} & \text{ if } i\neq j, \, \set{i,j}\in\Eu \\
        \frac{M\delta}{T} & \text{ if } i=j \\
        0 &  \text{otherwise.}
    \end{array}\right.
\end{equation}
This implies for all $m \in [n]$ that
\begin{align*}
    \abs{I_1^m} = \abs{\sum_{j:j \neq m} \pistarj(\Peststarjm-\Pbarjm} \leq  \norm{\pistar}_{\infty}\left(\dumax \frac{M\delta}{T\du} + \frac{M\delta}{T} \right). 
\end{align*}
Due to Lemma \ref{lem:prop_gnp}, we know that $3n\pu = \du \geq \dumax$ holds with probability at least $1-O(n^{-10})$. This leads to the statement in Lemma \ref{lem:I1}.
\begin{flushright}
$\square$
\end{flushright}
\subsection{Proof of Lemma \ref{lem:I2}}
We will follow the steps in \citep[Lemma 3]{Chen_2019}. 
Recalling that $I_2^m = \left(\piestm-\pistarm \right)\Pestmm$, let us bound $\Pestmm$ using the decomposition
\begin{equation*}
    \Pestmm = \left(\Pestmm-\Peststarmm\right) + \Peststarmm.
\end{equation*}
For any given collection of graphs $(G_{t'})_{t' \in \Nt}$, note that
\begin{equation*}
    \Pestmm-\Peststarmm = -\frac{1}{L\du}\sum_{j:j\neq m}\sum_{t'\in\Njm}\sum_{l=1}^L \frac{1}{\cardNjm}\left(y_{mj}^{(l)}(t') - y_{mj}^*(t')\right)
\end{equation*}
is a sum of independent, centered random variables. 
Then Hoeffding's inequality implies for each $m \in [n]$,
\begin{align*}
    \prob\left(\abs{\Pestmm-\Peststarmm}>C_1\sqrt{\frac{\Nmax\dumax\log n}{L\dusq\Nminsq}}\right) \leq 2n^{-\frac{C_1^2}{2}}
\end{align*}
for any constant $C_1\geq 2$.

If $G_{t'} \sim \calG(n,p(t'))$ for all  $t' \in \Nt$, then once again combining this result with Lemma \ref{lem:prop_gnp}, we see  (after a union bound over $[n]$) that there exists a suitably large constant $C_2 \geq 1$ such that with probability at least $1-O(n^{-9})$, 
\begin{equation*}
    \abs{\Pestmm-\Peststarmm}\leq C_2\sqrt{\frac{
    \log n}{Ln\pu\psum}}, \quad \forall m \in [n].
\end{equation*}
Now we bound $\abs{\Peststarmm}$. Recalling that $\bmax = \max_{t' \in \Nt}b(t')$ and $b(t') = \max_{i,j \in [n]}\frac{w^*_{t',i}}{w^*_{t',j}}$, we have for any given collection of graphs $(G_{t'})_{t' \in \Nt}$ that
all $m \in [n]$, 
\begin{align*}
    \Peststarmm & = 1-\sum_{j:j\neq m}\frac{1}{\du\cardNjm}\sum_{t'\in\Njm}y^*_{mj}(t') \\
    & = 1 - \sum_{j:j\neq m}\sum_{t'\in\Njm}\frac{1}{\du\cardNjm}\left(\frac{w^*_{t'j}}{w^*_{t'j}+w^*_{t'm}}\right) \\
    & \leq 1 - \sum_{j:j\neq m}\sum_{t'\in\Njm}\frac{1}{\du\cardNjm}\left(\frac{1}{1+b(t')}\right) \\
    & \leq 1- \frac{\dumin}{\du}\frac{1}{2\bmax}.
\end{align*}
If $G_{t'} \sim \calG(n,p(t'))$ for all $t' \in \Nt$, then we have using Lemma \ref{lem:prop_gnp} that $\dumin \geq n\pu/2$ with probability at least $1-O(n^{-10})$, this in turn implies that $\abs{\Peststarmm} \leq 1-\frac{1}{12\bmax}$ for all $m \in [n]$. 

Hence, we conclude that there exists a suitably large constant $C_2 \geq 1$ such that with probability at least $1-O(n^{-9})$,
\begin{equation*}
    \abs{I_2^m} \leq \left(1-\frac{1}{12\bmax} +C_2\sqrt{\frac{
    \log n}{Ln\pu\psum}} \right)\norm{\piest-\pistar}_{\infty}, \quad \forall m \in [n].
\end{equation*}
\begin{flushright}
$\square$
\end{flushright}
\subsection{Proof of Lemma \ref{lem:I3}}
This proof follows the lines of the proof of \citep[Lemma 4]{Chen_2019}. Consider any given collection of graphs $(G_{t'})_{t' \in \Nt}$. Using the properties of $\norm{.}_{\pistar}$ in \eqref{eq:bound_pinorm} and Theorem \ref{thm:thm8} with the matrices $\Pest,\Pmest,\Pbar$, it holds for all $m \in [n]$ that

\begin{align}
    \norm{\piest-\pimest}_2 & \leq \frac{1}{\sqrt{\pistarmin}} \left(\frac{1}{1-\lammax(\Pbar)-\norm{\Pest-\Pbar}_{\pistar}} \right) \norm{\pimest^\top(\Pmest-\Pest)}_{\pistar} \nonumber \\
    & \leq \sqrt{\frac{\pistarmax}{\pistarmin}} \left(\frac{8\du b^3(t)}{\xid\dumin} \right)\norm{\pimest^\top(\Pmest-\Pest)}_2 \label{eq:init_bd_lem9}
\end{align}
where we used \eqref{eq:gamma}. 

In order to bound $\norm{\pimest^\top\left(\Pmest-\Pest\right)}_2$, we need to get rid of the statistical dependency between $\Pmest$ and $\pimest$. To do so, let us introduce $\Pmg$, defined for $i\neq j$ as
\begin{equation*}
    \Pmgij = \left\lbrace \begin{array}{cc}
        \Pestij & \text{if }i\neq m, j\neq m \\
        \frac{1}{\du}\frac{w_{t,j}^*}{w_{t,j}^*+w_{t,i}^*}\indic{\set{i,j} \in \Eu} & \text{if }i= m \text{ or } j= m.
    \end{array} \right.
\end{equation*}
Its diagonal entries are defined as
\begin{equation*}
    \Pmgii =  1 - \displaystyle\sum_{k:k\neq i}\Pmgik
\end{equation*}
to ensure that $\Pmg$ is a transition matrix. 
Note that the $m^{th}$ line and columns of $\Pmest$ are the expectations of those in $\Pmg$. 
Moreover,  $\est{P}^{(m),\calG}_{.m}(t) = \bar{P}_{.m}(t)$ and $\est{P}^{(m),\calG}_{m.}(t) = \bar{P}_{m.}(t)$. Now starting with the decomposition
\begin{equation*}
    \pimest^\top(\Pmest-\Pest) = \underset{J_1^m}{\underbrace{\pimest^\top\left(\Pest-\Pmg\right)}} + \underset{J_2^m}{\underbrace{\pimest^\top\left(\Pmg-\Pmest\right)}}
\end{equation*}
we will separately bound $\norm{J_1^m}_2$,  $\norm{J_2^m}_2$.
\subsubsection{Bound on $\norm{J_1^m}_2$}
Consider any given collection of graphs $(G_{t'})_{t' \in \Nt}$. Since $\est{P}^{(m),\calG}_{.m}(t) = \bar{P}_{.m}(t)$, we have the decomposition
\begin{align*}
    J^m_{1,m} = & \sum_{j=1}^n \pimestj\left(\Pestjm-\Pbarjm\right) \\
    = & \sum_{j:j\neq m} \pimestj\left(\Pestjm-\Pbarjm\right) + \pimestm\left(\Pestmm-\Pbarmm\right) \\
    = & \sum_{j:j\neq m} (\pimestj+\pimestm)\left(\Pestjm-\Pbarjm\right) \\
    = & \sum_{j:j\neq m} (\pimestj+\pimestm)\left(\Pestjm-\Peststarjm\right)+ \sum_{j:j\neq m} (\pimestj+\pimestm)\left(\Peststarjm-\Pbarjm\right). 
\end{align*}
Since $\pimest$ is independent of $(\Pestjm-\Peststarjm)_{j \neq m}$, the first term can be bounded using Hoeffding's inequality as 
\begin{equation*}
    \prob\left( \abs{\sum_{j:j\neq m} (\pimestj+\pimestm)\left(\Pestjm-\Peststarjm\right)} > C_1 \norm{\pimest}_{\infty}\sqrt{\frac{\Nmax\dumax\log n}{L\dusq\Nminsq}}\right) \leq 2n^{-\frac{C_1^2}{2}}
\end{equation*}
for any constant $C_1 \geq 2$. 
%
%
Moreover, due to \eqref{eq:deterministic_infty_bound}, we have for all $m$ that
\begin{equation*}
    \sum_{j:j\neq m} (\pimestj+\pimestm)\left(\Peststarjm-\Pbarjm\right) \leq 2\dumax\left(\frac{M\delta}{T\du} \right) \norm{\pimest}_{\infty}. 
\end{equation*}
Hence for any given $m \in [n]$, it holds with probability at least $1-O(n^{-\frac{C_1^2}{2}})$ that
\begin{equation}\label{eq:J1m_m}
    \abs{J^m_{1,m}} \leq \left(\frac{2\dumax M\delta}{T\du}   + C_1 \sqrt{\frac{\Nmax\dumax\log n}{L\dusq\Nminsq}}\right)\norm{\pimest}_{\infty}.
\end{equation}
Let us now bound the other coefficients of $J_1^m$. For any $j \neq m$,
\begin{align*}
    J^m_{1,j} = &  \sum_{i:i \neq j} \pimesti(\Pestij - \Pmgij) + \pimestj\left(\Pestjj-\Pmgjj\right) \\
    = & \pimestm\left(\Pestmj-\Pmgmj\right) + \pimestj\left(\Pestjj-\Pmgjj\right) \\
     = & (\pimestm + \pimestj)\left( \Pestmj-\Peststarmj + \Peststarmj-\Pbarmj\right) \\
   \implies \abs{J^m_{1,j}} \leq & \  2\norm{\pimest}_{\infty}\left(\abs{\Pestmj-\Peststarmj}+\frac{M\delta}{T\du}\right)
\end{align*}
where we used \eqref{eq:deterministic_infty_bound} and the fact that the $m^{th}$ rows of $\Pmg$ and $\Pbar$ are identical. Using Hoeffding's inequality and a union bound, we have
\begin{equation*}
    \prob\left(\max_{j \neq m} \abs{\Pestmj-\Peststarmj} > C_2\sqrt{\frac{\Nmax\log n}{Ld^2_{\delta}(t)\Nminsq}}\right) \leq 2n^{1-\frac{C_2^2}{2}}
\end{equation*}
for any constant $C_2 \geq 2$. Hence, for any $j \neq m$, we have with probability at least $1-O(n^{1-\frac{C_2^2}{2}})$ that
\begin{equation}\label{eq:J1m_j}
    \abs{J_{1,j}^m} \leq \left\lbrace\begin{array}{cc}
        2\norm{\pimest}_{\infty}\left(\frac{M\delta}{T\du}+C_2\sqrt{\frac{\Nmax\log n}{Ld^2_{\delta}(t)\Nminsq}} \right) & \text{ if } \set{j,m}\in\Eu \\
        0 & \text{ otherwise.}
    \end{array}\right.
\end{equation}

Finally combining \eqref{eq:J1m_m},  \eqref{eq:J1m_j} along with Lemma \ref{lem:prop_gnp}, we see after taking a union bound over $[n]$ that when $G_{t'} \sim \calG(n,p(t'))$ for all  $t' \in \Nt$, then there exists a suitably large constant $C_3 \geq 1$ such that with probability at least $1-O(n^{-9})$, 
\begin{align}\label{eq:bound_J1M}
    \norm{J_1^m}_2 
    \leq \left(\frac{4M\delta}{T}+C_3\sqrt{\frac{\log n}{Ln\pu\psum}} \right)\norm{\pimest}_{\infty} \quad \forall m \in [n].
\end{align}
\subsubsection{Bound on $\norm{J_2^m}_2$}
Consider again any given collection of graphs $(G_{t'})_{t' \in \Nt}$. Note that the off-diagonal entries of $\Pmest-\Pmg$ are non zero if and only if they belong to the $m^{th}$ row /column. Then for $i\neq j$ with $i=m$ or $j=m$, we have
\begin{equation*}
    \left(\Pmest-\Pmg\right)_{ij} = \frac{1}{\du}\left(\pu-\indic{\set{i,j}\in\Eu}\right)y^*_{ij}(t)
\end{equation*}
Moreover, the diagonal entries are given as
\begin{equation*}
   \left(\Pmest-\Pmg\right)_{ii} = \left\lbrace\begin{array}{cc}
    -\frac{1}{\du}\sum_{j:j\neq m} \left(\pu-\indic{\set{m,j}\in\Eu}\right)y^*_{mj}(t) & \text{ if } i=m \\
    -\frac{1}{\du}\left(\pu-\indic{\set{i,m}\in\Eu}\right)y^*_{im}(t) & \text{ if } i\neq m.
   \end{array}\right.
\end{equation*}
Let us first show that $\pistar^\top\left(\Pmest-\Pmg\right)=0$ by considering the cases below.
\begin{enumerate}
\item For any $j\neq m$,
\begin{align*}
    \left[\pistar^\top\left(\Pmest-\Pmg\right)\right]_j =& \frac{\pistarm}{\du}(\pu-\indic{\set{m,j} \in \Eu})y_{mj}^*(t) - \frac{\pistarj}{\du}(\pu-\indic{\set{j,m} \in \Eu})y_{jm}^*(t) \\
    =& \frac{\pu-\indic{\set{m,j} \in \Eu}}{\du}\left(\pistarm y_{mj}^*(t) - \pistarj y_{jm}^*(t)\right) \\
    = & 0.
\end{align*}
This equality comes from the definition of $\pistar$ and $y_{jm}^*(t)$.

\item For $j=m$, 
\begin{align*}
    &\left[\pistar^\top\left(\Pmest-\Pmg\right)\right]_m \\
    =& -\frac{\pistarm}{\du}\sum_{k:k\neq m}(\pu-\indic{\set{m,k} \in \Eu})y_{mk}^*(t) 
    + \frac{1}{\du}\sum_{k:k\neq m}\pistark(\pu-\indic{\set{k,m} \in \Eu})y_{km}^*(t) \\
    =& \frac{1}{\du}\sum_{k:k\neq m}\left(\pu-\indic{\set{k,m} \in \Eu}\right)\left(\pistark y_{km}^*(t) - \pistarm y_{mk}^*(t)\right) \\
    = & 0.
\end{align*}
\end{enumerate}
Hence, $\pistar^\top\left(\Pmest-\Pmg\right)=0$ and it holds that
\begin{equation*}
    J_2^m = \left(\pimest-\pistar\right)^\top\left(\Pmest-\Pmg\right).
\end{equation*}
Now for $j \neq m$,
\begin{align*}
\abs{J^m_{2,j}} =& \abs{\left(\pimestm-\pistarm\right) \left(\frac{\pu-\indic{\set{m,j} \in \Eu}}{\du}\right)y_{mj}^*(t) - \left(\pimestj-\pistarj\right) \left(\frac{\pu-\indic{\set{j,m} \in \Eu}}{\du}\right)y_{jm}^*(t)} \\
\leq & \left\lbrace\begin{array}{cc}
    \frac{4}{\du}\norm{\pimest-\pistar}_{\infty} & \text{if } \set{j,m}\in\Eu \\
    2\frac{\pu}{\du}\norm{\pimest-\pistar}_{\infty} & \text{if } \set{j,m}\notin\Eu.
\end{array}\right.
\end{align*}
Moreover, we can bound $\abs{J^m_{2,m}}$ as follows.
\begin{align*}
\abs{J^m_{2,m}} =& \left\lvert -\left(\pimestm-\pistarm\right)\sum_{j:j\neq m}\frac{1}{\du}\left(\pu-\indic{\set{m,j} \in \Eu}\right)y^*_{mj}(t)\right. \\
& \left. +\sum_{j:j \neq m} \left(\pimestj-\pistarj\right)\frac{1}{\du}\left(\pu-\indic{\set{j,m} \in \Eu}\right)y^*_{jm}(t)
\right\rvert \\
& \leq \abs{J_3^m} + \abs{J_4^m} 
\end{align*}
where
\begin{equation*}
    J_3^m = \sum_{j:j\neq m}\left(\pimestm-\pistarm\right)\frac{1}{\du}\left(\pu-\indic{\set{m,j} \in \Eu}\right)y^*_{mj}(t), 
\end{equation*}
and
\begin{equation*}
J_4^m = \sum_{j:j \neq m} \left(\pimestj-\pistarj\right)\frac{1}{\du}\left(\pu-\indic{\set{j,m} \in \Eu}\right)y^*_{jm}(t).
\end{equation*}
When $G_{t'} \sim \calG(n,p(t'))$ for all  $t' \in \Nt$, then denoting $\xi_j^{(m)} = \left(\pimestm-\pistarm\right)\frac{y^*_{mj}(t)}{\du}$, a bound on $J_3^m$ can be given by Bernstein's inequality, introducing the random variables $Z_j^{(m)} = \xi_j^{(m)} (\pu-\indic{\set{j,m}\in\Eu})$. Since for all $j \neq m$,  $$\abs{Z_j^{(m)}} \leq \frac{1}{\du}\norm{\pimest-\pistar}_{\infty}, \quad \expec[Z_j^{(m)}]^2 \leq \frac{\norm{\pimest-\pistar}_{\infty}^2}{\dusq} \pu$$ hence 
%
%
there exists a constant $c \geq 2$ such that it holds with probability at least $1-2n^{-3c/2}$ that
\begin{equation}\label{eq:J3m_bound}
    \abs{J_3^m} \leq c\frac{\sqrt{n\pu\log n}+\log n}{\du}\norm{\pimest-\pistar}_{\infty}.
\end{equation}
%

The same bound holds for $\abs{J_4^m}$. Using Lemma \ref{lem:prop_gnp}, we then see that there exists a suitably large constant $C_4 \geq 1$ such that with probability at least $1-O(n^{-10})$,
\begin{equation*}
    \norm{J_2^m}_2 \leq \left(C_4\frac{\sqrt{n\pu\log n}+\log n}{\du} + \frac{4\sqrt{\du}}{\du} + \frac{2\pu\sqrt{n}}{\du}\right)\norm{\pimest-\pistar}_{\infty}.
\end{equation*}
Taking the union bound over $[n]$ and performing a minor simplification of the previous bound, we finally observe that with with probability at least $1-O(n^{-9})$, 
%
\begin{equation}\label{eq:bound_J2M}
    \norm{J_2^m}_2 \leq C_5\sqrt{\frac{\log n}{n\pu}}\norm{\pimest-\pistar}_{\infty} \quad \forall m \in [n].
\end{equation}
for some suitably large constant $C_5 \geq 1$.

%
\subsubsection{Putting it together}
Using Lemma \ref{lem:prop_gnp}, we know that with probability at least $1-O(n^{-10})$,  
\begin{equation} \label{eq:tmp1_bd}
  \sqrt{\frac{\pistarmax}{\pistarmin}} \left(\frac{8\du b^3(t)}{\xid\dumin} \right)   \leq 96b^{\frac{7}{2}}(t).
\end{equation}
Finally, combining \eqref{eq:init_bd_lem9},  \eqref{eq:bound_J1M},  \eqref{eq:bound_J2M} and \eqref{eq:tmp1_bd}, we have that with probability at least $1-O(n^{-9})$, it holds for all $m \in [n]$,  
\begin{align*}
    \norm{\pimest-\piest}_2 \leq & 96b^{\frac{5}{2}}(t)(\norm{J_1^m}_2 + \norm{J_2^m}_2) \\
    \leq & 96b^{\frac{5}{2}}(t)\left(\frac{4M\delta}{T} + C_3\sqrt{\frac{\log n}{Ln\pu\psum}} \right)\norm{\pimest}_{\infty} \\
    & + 96C_5b^{\frac{5}{2}}(t)\sqrt{\frac{\log n}{n\pu}}\norm{\pimest-\pistar}_{\infty} \\
    \leq & 96b^{\frac{5}{2}}(t)\left(\frac{4M\delta}{T} + C_3\sqrt{\frac{\log n}{Ln\pu\psum}} \right)\norm{\pistar}_{\infty} \\
    & + 96b^{\frac{5}{2}}(t)\left(\frac{4M\delta}{T} + C_3\sqrt{\frac{\log n}{Ln\pu\psum}}+C_5\sqrt{\frac{\log n}{n\pu}}\right)\norm{\pimest-\pistar}_{\infty} \\
    \leq & 96b^{\frac{5}{2}}(t)\left(\frac{4M\delta}{T} + C_3\sqrt{\frac{\log n}{Ln\pu\psum}} \right)\norm{\pistar}_{\infty} \\
    & + 96b^{\frac{5}{2}}(t)\left(\frac{4M\delta}{T} + C_6\sqrt{\frac{\log n}{n\pu}}\right)\norm{\pimest-\pistar}_{\infty}
\end{align*}
for some suitably large constant $C_6 \geq 1$. Condition \eqref{eq:cond_lem9} further implies that for all $m \in [n]$, 
\begin{equation*}
    \norm{\pimest-\piest}_2 \leq 96b^{\frac{5}{2}}(t)\left(\frac{4M\delta}{T} + C_3\sqrt{\frac{\log n}{Ln\pu\psum}} \right)\norm{\pistar}_{\infty} + \frac{1}{2}\norm{\pimest-\pistar}_{\infty}.
\end{equation*}
Finally, the triangular inequality
\begin{equation*}
    \norm{\pimest-\pistar}_{\infty}\leq \norm{\pimest-\piest}_2 + \norm{\piest-\pistar}_{\infty}
\end{equation*}
implies that 
\begin{align*}
    &\norm{\pimest-\piest}_2 \\
    \leq & 96b^{\frac{5}{2}}(t)\left(\frac{4M\delta}{T} + C_3\sqrt{\frac{\log n}{Ln\pu\psum}} \right)\norm{\pistar}_{\infty} 
    + \frac{1}{2}\left(\norm{\pimest-\piest}_2 +\norm{\piest-\pistar}_{\infty}\right) \\
    \leq & 192b^{\frac{5}{2}}(t)\left(\frac{4M\delta}{T} + C_3\sqrt{\frac{\log n}{Ln\pu\psum}} \right)\norm{\pistar}_{\infty}+\norm{\piest-\pistar}_{\infty}.
\end{align*}

%
\subsection{Proof of Lemma \ref{lem:I4}}
Since we follow the same steps as in \citep{chen2020partial}, let us introduce some new notations in our particular setup. For all $t \in \calT$, all $i \neq j$ and all $l \in [L]$, let us denote
\begin{equation}\label{eq:y_tilde}
    \Tilde{y}_{ij}^{(l)}(t) \sim \calB\left(y^*_{ij}(t)\right) \quad \text{and} \quad \Tilde{y}_{ij}(t) = \frac{1}{L}\sum_{l=1}^L\Tilde{y}_{ij}^{(l)}(t).
\end{equation}
so for all $t'\in \Nt$, we have  
$$y_{ij}(t') = \Tilde{y}_{ij}(t')\indic{\set{i,j}\in\calE_{t'}} \leq \Tilde{y}_{ij}(t')\indic{\set{i,j}\in\Eu}.$$ %
Then $\abs{I_4^m}$ can be bounded as
\begin{align*}
    \abs{I_4^m} & = \abs{\sum_{j:j \neq m}\left(\pimestj-\pistarj\right)\frac{1}{L\du\cardNjm}\sum_{t'\in\Njm}\sum_{l=1}^L y_{jm}^{(l)}(t')} \\
    & \leq \sum_{j:j\neq m} \abs{\pimestj-\pistarj}\left[\frac{1}{L\du\cardNjm}\sum_{t'\in\Njm}\sum_{l=1}^L \Tilde{y}_{jm}^{(l)}(t')\right]\indic{\set{j,m}\in\Eu} \\
    & \leq \frac{\Nmax}{\Nmin}\sum_{j:j\neq m}\frac{1}{\du}\abs{\pimestj-\pistarj} \indic{\set{j,m}\in\Eu}.
\end{align*}    
When $G_{t'} \sim \calG(n,p(t'))$ for all  $t' \in \Nt$, then we know from Lemma \ref{lem:prop_gnp} that
\begin{equation}\label{eq:bound_nmaxI4}
    \frac{\Nmax}{\Nmin} \leq 4
\end{equation}
holds with probability at least $1-O(n^{-10})$.
Hence, to bound $\abs{I_4^m}$, it suffices to bound $$I_5^m := \sum_{j:j\neq m}\frac{1}{\du}\abs{\pimestj-\pistarj} \indic{\set{j,m}\in\Eu}.$$ To this end, let us first denote $\Gum$ to be the union graph $\Gu$ where the item $m$ has been removed, and $\Tilde{y}(t)$ to be the variables computed in \eqref{eq:y_tilde}. 
Then we have by triangle inequality that 
\begin{equation} \label{eq:I5m_bd}
    I_5^m \leq \expec[I_5^m \vert \Gum,\Tilde{y}] + \abs{I_5^m - \expec[I_5^m \vert \Gum,\Tilde{y}]}.
\end{equation}
The expectation can be bound using Cauchy-Schwarz inequality.
\begin{align*}
    \expec[I_5^m \vert \Gum,\Tilde{y}] = & \sum_{j\neq m} \frac{1}{\du}\abs{\pimestj-\pistarj}\pu \\
    \leq &\frac{\sqrt{n}\pu}{\du}\norm{\pimest-\pistar}_2 \\
    \leq& \frac{1}{3\sqrt{n}}\norm{\pimest-\pistar}_2 \\
    \leq &\frac{1}{3\sqrt{n}}\norm{\pimest-\piest}_2 + \frac{1}{3\sqrt{n}}\norm{\piest-\pistar}_2.
\end{align*}
A bound on the first term is given by Lemma \ref{lem:I3} and by Theorem \ref{thm:l2_gnp} for the second term. Hence there exist constants $C_3, \Tilde{C}_2 \geq 1$ so that with probability at least $1-O(n^{-9})$, we have for all $m \in [n]$ that
\begin{align}
    & \expec[I_5^m \vert \Gum,\Tilde{y}] \nonumber\\ 
    \leq & 64b^{\frac{5}{2}}(t)\left(\frac{4M\delta}{T\sqrt{n}} + C_3\sqrt{\frac{\log n}{Ln^2\pu\psum}} \right)\norm{\pistar}_{\infty}+\frac{1}{3\sqrt{n}}\norm{\piest-\pistar}_{\infty} \nonumber \\
    & + \frac{1}{\sqrt{n}}\left( 521\frac{M\delta nb^{\frac{7}{2}}(t)}{T} + 64\Tilde{C}_2 b^{\frac{9}{2}}(t)\sqrt{\frac{1}{Ln\pu\psum}}\right)\norm{\pistar}_2 \nonumber \\
    \leq & 64b^{\frac{5}{2}}(t)\left(\frac{4M\delta}{T\sqrt{n}} + C_3\sqrt{\frac{\log n}{Ln^2\pu\psum}} \right)\norm{\pistar}_{\infty}+\frac{1}{3\sqrt{n}}\norm{\piest-\pistar}_{\infty} \nonumber \\
    & + \left( 512\frac{M\delta nb^{\frac{7}{2}}(t)}{T} + 64 \Tilde{C}_2 b^{\frac{9}{2}}(t)\sqrt{\frac{1}{Ln\pu\psum}}\right)\norm{\pistar}_{\infty} \nonumber \\
   &\leq  \left(C_3'\frac{Mn\delta b^{\frac{7}{2}}(t)}{T} + C_4' \frac{b^{\frac{9}{2}}(t)}{\sqrt{Ln\pu\psum}}\right)\norm{\pistar}_{\infty} + \frac{1}{3\sqrt{n}}\norm{\piest-\pistar}_{\infty} \label{eq:bound_expI5}
\end{align}
for some constants  $C_3',C_4' \geq 1$.

To bound the deviation term in \eqref{eq:I5m_bd}, we will use Bernstein's inequality. Then, there exists a constant $c\geq 1$ such that for every $m \in [n]$, 
\begin{equation*}
    \prob\left( \abs{I_5^m - \expec[I_5^m \vert \Gum,\Tilde{y}]} \geq c\frac{\sqrt{n\pu\log n}+\log n}{\du}\norm{\pimest-\pistar}_{\infty}\right) \leq 2n^{-10}.
\end{equation*}
Hence using a union bound over $[n]$, we see that there exists a suitably large constant $c' \geq 1$ so that with probability at least $1-O(n^{-9})$,
\begin{equation}
    \abs{I_5^m - \expec[I_5^m \vert \Gum,\Tilde{y}]} \leq c'\sqrt{\frac{\log n}{n\pu}}\norm{\pimest-\pistar}_{\infty} \quad \forall m \in [n]. \label{eq:I5m_dev_term_bd}
\end{equation}
It remains to bound $\norm{\pimest-\pistar}_{\infty}$. 
Using Lemma \ref{lem:I3} along with the triangular inequality
$$\norm{\pimest-\pistar}_{\infty} \leq \norm{\pimest - \piest}_2 + \norm{\piest-\pistar}_{\infty},$$
we have that with probability at least $1-O(n^{-9})$,  
\begin{equation*}
    \norm{\pimest-\pistar}_{\infty} \leq 192 b^{\frac{5}{2}}(t)\left(\frac{4M\delta}{T}+C_3\sqrt{\frac{\log n}{Ln\pu\psum}} \right)\norm{\pistar}_{\infty} + 2\norm{\piest-\pistar}_{\infty}.
\end{equation*}
Upon plugging this in \eqref{eq:I5m_dev_term_bd}, the latter simplifies to
\begin{align}
    &\abs{I_5^m - \expec[I_5^m \vert \Gum,\Tilde{y}]} \nonumber \\ 
    &\leq C_3''\sqrt{\frac{\log n}{n\pu}}\norm{\piest-\pistar}_{\infty} + C_4'' b^{\frac{5}{2}}(t) \sqrt{\frac{\log n}{n\pu}} \left(\frac{4M\delta}{T}+\sqrt{\frac{\log n}{Ln\pu\psum}} \right)\norm{\pistar}_{\infty} \label{eq:bound_varI5}
\end{align}
for some constants $C_3'', C_4'' \geq 1$.


Finally, combining \eqref{eq:bound_nmaxI4}, \eqref{eq:I5m_bd},  \eqref{eq:bound_expI5} and \eqref{eq:bound_varI5}, we conclude that with probability at least $1-O(n^{-9})$, it holds for all $m \in [n]$ that
\begin{align*}
    \abs{I_4^m} \leq & \left(C_7\frac{Mn\delta b^{\frac{7}{2}}(t)}{T} + C_8 \frac{b^{\frac{5}{2}}(t)\max\set{b^2(t),\frac{\log n}{\sqrt{n\pu}}}}{\sqrt{Ln\pu\psum}}\right)\norm{\pistar}_{\infty} + C_9\sqrt{\frac{\log n}{n\pu}}\norm{\piest-\pistar}_{\infty}
\end{align*}
for some constants $C_7,C_8, C_9 \geq 1$.

\section{Proof of Lemma \ref{lem:ptil}} 
\label{appsec:proof_diff_constr_graph}
%
%
\rev{Since $\cardNij$ are i.i.d random variables for each $i < j$, hence $\Gutil$ is distributed as an Erdös-Renyi graph with probability $\putil$ defined as
\[ \putil = \prob\left(\abs{\calN_{12,\delta}(t)} \in \left[ \max\set{1,\frac{\psum}{2}},\max\set{2\psum,6\log\frac{4}{\delta\pmin}}\right]\right).\]
Note that $\delta\pmin \leq \cardN\pmin \leq \psum \leq \cardN\pmax \leq 4\delta\pmax$, thus it follows that 
\begin{itemize}
    \item if $\delta\pmin \geq 3$, then $\Eutil := \set{\set{i,j} \, : \, \cardNij \in \left[ \frac{\psum}{2},2\psum\right]};$
    \item if $\delta\pmax \leq \frac{1}{8}$, then $\Eutil := \set{\set{i,j} \, : \, \cardNij \in \left[ 1,6\log\frac{4}{\delta\pmin}\right]}.$
\end{itemize}
Let us now bound $\putil$ in each case using Chernoff bounds (see Theorem \ref{thm:chernoff}).}
\begin{itemize}

    \item \rev{If $\delta\pmin \gtrsim 1$, then Chernoff's bound implies that $ \putil \geq 1-2e^{-\frac{\psum}{12}} \geq 1-2e^{-1/4}$.
    \item If $\delta\pmax < \frac{1}{8}$, then
    \begin{equation}\label{eq:putil}
        \putil = \prob\left(\abs{\calN_{12,\delta}(t)} \in \left[1,6\log\frac{4}{\delta\pmin}\right]\right) = \pu - \prob\left(\abs{\calN_{12,\delta}(t)} > 6\log\frac{4}{\delta\pmin}\right).
    \end{equation}
    }

    \rev{Using Chernoff's bound, it holds that
    \begin{equation}\label{eq:putil_chernoff}
    \prob\left(\abs{\calN_{12,\delta}(t)} > 6\log\frac{4}{\delta\pmin}\right) \leq \frac{\delta\pmin}{4}.
    \end{equation}
    Moreover, $\delta\pmax < \frac{1}{8}$ implies that $\psum\leq \frac{1}{2}$, and so, one can  bound $\pu$ using Proposition \ref{prop:useful_ineg} as follows.
    \begin{equation}\label{eq:bound_pu}
        \frac{\delta\pmin}{2}\leq 1-e^{-\delta\pmin}\leq \pu\leq 1-e^{-8\delta\pmax}\leq 8\delta\pmax.
    \end{equation}
    Finally, combining \eqref{eq:putil},\eqref{eq:putil_chernoff} and \eqref{eq:bound_pu}, it holds that $\putil \geq \frac{\delta\pmin}{4}$. The upper bound on $\putil$ comes from \eqref{eq:bound_pu} and the fact that $\putil \leq \pu$. }
\end{itemize}

\end{document}